\documentclass{article}

\usepackage[letterpaper,top=2cm,bottom=2cm,left=3cm,right=3cm,marginparwidth=1.75cm]{geometry}
\usepackage{amsmath}
\usepackage{amsthm}
\usepackage{amsfonts}
\usepackage{mathtools}
\usepackage{graphicx}
\usepackage{amssymb}
\usepackage[colorlinks=true, allcolors=blue]{hyperref}
\usepackage[T1]{fontenc}
\usepackage{lmodern}
\usepackage{algorithm}
\usepackage[noend]{algpseudocode}
\usepackage{xcolor}
\usepackage{adjustbox}
\usepackage{todonotes}
\usepackage{tikz}
\usepackage{cleveref}
\usepackage{enumerate}
\usepackage{booktabs}
\usepackage{longtable}
\usepackage{enumitem}
\usepackage{xspace}
\usepackage{pgfplots}
\usepgfplotslibrary{fillbetween}
\usepackage{aliascnt}
\usepackage{subcaption}
\usetikzlibrary{patterns, arrows.meta, positioning, decorations.pathreplacing, shapes.geometric,patterns.meta,shapes.misc}

\newcommand{\R}{\mathbb{R}}
\newcommand{\conv}{\mathrm{conv}}

\newcommand{\polygon}{P}

\newcommand{\staircasepolygon}{ortho-polygon\xspace}

\newcommand{\orth}{\polygon_{\perp}}

\DeclareMathOperator*{\argmin}{arg\,min}

\newtheorem{theorem}{Theorem}[section]
\newaliascnt{lemma}{theorem}
\newtheorem{lemma}[lemma]{Lemma}
\aliascntresetthe{lemma}
\newaliascnt{corollary}{theorem}

\aliascntresetthe{corollary}

\newaliascnt{proposition}{theorem}
\newtheorem{proposition}[proposition]{Proposition}
\aliascntresetthe{proposition}

\newaliascnt{remark}{theorem}
\newtheorem{remark}[remark]{Remark}
\aliascntresetthe{remark}

\newtheorem{definition}[theorem]{Definition}
\newaliascnt{assumption}{theorem}
\newtheorem{assumption}[assumption]{Assumption}
\aliascntresetthe{assumption}

\newaliascnt{observation}{theorem}

\aliascntresetthe{observation}

\crefname{assumption}{assumption}{assumptions}
\Crefname{assumption}{Assumption}{Assumptions}
\crefname{remark}{remark}{remarks}
\Crefname{remark}{Remark}{Remarks}
\crefname{corollary}{corollary}{corollaries}
\Crefname{corollary}{Corollary}{Corollaries}
\crefname{observation}{observation}{observations}
\Crefname{observation}{Observation}{Observations}
\crefname{proposition}{proposition}{propositions}
\Crefname{proposition}{Proposition}{Propositions}
\crefname{lemma}{lemma}{lemmas}
\Crefname{lemma}{Lemma}{Lemmas}

\definecolor{color0}{RGB}{205, 145, 50}
\definecolor{color1}{RGB}{25, 42, 86}
\definecolor{color2}{RGB}{156, 56, 72}
\definecolor{color3}{RGB}{0, 119, 139}
\definecolor{color4}{RGB}{140, 80, 45}   % warm sienna
\definecolor{color5}{RGB}{80, 55, 110}    % deep plum
\definecolor{color6}{RGB}{45, 95, 72}     % forest green
\definecolor{color7}{RGB}{185, 95, 40}    % burnt orange
\definecolor{color8}{RGB}{65, 75, 95}     % slate blue-gray
\definecolor{color9}{RGB}{160, 95, 115}   % dusty mauve
\definecolor{color10}{RGB}{58, 110, 82}   % muted emerald
\definecolor{color11}{RGB}{120, 48, 48}   % dark brick red

\newcommand{\stairshort}{OL}
\newcommand{\locatellishort}{L}
\newcommand{\indicatorshort}{IL}
\newcommand{\baseshort}{BASE}

\newcommand{\nwvertex}{north-west vertex}
\newcommand{\nevertex}{north-east vertex}
\newcommand{\swvertex}{south-west vertex}
\newcommand{\severtex}{south-east vertex}
\newcommand{\nwvertices}{north-west vertices}
\newcommand{\nevertices}{north-east vertices}
\newcommand{\swvertices}{south-west vertices}
\newcommand{\severtices}{south-east vertices}

\usepackage[backend=biber,
citestyle=alphabetic, 
bibstyle=alphabetic, 
hyperref=true, 
maxbibnames=20, 
maxcitenames=6, 
uniquelist=false,
url=false, 
doi=true, 
firstinits=true]{biblatex}
\DeclareNameAlias{author}{last-first}

\bibliography{references}
\title{On convex envelopes of bilinear terms over polygonal domains}
\author{
Adrian Göß\textsuperscript{1}\quad
Jan Krause\textsuperscript{1}\quad
Martina Kuchlbauer\textsuperscript{1, 2}\quad
Tobias Kuen\textsuperscript{2}\\[0.8em]
\small\textsuperscript{1}University of Technology Nuremberg\\
\small\texttt{\{adrian.goess, jan.krause, martina.kuchlbauer\}@utn.de}\\[0.5em]
\small\textsuperscript{2}Fraunhofer Institute for Integrated Circuits\\
\small\texttt{\{martina.kuchlbauer, tobias.kuen\}@iis.fraunhofer.de}
}
\date{}

\begin{document}
	
\maketitle

\begin{abstract}

\noindent We develop an algorithm for computing the convex envelope of the bilinear
function $f(x,y)=xy$ over possibly nonconvex polygonal domains. Building on
existing methods for convex polytopes, our approach uses the
Karush--Kuhn--Tucker conditions and an analysis of the contact points
between affine underestimators and the graph of $f$. Nonconvex domains
admit additional contact configurations that prevent several
simplifications available in the convex case. We resolve these
configurations explicitly and show that the envelope admits a finite
polyhedral subdivision with affine, quadratic, or rational expressions
on its cells. For orthogonal polygons, whose edges are horizontal or
vertical, we establish substantial reductions in the generator sets
needed to compute the envelope. We also present methods for constructing
such domains as relaxations of two-dimensional projections of nonlinear
optimization problems. An illustrative example quantifies the loss in
relaxation strength caused by convexifying the domain before computing
the envelope. Computational experiments on MINLPLib instances show that
cuts derived from orthogonal polygons strengthen root relaxations and
reduce branch-and-bound nodes in aggregate.

\par\medskip
\noindent\begin{tabular}{@{}l@{\ }l@{}}
\textbf{Keywords:} & Convex envelopes \textperiodcentered\ Bilinear terms
\textperiodcentered\ Nonconvex polygons \textperiodcentered\ Global optimization\\
& KKT conditions \textperiodcentered\ Polyhedral subdivisions
\end{tabular}
\end{abstract}

\section{Introduction}
\label{sec:intro}

Spatial branch-and-bound is a widely used framework for the deterministic
global optimization of nonconvex mixed-integer nonlinear programs; see, e.g., \cite{Burer2012}. Its computational efficiency depends strongly on the strength of the
convex relaxations employed. A nonconvex constraint $g(\mathbf{x})\leq0$ can be
relaxed by replacing $g$ with a convex underestimator $\tilde{g}$,
yielding the convex constraint $\tilde{g}(\mathbf{x})\leq0$. The tightest such underestimator over a given domain is the convex envelope. Although computing convex envelopes is difficult in general, explicit characterizations are
available for particular classes of functions and domains. 
Of particular interest are bilinear terms $xy$, which arise in numerous
applications, including those of mixed-integer bilinear programming
\cite{Fischetti2020}. A prominent example is the pooling problem,
which appears in petroleum refining, wastewater treatment, and other
process industries \cite{misener2009advances}. Bilinear terms also
serve as building blocks in reformulations of more general nonlinear
expressions.

The convex envelope of $xy$ over a rectangle has been known for decades
\cite{McCormick1976}, and the resulting McCormick relaxation is a
standard component of modern global optimization solvers. 
Convex envelopes over more general domains have been studied in
subsequent work; an overview of relevant contributions is provided in
\autocite[Section~1]{locatelli2016_polyhedral}. Our work builds directly
on \autocite{locatelli2016_polyhedral,locatelli2018convex}, which consider a
broader class of bivariate functions over general polytopes, with
particular attention to the bilinear term $xy$.
In \autocite{locatelli2016_polyhedral}, a constructive enumeration
procedure establishes that the convex envelope is characterized by a
unique polyhedral subdivision of the domain, with an affine, quadratic,
or rational expression on each cell. The number of subproblems, however,
grows exponentially with the number of positive-slope edges.
In \autocite{locatelli2018convex}, Locatelli exploits the associated KKT
system to obtain a polynomial-time procedure.

These approaches assume polytopes as domains. We investigate the extension to
possibly nonconvex polygons. Many of the underlying ideas in
\autocite{locatelli2016_polyhedral,locatelli2018convex} carry over without
requiring convexity of the domain. However, a detailed analysis of the contact points between affine underestimators and the graph of $xy$ reveals substantial differences between convex and nonconvex domains: several simplifications of the KKT-based approach no longer apply. By explicitly
resolving the additional contact point configurations, we obtain an extended
KKT-based algorithm for computing the convex envelope of $xy$ over
polygonal domains. We also extend the polyhedral subdivision result and
its associated functional forms to this setting. An implementation of the
algorithm will be made publicly available.
Obtaining useful polygonal domains for individual bilinear terms in
general high-dimensional nonlinear programs is itself a challenge.
We therefore consider the special case of orthogonal polygons, whose
edges are horizontal or vertical. For this class, we substantially
simplify the computation of the convex envelope and present practical
methods for constructing such polygons as relaxations of the feasible
domain of $(x,y)$ within a larger optimization problem.

Computational experiments on MINLPLib instances assess the impact of
cuts derived from orthogonal polygons. Compared with the default solver setting and with cuts based on convexified domains, these cuts strengthen
root relaxations and reduce the number of branch-and-bound nodes in
aggregate, while solving times remain comparable. The results illustrate
the practical potential of exploiting nonconvex domain geometry.

The paper is organized as follows.
\Cref{sec:convenv} introduces the problem and preliminaries.
\Cref{sec:comp-conv} develops the KKT-based algorithm and establishes the
polyhedral subdivision result.
\Cref{sec:orthogonal_polygons} presents reductions for orthogonal polygons
and their use in nonlinear optimization.
\Cref{sec:computational_study} reports the computational experiments,
followed by concluding remarks.

\section{Problem description and preliminaries}
\label{sec:convenv}

In this section, we introduce the polygonal domains and notation, define the convex
envelope, and recall its reformulation as a continuously differentiable
convex program.

\begin{definition}[Simple polygon]\label{def:polygon}
Let $(v_1,\dots,v_n)$, $n\geq3$, be pairwise distinct points in
$\mathbb{R}^2$ defining a closed polygonal chain with edges
$[v_i,v_{i+1}]$ for $i=1,\dots,n$, where $v_{n+1}\coloneqq v_1$.
Suppose that nonconsecutive edges are disjoint, while consecutive
edges are noncollinear and intersect exactly at their common endpoint.
A \textit{simple polygon} $\polygon$ is the union of this chain and
its bounded interior. We denote its sets of vertices and edges by
$V(\polygon)$ and $E(\polygon)$, respectively.
\end{definition}
\noindent Note that the enclosed region is well-defined by the Jordan curve theorem.
Throughout this paper, the term \textit{polygon} refers to a simple polygon. Now, the central object of interest can be defined:

\begin{definition}[Convex envelope]\label{def2:convex_envelope}
Let $\polygon\subseteq\mathbb{R}^2$ be a polygon and let
$f:\polygon\to\mathbb{R}$ be lower semicontinuous.
The \emph{convex envelope} of $f$ over $\polygon$ is for all $\mathbf{x}\in\operatorname{conv}(\polygon)$ defined by
\[
\conv_{f,\polygon}(\mathbf{x})
\coloneqq
\sup\left\{
g(\mathbf{x}) :
g:\operatorname{conv}(\polygon)\to\mathbb{R}
\text{ is convex and }
g(\mathbf{y})\le f(\mathbf{y})
\text{ for all }\mathbf{y}\in\polygon
\right\}.
\]
\end{definition}
\noindent It is directly implied by this definition that $\conv_{f,\polygon}$
is the tightest convex underestimator of $f$ over $\polygon$. The approach followed in this paper relies on an alternative representation of the convex envelope, stated in the following remark.

\begin{remark}\label{def:convex_envelope}
Under the assumptions of \Cref{def2:convex_envelope}, the convex
envelope admits the following representation:
\begin{equation}\label{eq:def_conv}
	\conv_{f,\polygon}(\mathbf{x}) = \sup_{\boldsymbol{\alpha}\in\R^2, c\in\R} \{\boldsymbol{\alpha}^\top \mathbf{x} + c: \boldsymbol{\alpha}^\top \mathbf{y} + c\leq f(\mathbf{y}) \text{ for all } \mathbf{y}\in \polygon\}.
\end{equation}
This representation follows from
\autocite[Theorem~2.2]{Falk1969}, which identifies the convex envelope
with the biconjugate of $f$, itself the pointwise supremum of all
affine minorants by the definition of conjugation.
\end{remark}

\noindent As a preliminary step towards developing our algorithm, we recall the reformulation of \eqref{eq:def_conv} as a continuously
differentiable convex program, presented in
\autocite[Section~2]{locatelli2016_polyhedral} and
\autocite[Section~2]{locatelli2018convex} which extends to nonconvex polygons.
We use the simpler version of the assumptions introduced
in the aforementioned works:

\begin{assumption}\label{assumption}
	Let $f\in \mathcal{C}^2$ be a bivariate function and $\polygon \subseteq \mathbb{R}^2$ be a polygon such that 
	\begin{enumerate}[label=(\roman*)]
		\item the Hessian of $f$ is indefinite in the interior of $\polygon$, 
		\item the restriction of $f$ along each edge of $\polygon$ is either concave or strictly convex.
	\end{enumerate}
\end{assumption}

\noindent The affine underestimator condition in \eqref{eq:def_conv} is equivalent to
\begin{equation}\label{eq:first_reformulation}
\min_{\mathbf{y}\in\polygon}
\bigl(f(\mathbf{y})-\boldsymbol{\alpha}^{\top}\mathbf{y}\bigr)\geq c.
\end{equation}
Under \Cref{assumption}, this minimum is attained on the boundary;
on concave edges, it suffices to consider the endpoints.
Thus, only vertices and the edges $E'(\polygon)$ along which $f$
is strictly convex are needed. This argument does not require
convexity of $\polygon$. For $e\in E'(\polygon)$ with supporting line $y=m_e x+q_e$,
let $x_e^1<x_e^2$ be its endpoint $x$-coordinates and set
$g_e(x)\coloneqq f(x,m_e x+q_e)$ and $\boldsymbol{\alpha}\coloneqq(a,b)$. (Vertical edges are treated analogously by interchanging the coordinates.)
The edge minimum becomes
\[
\min_{x\in[x_e^1,x_e^2]}\bigl(g_e(x)-(a+m_e b)x\bigr)-bq_e.
\]

\noindent Strict convexity gives a unique minimizer $x_e(a,b)$.
Define the domains
\begin{align*}
D_e^-&\coloneqq\{(a,b):a+m_e b\leq g_e'(x_e^1)\},\\
D_e^+&\coloneqq\{(a,b):a+m_e b\geq g_e'(x_e^2)\},\\
D_e&\coloneqq\{(a,b):g_e'(x_e^1)\leq a+m_e b\leq g_e'(x_e^2)\}.
\end{align*}
\noindent Then
\begin{equation}\label{eq:def_x_j}
x_e(a,b) =
	\begin{cases}
		x_e^1, & \text{for }  (a, b) \in D_e^-, \\
		x_e^2, & \text{for }  (a, b) \in D_e^+,\\
		s_e(a,b), & \text{for } (a, b) \in D_e,
	\end{cases}
\end{equation}
where $s_e(a,b)$ is the unique solution in $[x_e^1,x_e^2]$
of $g_e'(s_e(a,b))=a+m_e b$ for $(a,b)\in D_e$.

\noindent For edges $e \in E'(\polygon)$ and vertices $v = (v_x, v_y) \in V(\polygon)$, respectively, set
\[
\mathbf{x}_e(\boldsymbol{\alpha})\coloneqq
\bigl(x_e(a,b),y_e(a,b)\bigr),\qquad
y_e(a,b)\coloneqq m_e x_e(a,b)+q_e,
\]
\[
\mathbf{x}_v(\boldsymbol{\alpha})\coloneqq \bigl(x_v(a,b),y_v(a,b)\bigr),\qquad
x_v(a,b)\coloneqq v_x,\qquad y_v(a,b)\coloneqq v_y.
\]
The finite collection of \textit{generators} and their associated
minimum values are
\[
G(\polygon)\coloneqq V(\polygon)\cup E'(\polygon),\qquad
\eta_g(\boldsymbol{\alpha})\coloneqq
f(\mathbf{x}_g(\boldsymbol{\alpha}))-
\boldsymbol{\alpha}^{\top}\mathbf{x}_g(\boldsymbol{\alpha}),
\quad g\in G(\polygon).
\]
We identify a vertex generator with its singleton set whenever set
operations are used. These definitions yield the polygonal counterpart
of \autocite[Lemma~2.1]{locatelli2018convex}:

\begin{lemma}
	Under Assumption \ref{assumption}, the convex envelope of a bivariate function $f$ over a polygon $\polygon$ at some point $\mathbf{x} \in \polygon$ can be computed by solving the following problem
\begin{equation}\label{eq:convex_program}
	\begin{aligned}
			\conv_{f,\polygon}(\mathbf{x}) = \max_{\boldsymbol{\alpha}, c} \quad & \boldsymbol{\alpha}^\top \mathbf{x} + c  \\
			\text{s.t.} \quad &  \eta_g(\boldsymbol{\alpha}) \ge c, \quad g\in G(\polygon).
	\end{aligned}
\end{equation}

\end{lemma}

\noindent The functions $\eta_g$ are concave and continuously differentiable,
so \eqref{eq:convex_program} is a continuously differentiable convex
problem; see \autocite[Theorem~2.1]{locatellischoen2014convex},
whose proof applies equally to nonconvex polygons.

\section{Computation of convex envelopes for the bilinear case}
\label{sec:comp-conv}
We now focus on the bilinear function $f(x,y)=x\cdot y$.
In this case, $E'(\polygon)$ consists of the positive-slope edges, and
\[
\eta_e(a,b)=-\frac{(a+m_e b-q_e)^2}{4m_e}-bq_e
\qquad\text{for }(a,b)\in D_e.
\]
In \cite{locatelli2018convex}, Locatelli develops a KKT-based procedure
for computing convex envelopes over polytopes. We adapt this approach
to possibly nonconvex polygonal domains. Several structural
simplifications available for convex polytopes no longer apply in
this setting. Examining these limitations provides a precise
understanding of the differences between the two classes of domains.

 Central to this approach is the following KKT system associated with problem \eqref{eq:convex_program}, stated for a given point $(x,y)\in\polygon$:
\begin{align*}
\sum_{g\in G(\polygon)} \lambda_g x_g(\boldsymbol{\alpha}) &= x, \\
\sum_{g\in G(\polygon)} \lambda_g y_g(\boldsymbol{\alpha}) &= y, \\
\sum_{g\in G(\polygon)} \lambda_g &= 1, \\
\lambda_g &\ge 0 \quad \forall g\in G(\polygon), \\
\eta_g(\boldsymbol{\alpha}) &\ge c \quad \forall g\in G(\polygon), \\
\lambda_g \left( \eta_g(\boldsymbol{\alpha}) - c \right) &= 0 \quad \forall g\in G(\polygon).
\end{align*}

By Carathéodory's theorem, it suffices to consider KKT solutions with at most three positive multipliers. We may therefore restrict attention to generator sets
\[
\mathcal{J}
\coloneqq
\left\{
J\subseteq G(\polygon)
:\;1\leq |J|\leq 3
\right\}.
\]
Each $J$ thus consists directly of vertices and edges. The corresponding constraints are required to be
active, that is,
\[
\eta_g(\boldsymbol{\alpha})=c
\qquad\text{for all }g\in J.
\]

Equating the values of $\eta_g(\boldsymbol{\alpha})$ for $g \in J$ allows us to eliminate $c$, resulting in the following systems for all $J \in \mathcal{J}$:

\begin{equation}
	\begin{aligned}
		\sum_{g \in J} \lambda_g x_g(\boldsymbol{\alpha}) &= x, \\
		\sum_{g \in J} \lambda_g y_g(\boldsymbol{\alpha}) &= y, \\
		\sum_{g \in J} \lambda_g &= 1, \\
		\lambda_g &\ge 0 \quad \forall g \in J, \\
		\eta_g(\boldsymbol{\alpha}) - \eta_h(\boldsymbol{\alpha}) &= 0 \quad \forall g,h\in J,\ g\neq h, \\
		\eta_g(\boldsymbol{\alpha}) &\le \eta_r(\boldsymbol{\alpha}) \quad \forall g\in J,\ r\in G(\polygon)\setminus J.
	\end{aligned}
	\tag{KKT[x, y, J]}
\end{equation}
Consistent with \cite{locatelli2018convex} we denote by $\Gamma_J$ the \textit{KKT region associated with $J$}, that is, the set of points $(x, y) \in \polygon$ for which $\mathrm{KKT}[x,y,J]$ is feasible.
Not all generator sets $J\in\mathcal{J}$ need to be considered; the following subsection identifies those that can be excluded.

\subsection{Exclusion of generator sets}

For the bilinear case over polytopes, \cite{locatelli2018convex} reduces the relevant generator sets $J$ to three types: three vertices; one vertex and either a lower or an upper edge; and one lower together with one upper edge. Here, lower and upper edges are positive-slope edges such that the entire domain lies above or below their supporting lines, respectively. More precisely,
\begin{align*}
	&\bar{E}^l(\polygon) = \{e \in E(\polygon) : m_e > 0, \polygon \subseteq \{(x, y) \in \mathbb{R}^2: y \ge m_e x + q_e\}\} \\
	&\bar{E}^u(\polygon) = \{e \in E(\polygon) : m_e > 0, \polygon \subseteq \{(x, y) \in \mathbb{R}^2: y \le m_e x + q_e\}\}.
\end{align*} 
 Some of the reductions established for polytopes remain valid for nonconvex polygons, as summarized in the following proposition:
\begin{proposition}\label{prop:kkt_index_reductions}The following restrictions can be applied:
	\begin{itemize}
		\item If $e\in J$ is a positive-slope edge, we may restrict attention to KKT solutions satisfying $\boldsymbol{\alpha}\in\operatorname{int}(D_e)$ (cf. \autocite[Observation~3.1]{locatelli2018convex}). Thus, selected edge generators have their contact points in the relative interiors of their edges; endpoint contacts are represented by vertex generators.
		\item We can exclude $J$ with $\dim(\conv(\bigcup_{g\in J}g))<2$. In particular, we can restrict our attention to $J \in \mathcal{J}$ such that $|J| \ge 2$. Moreover, we can also discard all sets $J=\{v_1,v_2\}$ consisting of two vertices (cf. \autocite[Observation~3.2, Corollary~3.1]{locatelli2018convex}).
		\item We can exclude from $\mathcal{J}$ every set $J$ such that
		$\{e_1,e_2\}\subseteq J$, $e_1\neq e_2$, and
		$e_1,e_2\in\bar{E}^{l}(\polygon)$ or
		$e_1,e_2\in\bar{E}^{u}(\polygon)$ (cf. \autocite[Observation~4.1]{locatelli2018convex}).
	\end{itemize}
\end{proposition} 
\noindent The arguments used to establish some of the remaining reductions for polytopes cannot be applied directly to nonconvex polygons. The following geometric interpretation, specific to the bilinear function, illustrates the underlying differences between the two settings: Let $\ell_{a,b,c}(x,y) \coloneqq ax+by+c$ be an affine function and consider its difference from $f(x,y)=xy$,
\begin{equation*}\label{eq:bilinear_affine_difference}
	\delta_{a,b,c}(x,y)
	\coloneqq xy-\ell_{a,b,c}(x,y)
	=(x-b)(y-a)-(ab+c).
\end{equation*}
The function $\ell_{a,b,c}$ is an affine underestimator of $f$ on $\polygon$ if and only if
\begin{equation}\label{eq:bilinear_underestimator_condition}
	\delta_{a,b,c}(x,y)\geq 0
	\qquad\text{for all }(x,y)\in\polygon.
\end{equation}
Consequently, its contact set with the graph of $f$ is
\begin{equation}\label{eq:bilinear_contact_set}
	C(a,b,c)
	\coloneqq \bigl\{(x,y)\in\polygon : \delta_{a,b,c}(x,y)=0\bigr\}
	=\polygon\cap\bigl\{(x,y) \in \mathbb{R}^2:(x-b)(y-a)=ab+c\bigr\}.
\end{equation}

\noindent The following proposition gives a geometric restriction on pairs of contact points of a valid affine underestimator.

\begin{proposition}\label{prop:positive_slope_contacts}
	Let $\ell_{a,b,c}$ be an affine underestimator of $f(x,y)=xy$ on $\polygon$, and let $\mathbf{p},\mathbf{q}\in C(a,b,c)$ be two distinct contact points. If the line through $\mathbf{p}$ and $\mathbf{q}$ has positive slope, i.e.,
	\[
		(q_x-p_x)(q_y-p_y)>0,
	\]
	then the open line segment $(\mathbf{p},\mathbf{q})$ does not intersect $\polygon$.
	Therefore, we can exclude from $\mathcal{J}$ every set $J$
	containing distinct generators $g,h$ such that, for every
	$\mathbf{p}\in\operatorname{relint}(g)$ and
	$\mathbf{q}\in\operatorname{relint}(h)$, we have
	$(q_x-p_x)(q_y-p_y)>0$ and the open line segment
	$(\mathbf{p},\mathbf{q})$ intersects $\polygon$.
\end{proposition}
\begin{proof}
	For $t\in(0,1)$, let $\mathbf{z}(t)=t\mathbf{p}+(1-t)\mathbf{q}$, which lies in the open segment $(\mathbf{p},\mathbf{q})$. Since $\ell_{a,b,c}$ is affine and agrees with $f$ at both contact points,
	\[
	\ell_{a,b,c}(\mathbf{z}(t))=t p_xp_y+(1-t)q_xq_y.
	\]
	Writing $\Delta_x\coloneqq q_x-p_x$ and $\Delta_y\coloneqq q_y-p_y$, we have $z_x(t)=q_x-t\Delta_x$ and $z_y(t)=q_y-t\Delta_y$. Therefore,
	\begin{align*}
	\delta_{a,b,c}(\mathbf{z}(t))
	&=(q_x-t\Delta_x)(q_y-t\Delta_y)-\bigl(t p_xp_y+(1-t)q_xq_y\bigr)\\
	&=-t(1-t)\Delta_x\Delta_y\\
	&=-t(1-t)(q_x-p_x)(q_y-p_y)<0,
	\end{align*}
	where the final inequality follows from the assumption. Thus, any point of $(\mathbf{p},\mathbf{q})\cap\polygon$ would violate \eqref{eq:bilinear_underestimator_condition}, and therefore the open segment cannot intersect $\polygon$.
\end{proof}

For a polytope in $\mathbb{R}^2$, i.e., a convex polygon, the entire segment between any two points of the domain belongs to the domain. Therefore, no two distinct contact points can lie on a line of positive slope. For a nonconvex polygon, however, the segment connecting two such contact points may pass through a notch outside the polygon. In that case, both points can belong to the contact set of the same affine underestimator. This is a fundamental difference between the two settings: combinations of generators that are impossible in the convex case may become feasible in the nonconvex case.

We now examine the geometry of the contact set together with the feasible regions for the polygon in more detail. As can be seen in \eqref{eq:bilinear_contact_set}, the contact points of any affine underestimator lie on a hyperbola centered at $(b,a)$. Writing $\kappa\coloneqq ab+c$, three cases arise, which can be seen in \Cref{fig:hyperbola-cases}. If $\kappa>0$, the two branches lie in the regions $x>b,y>a$ and $x<b,y<a$. If $\kappa<0$, they lie in the regions $x<b,y>a$ and $x>b,y<a$. Finally, if $\kappa=0$, the hyperbola degenerates into the two lines $x=b$ and $y=a$. In each case, condition \eqref{eq:bilinear_underestimator_condition} requires the whole polygon to lie in the corresponding superlevel set
\[
	\bigl\{(x,y):(x-b)(y-a)\geq\kappa\bigr\}.
\]

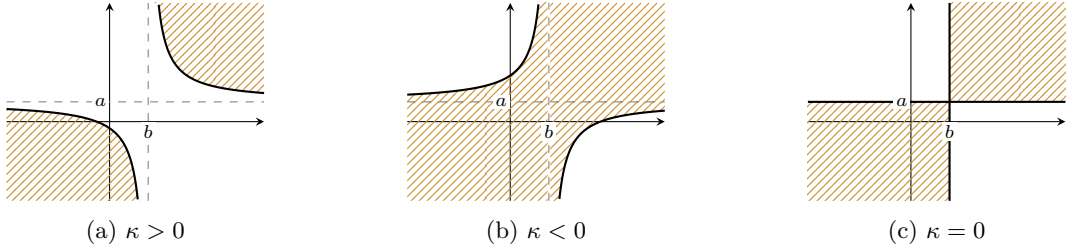
\begin{figure}[H]
    \centering
    \begin{subfigure}[t]{0.32\textwidth}
        \centering
        \begin{tikzpicture}
            \begin{axis}[
                width=\linewidth,
                height=4.2cm,
                axis lines=middle,
                xmin=-4, xmax=6,
                ymin=-4, ymax=6,
                samples=200,
                clip=true,
                ticks=none
            ]
                \def\a{1}
                \def\b{1.5}
                \def\kappa{2}
                \def\eps{0.05}

                \addplot[dashed,gray] coordinates {(\b,-4) (\b,6)};
                \addplot[dashed,gray] coordinates {(-4,\a) (6,\a)};

                \addplot[draw=none,domain={\b+\eps}:6,name path=positive-top] {6};
                \addplot[draw=none,domain={\b+\eps}:6,name path=positive-right]
                    {\a+\kappa/(x-\b)};
                \addplot[pattern=north east lines,pattern color=color0,draw=none]
                    fill between[of=positive-right and positive-top];

                \addplot[draw=none,domain=-4:{\b-\eps},name path=positive-bottom] {-4};
                \addplot[draw=none,domain=-4:{\b-\eps},name path=positive-left]
                    {\a+\kappa/(x-\b)};
                \addplot[pattern=north east lines,pattern color=color0,draw=none]
                    fill between[of=positive-left and positive-bottom];

                \addplot[thick,domain=-4:{\b-\eps},restrict y to domain=-4:6]
                    {\a+\kappa/(x-\b)};
                \addplot[thick,domain={\b+\eps}:6,restrict y to domain=-4:6]
                    {\a+\kappa/(x-\b)};

                \node[below,font=\scriptsize,fill=white,inner sep=1pt]
                    at (axis cs:\b,0) {$b$};
                \node[left,font=\scriptsize,fill=white,inner sep=1pt]
                    at (axis cs:0,\a) {$a$};
            \end{axis}
        \end{tikzpicture}
        \caption{$\kappa>0$}
        \label{fig:hyperbola-positive}
    \end{subfigure}
    \hfill
    \begin{subfigure}[t]{0.32\textwidth}
        \centering
        \begin{tikzpicture}
            \begin{axis}[
                width=\linewidth,
                height=4.2cm,
                axis lines=middle,
                xmin=-4, xmax=6,
                ymin=-4, ymax=6,
                samples=200,
                clip=true,
                ticks=none
            ]
                \def\a{1}
                \def\b{1.5}
                \def\kappa{2}
                \def\eps{0.05}

                \addplot[dashed,gray] coordinates {(\b,-4) (\b,6)};
                \addplot[dashed,gray] coordinates {(-4,\a) (6,\a)};

                \addplot[draw=none,domain=-4:{\b-\eps},name path=negative-bottom] {-4};
                \addplot[draw=none,domain=-4:{\b-\eps},name path=negative-left]
                    {\a-\kappa/(x-\b)};
                \addplot[pattern=north east lines,pattern color=color0,draw=none]
                    fill between[of=negative-bottom and negative-left];

                \addplot[draw=none,domain={\b+\eps}:6,name path=negative-top] {6};
                \addplot[draw=none,domain={\b+\eps}:6,name path=negative-right]
                    {\a-\kappa/(x-\b)};
                \addplot[pattern=north east lines,pattern color=color0,draw=none]
                    fill between[of=negative-right and negative-top];

                \addplot[thick,domain=-4:{\b-\eps},restrict y to domain=-4:6]
                    {\a-\kappa/(x-\b)};
                \addplot[thick,domain={\b+\eps}:6,restrict y to domain=-4:6]
                    {\a-\kappa/(x-\b)};

                \node[below,font=\scriptsize,fill=white,inner sep=1pt]
                    at (axis cs:\b,0) {$b$};
                \node[left,font=\scriptsize,fill=white,inner sep=1pt]
                    at (axis cs:0,\a) {$a$};
            \end{axis}
        \end{tikzpicture}
        \caption{$\kappa<0$}
        \label{fig:hyperbola-negative}
    \end{subfigure}
    \hfill
    \begin{subfigure}[t]{0.32\textwidth}
        \centering
        \begin{tikzpicture}
            \begin{axis}[
                width=\linewidth,
                height=4.2cm,
                axis lines=middle,
                xmin=-4, xmax=6,
                ymin=-4, ymax=6,
                clip=true,
                ticks=none
            ]
                \def\a{1}
                \def\b{1.5}

                \path[pattern=north east lines,pattern color=color0,draw=none]
                    (axis cs:-4,-4) rectangle (axis cs:\b,\a);
                \path[pattern=north east lines,pattern color=color0,draw=none]
                    (axis cs:\b,\a) rectangle (axis cs:6,6);

                \addplot[thick] coordinates {(\b,-4) (\b,6)};
                \addplot[thick] coordinates {(-4,\a) (6,\a)};

                \node[below,font=\scriptsize,fill=white,inner sep=1pt]
                    at (axis cs:\b,0) {$b$};
                \node[left,font=\scriptsize,fill=white,inner sep=1pt]
                    at (axis cs:0,\a) {$a$};
            \end{axis}
        \end{tikzpicture}
        \caption{$\kappa=0$}
        \label{fig:hyperbola-zero}
    \end{subfigure}
    \caption{Level sets $(x-b)(y-a)=\kappa$ for the three possible signs of
    $\kappa=ab+c$. The hatched regions represent the corresponding superlevel
    sets $(x-b)(y-a)\geq\kappa$ in which the polygon must be contained.}
    \label{fig:hyperbola-cases}
\end{figure}

The three cases also provide a geometric interpretation of \Cref{prop:positive_slope_contacts}. If $\kappa>0$, the feasible superlevel set has two disconnected components. Since $\polygon$ is connected, its contact points lie on a single hyperbola branch, which is decreasing. If $\kappa=0$, the feasible set consists of two diagonally opposite quadrants meeting only at $(b,a)$. Since we assume that polygons are simple, $\polygon$ cannot occupy both quadrants. Within either quadrant, contact points lie on the boundary rays and cannot form a line of positive slope. Consequently, two contact points joined by a line of positive slope can occur only when $\kappa<0$, and they must lie on the
same branch of the hyperbola.

Now consider an affine underestimator with a contact point in the relative interior of a positive-slope edge $e$. Such a contact is possible only if $\kappa<0$. For $\kappa>0$, the edge would cross a decreasing hyperbola branch and enter the infeasible region. For $\kappa=0$, it would have to pass through $(b,a)$ and occupy both feasible quadrants, which is impossible for a simple polygon, as argued above. This geometric interpretation explains why certain combinations of generators can be excluded directly for polytopes. For instance, any triple containing $e$ must contain at least two generators whose contact points lie on the same branch. The segment connecting these two contact points would enter the infeasible region, contradicting the convexity of the polytope. Hence, all generator sets of cardinality three containing a positive-slope edge can be excluded, leaving only triples of vertices; see \autocite[Observation~4.2]{locatelli2018convex}.

For nonconvex polygons, triples containing positive-slope edges can be necessary in the KKT enumeration. \Cref{fig:edge_triples} provides examples for all three types: one vertex and two edges, two vertices and one edge, and three edges. The following proposition establishes the uniqueness of the optimal affine underestimator and explains why the corresponding triples cannot be discarded.

\begin{figure}[H]
    \centering
    % In all three panels, a=b=1/2 and kappa=-1.
    \begin{subfigure}[t]{0.32\textwidth}
        \centering
        \begin{tikzpicture}
            \begin{axis}[
                width=\linewidth,
                height=4.5cm,
                axis lines=middle,
                xmin=-1.6, xmax=2.6,
                ymin=-1.2, ymax=3.2,
                samples=200,
                clip=true,
                ticks=none
            ]
                \def\a{0.5}
                \def\b{0.5}
                \def\eps{0.035}

                \addplot[dashed,gray] coordinates {(\b,-1.2) (\b,3.2)};
                \addplot[dashed,gray] coordinates {(-1.6,\a) (2.6,\a)};

                \addplot[draw=none,domain=-1.6:{\b-\eps},name path=vee-bottom] {-1.2};
                \addplot[draw=none,domain=-1.6:{\b-\eps},name path=vee-left]
                    {\a-1/(x-\b)};
                \addplot[pattern=north east lines,pattern color=color0,draw=none]
                    fill between[of=vee-bottom and vee-left];
                \addplot[draw=none,domain={\b+\eps}:2.6,name path=vee-top] {3.2};
                \addplot[draw=none,domain={\b+\eps}:2.6,name path=vee-right]
                    {\a-1/(x-\b)};
                \addplot[pattern=north east lines,pattern color=color0,draw=none]
                    fill between[of=vee-right and vee-top];

                \addplot[thick,domain=-1.6:{\b-\eps},restrict y to domain=-1.2:3.2]
                    {\a-1/(x-\b)};
                \addplot[thick,domain={\b+\eps}:2.6,restrict y to domain=-1.2:3.2]
                    {\a-1/(x-\b)};

                \addplot[draw=color1,fill=color3!20,fill opacity=0.72,thick]
                    coordinates {(-0.9,1.1) (-0.1,1.9) (0,2.5) (0.3,1.5)
                    (2.3,1.5) (1.9,-0.1) (1.1,-0.9) (0.7,-0.5)
                    (0.3,-0.5) (-1.3,-0.5) (-0.9,1.1)};

                \addplot[color2,very thick] coordinates {(-0.9,1.1) (-0.1,1.9)};
                \addplot[color2,very thick] coordinates {(1.1,-0.9) (1.9,-0.1)};
                \addplot[only marks,mark=*,mark size=1.8pt,color2]
                    coordinates {(-0.5,1.5) (0,2.5) (1.5,-0.5)};

                \node[below,font=\scriptsize,fill=white,inner sep=1pt]
                    at (axis cs:\b,0) {$b$};
                \node[left,font=\scriptsize,fill=white,inner sep=1pt]
                    at (axis cs:0,\a) {$a$};
            \end{axis}
        \end{tikzpicture}
        \caption{One vertex and two edges.}
        \label{fig:vertex-edge-edge}
    \end{subfigure}
    \hfill
    \begin{subfigure}[t]{0.32\textwidth}
        \centering
        \begin{tikzpicture}
            \begin{axis}[
                width=\linewidth,
                height=4.5cm,
                axis lines=middle,
                xmin=-1.6, xmax=2.6,
                ymin=-1.2, ymax=3.2,
                samples=200,
                clip=true,
                ticks=none
            ]
                \def\a{0.5}
                \def\b{0.5}
                \def\eps{0.035}

                \addplot[dashed,gray] coordinates {(\b,-1.2) (\b,3.2)};
                \addplot[dashed,gray] coordinates {(-1.6,\a) (2.6,\a)};

                \addplot[draw=none,domain=-1.6:{\b-\eps},name path=vve-bottom] {-1.2};
                \addplot[draw=none,domain=-1.6:{\b-\eps},name path=vve-left]
                    {\a-1/(x-\b)};
                \addplot[pattern=north east lines,pattern color=color0,draw=none]
                    fill between[of=vve-bottom and vve-left];
                \addplot[draw=none,domain={\b+\eps}:2.6,name path=vve-top] {3.2};
                \addplot[draw=none,domain={\b+\eps}:2.6,name path=vve-right]
                    {\a-1/(x-\b)};
                \addplot[pattern=north east lines,pattern color=color0,draw=none]
                    fill between[of=vve-right and vve-top];

                \addplot[thick,domain=-1.6:{\b-\eps},restrict y to domain=-1.2:3.2]
                    {\a-1/(x-\b)};
                \addplot[thick,domain={\b+\eps}:2.6,restrict y to domain=-1.2:3.2]
                    {\a-1/(x-\b)};

                \addplot[draw=color1,fill=color3!20,fill opacity=0.72,thick]
                    coordinates {(-0.5,1.5) (-0.25,0.5) (0,2.5) (0.3,1.5)
                    (2.3,1.5) (1.9,-0.1) (1.1,-0.9) (0.7,-0.5)
                    (0.3,-0.5) (-1.3,-0.5) (-0.5,1.5)};

                \addplot[color2,very thick] coordinates {(1.1,-0.9) (1.9,-0.1)};
                \addplot[only marks,mark=*,mark size=1.8pt,color2]
                    coordinates {(-0.5,1.5) (0,2.5) (1.5,-0.5)};

                \node[below,font=\scriptsize,fill=white,inner sep=1pt]
                    at (axis cs:\b,0) {$b$};
                \node[left,font=\scriptsize,fill=white,inner sep=1pt]
                    at (axis cs:0,\a) {$a$};
            \end{axis}
        \end{tikzpicture}
        \caption{Two vertices and one edge.}
        \label{fig:vertex-vertex-edge}
    \end{subfigure}
    \hfill
    \begin{subfigure}[t]{0.32\textwidth}
        \centering
        \begin{tikzpicture}
            \begin{axis}[
                width=\linewidth,
                height=4.5cm,
                axis lines=middle,
                xmin=-1.6, xmax=2.6,
                ymin=-1.2, ymax=3.2,
                samples=200,
                clip=true,
                ticks=none
            ]
                \def\a{0.5}
                \def\b{0.5}
                \def\eps{0.035}

                \addplot[dashed,gray] coordinates {(\b,-1.2) (\b,3.2)};
                \addplot[dashed,gray] coordinates {(-1.6,\a) (2.6,\a)};

                \addplot[draw=none,domain=-1.6:{\b-\eps},name path=eee-bottom] {-1.2};
                \addplot[draw=none,domain=-1.6:{\b-\eps},name path=eee-left]
                    {\a-1/(x-\b)};
                \addplot[pattern=north east lines,pattern color=color0,draw=none]
                    fill between[of=eee-bottom and eee-left];
                \addplot[draw=none,domain={\b+\eps}:2.6,name path=eee-top] {3.2};
                \addplot[draw=none,domain={\b+\eps}:2.6,name path=eee-right]
                    {\a-1/(x-\b)};
                \addplot[pattern=north east lines,pattern color=color0,draw=none]
                    fill between[of=eee-right and eee-top];

                \addplot[thick,domain=-1.6:{\b-\eps},restrict y to domain=-1.2:3.2]
                    {\a-1/(x-\b)};
                \addplot[thick,domain={\b+\eps}:2.6,restrict y to domain=-1.2:3.2]
                    {\a-1/(x-\b)};

                \addplot[draw=color1,fill=color3!20,fill opacity=0.72,thick]
                    coordinates {(-0.9,1.1) (-0.1,1.9) (-0.1,2.1)
                    (0.1,2.9) (0.3,1.5) (2.3,1.5)
                    (1.9,-0.1) (1.1,-0.9) (0.7,-0.5) (0.3,-0.5)
                    (-1.3,-0.5) (-0.9,1.1)};

                \addplot[color2,very thick] coordinates {(-0.9,1.1) (-0.1,1.9)};
                \addplot[color2,very thick] coordinates {(-0.1,2.1) (0.1,2.9)};
                \addplot[color2,very thick] coordinates {(1.1,-0.9) (1.9,-0.1)};
                \addplot[only marks,mark=*,mark size=1.8pt,color2]
                    coordinates {(-0.5,1.5) (0,2.5) (1.5,-0.5)};

                \node[below,font=\scriptsize,fill=white,inner sep=1pt]
                    at (axis cs:\b,0) {$b$};
                \node[left,font=\scriptsize,fill=white,inner sep=1pt]
                    at (axis cs:0,\a) {$a$};
            \end{axis}
        \end{tikzpicture}
        \caption{Three edges.}
        \label{fig:edge-edge-edge}
    \end{subfigure}

    \caption{Examples of necessary generator triples for nonconvex polygons.}
    \label{fig:edge_triples}
\end{figure}
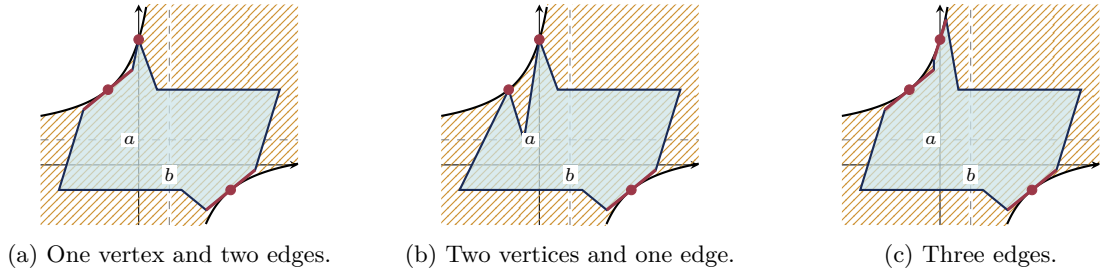

\begin{proposition}\label{prop:unique_contact_triple}
	Let $\ell_{a,b,c}$ be an affine underestimator of $f(x,y)=x \cdot y$ on $\polygon$ whose contact set consists of exactly three points $\mathbf{p}_1,\mathbf{p}_2,\mathbf{p}_3$ that do not lie on a common line. Then $\ell_{a,b,c}$ is the unique optimal affine underestimator at every point
	\[
	\mathbf{z}\in\polygon\cap\operatorname{int}\bigl(\operatorname{conv}\{\mathbf{p}_1,\mathbf{p}_2,\mathbf{p}_3\}\bigr).
	\]
	Thus, if each contact point is either a vertex or lies in the relative interior of a positive-slope edge, the three corresponding generators form the unique generator set $J$ satisfying the KKT system at $\mathbf{z}$ under the convention that selected edge generators have contact points in their relative interiors.
\end{proposition}
\begin{proof}
	Write $\mathbf{z}=\sum_{i=1}^3\lambda_i\mathbf{p}_i$, where $\lambda_i>0$ and $\sum_{i=1}^3\lambda_i=1$. For any affine underestimator $\tilde\ell$ on $\polygon$,
	\begin{equation}\label{eq:unique_underestimator}
	\tilde\ell(\mathbf{z})
	=\sum_{i=1}^3\lambda_i\tilde\ell(\mathbf{p}_i)
	\leq\sum_{i=1}^3\lambda_i f(\mathbf{p}_i)
	= \sum_{i=1}^3 \lambda_i \ell_{a,b,c}(\mathbf{p}_i)
	=\ell_{a,b,c}(\mathbf{z}).
	\end{equation}
	Thus, $\ell_{a,b,c}$ is optimal at $\mathbf{z}$. Now suppose that $\tilde\ell$ is also an optimal affine underestimator at $\mathbf{z}$. Then $\tilde\ell(\mathbf{z}) = \ell_{a,b,c}(\mathbf{z})$, so equality holds throughout \eqref{eq:unique_underestimator}. Since $\tilde\ell(\mathbf{p}_i) \le f(\mathbf{p}_i)$ and $\ell_{a,b,c}(\mathbf{p}_i) \le f(\mathbf{p}_i)$, the strict positivity of the multipliers $\lambda_i$ implies that $\tilde\ell(\mathbf{p}_i) = f(\mathbf{p}_i) = \ell_{a,b,c}(\mathbf{p}_i)$ for all $i$. Therefore, $\tilde\ell = \ell_{a,b,c}$, since an affine function on $\mathbb{R}^2$ is uniquely determined by its values at three points that do not lie on a common line. Hence, $\ell_{a,b,c}$ is the unique optimal affine underestimator at $\mathbf{z}$, which concludes the proof.
\end{proof}

A similar complication arises for generator sets of cardinality two. Unlike in the polytope case, the sets $\bar{E}^l(\polygon)\cup\bar{E}^u(\polygon)$ do not necessarily contain all positive-slope edges of a nonconvex polygon, particularly its inner edges. Consequently, pairs involving such edges cannot be excluded using the same geometric arguments as for lower and upper edges of a polytope.

Nevertheless, certain combinations can also be excluded for nonconvex
polygons, as formalized in the following proposition.

\begin{proposition}\label{prop:exclude_same_line}
	Let $J\in\mathcal{J}$. We can exclude $J$ if it contains distinct
	generators $g,h$ that are contained in the same line and at least one of them is
	an edge of $\polygon$.
\end{proposition}
\begin{proof}
Consider an affine underestimator $\ell_{a,b,c}$ and a positive-slope
edge $e$ with a contact point in its relative interior. Let
$y=m_e x+q_e$ be the supporting line of $e$. The restriction of
$\delta_{a,b,c}$ to this line is a strictly convex quadratic function.
Since it is nonnegative on $e$ and vanishes at a point in its relative interior,
this point is its unique minimizer, with minimum value zero.
Hence, there can be no other contact point on the supporting line.
\end{proof}
This restriction can also be seen geometrically in the middle panel of \Cref{fig:hyperbola-cases}: a positive-slope edge with a contact point in its relative interior is tangent to the hyperbola, so its supporting line has no second contact point.

The preceding results identify several generator sets that can be discarded. For the remaining sets, one must still determine the regions on which their KKT solutions are feasible. This step highlights another difference between the polytope and nonconvex polygon cases:

\begin{remark}
	For polytopes, Locatelli simplifies the cell descriptions associated with two-element generator sets by eliminating the inequalities $\eta_g(\boldsymbol{\alpha})\le \eta_r(\boldsymbol{\alpha})$ for $g\in J$ and $r\in G(\polygon)\setminus J$. In the vertex--edge case, these inequalities can be replaced by a small number of linear inequalities, while in the edge--edge case they can be omitted altogether; see \autocite[Appendix~A]{locatelli2018convex}. The arguments used to establish these simplifications rely on the structure of polytopes and do not readily carry over to nonconvex polygons. We therefore retain the $\eta$-inequalities in our polygonal setting and handle them explicitly.
\end{remark}

\subsection{Case distinction for remaining generator sets}
We now distinguish the remaining generator sets by their numbers of
vertices and edges and derive the corresponding KKT regions and
convex envelope expressions.

\subsubsection{Generator sets with three elements}

For a fixed generator set $J$, let $s$ index the feasible candidates. Each candidate determines an affine function $\ell_s(x,y)=a_s^\ast x+b_s^\ast y+c_s^\ast$ and an associated KKT region $\Gamma_J^{(s)}$ on which the convex envelope coincides with $\ell_s$. The full KKT region associated with $J$ is
\[
\Gamma_J=\bigcup_s\Gamma_J^{(s)}.
\]
The affine functions associated with different candidates need not coincide.

\subsubsection*{Three vertices}
\label{sec:kkt_three_vertices}
This case is analogous to \autocite[Appendix~A.1]{locatelli2018convex}. By \Cref{prop:kkt_index_reductions}, triples of collinear vertices can be excluded because their convex hull has dimension less than two.
Assume that $J=\{v_1,v_2,v_3\}$, where the three vertices do not lie on a common line. They uniquely determine an affine function $\ell(x,y)=ax+by+c$ satisfying $\ell(v_i)=f(v_i)$ for $i=1,2,3$. If $\eta_{v_1}(a,b)\leq\eta_g(a,b)$ for every $g\in G(\polygon)\setminus J$, then $\ell$ is an affine underestimator of $f$ on $\polygon$. In this case,
\[
\Gamma_J
=\polygon\cap\operatorname{conv}\{v_1,v_2,v_3\},
\qquad
\conv_{f,\polygon}(x,y)=\ell(x,y)
\quad\text{for all }(x,y)\in\Gamma_J.
\]

\subsubsection*{Two vertices -- one edge}
\label{sec:kkt_two_vertices_one_edge}
By \Cref{prop:exclude_same_line}, triples in which at least one of the two vertices lies on the supporting line of the edge can be excluded. Assume therefore
that $J=\{v_1,v_2,e\}$, where neither vertex lies on this line.
Introduce the substitution $z\coloneqq a+m_e b$. For
$(a,b)\in\operatorname{int}(D_e)$, a straightforward calculation shows that
\[
\eta_e(a,b)=\eta_{v_2}(a,b)
\quad\Longleftrightarrow\quad
b=\beta_2z^2+\beta_1z+\beta_0,
\]
where, for a vertex--edge pair $(v,e)$, the coefficients are given by
\begin{equation}\label{eq:vertex_edge_beta_coefficients}
\beta_2=-\frac{1}{C},
\qquad
\beta_1=\frac{2q_e+4m_e v_x}{C},
\qquad
\beta_0=-\frac{q_e^2+4m_e v_xv_y}{C},
\qquad
C\coloneqq4m_e(q_e+m_e v_x-v_y).
\end{equation}
Here, we use $v=v_2$; see
\autocite[Section~5.2]{locatelli2016_polyhedral}.
Substituting $a=z-m_e b$ and then $b=\beta_2z^2+\beta_1z+\beta_0$ into
$\eta_{v_1}(a,b)=\eta_{v_2}(a,b)$ yields a univariate polynomial
equation of degree at most two in $z$.
Its real roots give at most two candidate pairs
$(a_s^\ast,b_s^\ast)$, with
\[
c_s^\ast
\coloneqq v_{1,x}v_{1,y}
-a_s^\ast v_{1,x}-b_s^\ast v_{1,y}.
\]

For each candidate $s$, we check whether
\[
(a_s^\ast,b_s^\ast)\in\operatorname{int}(D_e)
\quad\text{and}\quad
\eta_{v_1}(a_s^\ast,b_s^\ast)
\leq\eta_g(a_s^\ast,b_s^\ast)
\quad\text{for every }g\in G(\polygon)\setminus J.
\]
If both conditions hold, its associated KKT region is
\[
\Gamma_J^{(s)}
=\polygon\cap\operatorname{conv}
\{v_1,v_2,\mathbf{x}_e(a_s^\ast,b_s^\ast)\},
\]
and
\[
\conv_{f,\polygon}(x,y)=a_s^\ast x+b_s^\ast y+c_s^\ast
\quad\text{for all }(x,y)\in\Gamma_J^{(s)}.
\]

\subsubsection*{One vertex -- two edges}
\label{sec:kkt_one_vertex_two_edges}
By \Cref{prop:exclude_same_line}, triples in which the vertex lies
on the supporting line of either edge, or both edges lie on the
same line, can be excluded. Assume therefore that
$J=\{v,e_1,e_2\}$ satisfies neither exclusion condition.

As in the preceding case, introduce the substitution
$z\coloneqq a+m_{e_1}b$. For
$(a,b)\in\operatorname{int}(D_{e_1})$, we obtain
\[
\eta_{e_1}(a,b)=\eta_v(a,b)
\quad\Longleftrightarrow\quad
b=\beta_2z^2+\beta_1z+\beta_0,
\]
where the coefficients $\beta_0,\beta_1,\beta_2$ are given by
\eqref{eq:vertex_edge_beta_coefficients} for the pair $(v,e_1)$.
Substituting $a=z-m_{e_1}b$ and then
$b=\beta_2z^2+\beta_1z+\beta_0$ into
$\eta_v(a,b)=\eta_{e_2}(a,b)$, using the expression for
$\eta_{e_2}$ on $\operatorname{int}(D_{e_2})$, yields a
univariate polynomial equation of degree at most four in $z$.
Its real roots give at most four candidate pairs
$(a_s^\ast,b_s^\ast)$, with
\[
c_s^\ast
\coloneqq v_xv_y-a_s^\ast v_x-b_s^\ast v_y.
\]

For each candidate $s$, we check whether
\[
(a_s^\ast,b_s^\ast)
\in\operatorname{int}(D_{e_1})\cap\operatorname{int}(D_{e_2})
\quad\text{and}\quad
\eta_v(a_s^\ast,b_s^\ast)
\leq\eta_g(a_s^\ast,b_s^\ast)
\quad\text{for every }g\in G(\polygon)\setminus J.
\]
If both conditions hold, its associated
KKT region is
\[
\Gamma_J^{(s)}
=\polygon\cap\operatorname{conv}
\{v,\mathbf{x}_{e_1}(a_s^\ast,b_s^\ast),
     \mathbf{x}_{e_2}(a_s^\ast,b_s^\ast)\},
\]
and
\[
\conv_{f,\polygon}(x,y)=a_s^\ast x+b_s^\ast y+c_s^\ast
\quad\text{for all }(x,y)\in\Gamma_J^{(s)}.
\]

\subsubsection*{Three edges}
\label{sec:kkt_three_edges}
By \Cref{prop:exclude_same_line}, triples containing two edges on the same supporting line can be excluded. Assume therefore that $J=\{e_1,e_2,e_3\}$ consists of positive-slope edges with pairwise distinct supporting lines.

For $(a,b)\in\operatorname{int}(D_{e_1})\cap\operatorname{int}(D_{e_2})$, the equality $\eta_{e_1}(a,b)=\eta_{e_2}(a,b)$ reduces to one or two linear relations between $a$ and $b$; see \autocite[Lemma~5.1]{locatelli2016_polyhedral}. If $m_{e_1}=m_{e_2}$, we obtain
\begin{equation}\label{eq:parallel_edge_relation}
\eta_{e_1}(a,b)=\eta_{e_2}(a,b)
\quad\Longleftrightarrow\quad
a=m_{e_1}b+\frac{q_{e_1}+q_{e_2}}{2}.
\end{equation}
Otherwise, the equality is equivalent to the union of the two lines
\begin{equation}\label{eq:nonparallel_edge_relations}
a=\frac{q_{e_2}m_{e_1}-q_{e_1}m_{e_2}
\pm\sqrt{m_{e_1}m_{e_2}}\bigl((m_{e_1}-m_{e_2})b+q_{e_1}-q_{e_2}\bigr)}
{m_{e_1}-m_{e_2}}.
\end{equation}
For each resulting relation, we substitute the expression for $a$ into $\eta_{e_2}(a,b)=\eta_{e_3}(a,b)$, using the expressions for these functions on $\operatorname{int}(D_{e_2})\cap\operatorname{int}(D_{e_3})$. This yields a univariate polynomial equation of degree at most two in $b$. Its real roots determine candidate pairs $(a_s^\ast,b_s^\ast)$. Thus, we obtain at most two candidates if $m_{e_1}=m_{e_2}$ and at most four otherwise. For each candidate, set
\[
c_s^\ast\coloneqq
-\frac{(a_s^\ast+m_{e_1}b_s^\ast-q_{e_1})^2}{4m_{e_1}}
-b_s^\ast q_{e_1}.
\]

For each candidate $s$, we check whether
\[
(a_s^\ast,b_s^\ast)\in\bigcap_{i=1}^3\operatorname{int}(D_{e_i})
\quad\text{and}\quad
\eta_{e_1}(a_s^\ast,b_s^\ast)
\leq\eta_g(a_s^\ast,b_s^\ast)
\quad\text{for every }g\in G(\polygon)\setminus J.
\]
If both conditions hold, its associated KKT region is
\[
\Gamma_J^{(s)}
=\polygon\cap\operatorname{conv}
\{\mathbf{x}_{e_1}(a_s^\ast,b_s^\ast),
  \mathbf{x}_{e_2}(a_s^\ast,b_s^\ast),
  \mathbf{x}_{e_3}(a_s^\ast,b_s^\ast)\},
\]
and
\[
\conv_{f,\polygon}(x,y)=a_s^\ast x+b_s^\ast y+c_s^\ast
\quad\text{for all }(x,y)\in\Gamma_J^{(s)}.
\]

\subsubsection{Generator sets with two elements}

\subsubsection*{One vertex -- one edge}
\label{sec:kkt_one_vertex_one_edge}
By \Cref{prop:exclude_same_line}, pairs in which the vertex lies
on the supporting line of the edge can be excluded. Assume
therefore that $J=\{v,e\}$, where $v$ does not lie on the
supporting line $y=m_e x+q_e$ of the positive-slope edge $e$.

Following \autocite[Appendix~A.2]{locatelli2018convex}, we first
derive a candidate envelope function for a given point $(x,y)$
from the KKT equations. Using the contact point
$\mathbf{x}_e(a,b)=(x_e(a,b),y_e(a,b))$ on $e$, where
$y_e(a,b)=m_e x_e(a,b)+q_e$, the convex combination equations are
\[
(x,y)=(1-\lambda)v+\lambda\mathbf{x}_e(a,b),
\qquad 0\leq\lambda\leq 1.
\]
For $(x,y)\neq v$, any such representation has $\lambda>0$.
Writing $h(x,y)\coloneqq m_e(x-v_x)+v_y-y$, we obtain
the multiplier and the $x$-coordinate of the candidate contact point as
\[
\lambda(x,y)=\frac{h(x,y)}{v_y-q_e-m_e v_x},
\qquad
x_e(x,y)=\frac{x(v_y-q_e)-v_x(y-q_e)}{h(x,y)},
\]
where the notation $x_e(x,y)$ emphasizes that the contact coordinate
is now expressed directly in terms of $(x,y)$ rather than $(a,b)$. 
The denominator of $\lambda(x, y)$ is nonzero because $v$ does not
lie on the supporting line of $e$; consequently, $h(x,y)$ is
also nonzero whenever $\lambda(x, y)>0$.
For a contact point in the relative interior of $e$, we have
$a = 2m_e x_e(a,b)+q_e - m_e b$. Combining this relation with
$\eta_e(a,b)=\eta_v(a,b)$ gives the candidate coefficients
\begin{align*}
b^\ast(x,y)
&=\frac{m_e\bigl(x_e(x,y)\bigr)^2
-2m_e v_x x_e(x,y)-q_e v_x+v_xv_y}
{v_y-m_e v_x-q_e},\\
a^\ast(x,y)
&=2m_e x_e(x,y)+q_e-m_e b^\ast(x,y),\\
c^\ast(x,y)
&=v_xv_y-a^\ast(x,y)v_x-b^\ast(x,y)v_y.
\end{align*}
The resulting candidate envelope function is
\[
F_J(x,y)\coloneqq a^\ast(x,y)x+b^\ast(x,y)y+c^\ast(x,y).
\]
Unlike in the three-generator cases, the coefficients of the
optimal affine underestimator generally depend on $(x,y)$.

It remains to determine where these candidates satisfy the
KKT inequalities. To compute the feasible parameter intervals,
we follow the reduction to univariate polynomial inequalities
described in \autocite[Section~5.2]{locatelli2016_polyhedral}.
As in the cases of two vertices and one edge
or one vertex and two edges, introduce $z\coloneqq a+m_e b$.
For $(a,b)\in\operatorname{int}(D_e)$, we obtain
\[
\eta_e(a,b)=\eta_v(a,b)
\quad\Longleftrightarrow\quad
b=\beta_2z^2+\beta_1z+\beta_0,
\]
where the coefficients are given by
\eqref{eq:vertex_edge_beta_coefficients} for the pair $(v,e)$.
Writing $b(z)=\beta_2z^2+\beta_1z+\beta_0$ and
$a(z)=z-m_e b(z)$, we determine the parameter values satisfying
\[
(a(z),b(z))\in\operatorname{int}(D_e)
\quad\text{and}\quad
\eta_v(a(z),b(z))
\leq\eta_g(a(z),b(z))
\quad\text{for every }g\in G(\polygon)\setminus J.
\]
For each edge $e'\in E'(\polygon)\setminus J$, we distinguish
the three domains $D_{e'}^-$, $D_{e'}$, and $D_{e'}^+$ on which
$\eta_{e'}$ has a fixed expression. For every combination
of these domains, substituting $a(z)$ and $b(z)$ yields
a system of univariate polynomial inequalities of degree
at most four in $z$. Solving these systems gives the feasible
parameter set as a finite union of intervals $I_s$.
For each interval $I_s$, the associated KKT region
is characterized by
\[
\Gamma_J^{(s)}
=\left\{(x,y)\in\polygon:
0<\lambda(x,y)\leq 1,\quad
z^\ast(x,y)\in I_s\right\}.
\]
Indeed, the condition on $\lambda$ places $(x,y)$ on the segment
between $v$ and the contact point with the edge, while the condition on
$z^\ast$ ensures feasibility of the corresponding underestimator.
The $\lambda$-inequalities as well as the interval bounds on $z^\ast$ can be converted into linear
inequalities in $(x,y)$, as follows:
Set
\[
d\coloneqq v_y-q_e-m_e v_x,
\qquad
N(x,y)\coloneqq
2m_e\bigl(x(v_y-q_e)-v_x(y-q_e)\bigr)+q_e h(x,y),
\]
so that
\[
\lambda(x,y)=h(x,y)/d,
\qquad
z^\ast(x,y)=a^\ast(x,y)+m_e b^\ast(x,y)=2m_e x_e(x,y)+q_e=N(x,y)/h(x,y).
\]
For a closed interval $I_s=[L_s,U_s]$, this gives
\[
\Gamma_J^{(s)}=
\begin{cases}
\left\{(x,y)\in\polygon:
\begin{aligned}
&0<h(x,y)\leq d,\\
&L_s h(x,y)\leq N(x,y)\leq U_s h(x,y)
\end{aligned}\right\}, & d>0,\\[1ex]
\left\{(x,y)\in\polygon:
\begin{aligned}
&d\leq h(x,y)<0,\\
&U_s h(x,y)\leq N(x,y)\leq L_s h(x,y)
\end{aligned}\right\}, & d<0.
\end{cases}
\]
If $I_s$ is an open or half-open interval, the inequalities
corresponding to its excluded endpoints become strict.
All displayed constraints are linear in $(x,y)$.
Thus, if at least one feasible parameter interval exists, then
\[
\Gamma_J=\{v\}\cup\bigcup_s\Gamma_J^{(s)}, \text{ }
\conv_{f,\polygon}(x,y)=F_J(x,y)
\text{ for all }(x,y)\in\Gamma_J\setminus\{v\}, \text{ and } \conv_{f,\polygon}(v)=v_x v_y.
\]

\subsubsection*{Two edges}
\label{sec:kkt_two_edges}
By \Cref{prop:exclude_same_line}, pairs of edges on the same
supporting line can be excluded. Assume therefore that
$J=\{e_1,e_2\}$ consists of positive-slope edges with distinct
supporting lines.

As in the three-edge case, for
$(a,b)\in\operatorname{int}(D_{e_1})\cap\operatorname{int}(D_{e_2})$,
the equality $\eta_{e_1}(a,b)=\eta_{e_2}(a,b)$ reduces to one
or two linear relations given by
\eqref{eq:parallel_edge_relation} and \eqref{eq:nonparallel_edge_relations}.
Index these relations by $\tau$ and write
$a_\tau(b)=\beta_{1,\tau}b+\beta_{0,\tau}$.
For ease of notation, we fix one relation and suppress
the index $\tau$ in the following derivation.
For parallel edges,
$\beta_1=m_{e_1}=m_{e_2}$ and
$\beta_0=(q_{e_1}+q_{e_2})/2$; otherwise,
$\beta_1=\pm\sqrt{m_{e_1}m_{e_2}}$ and $\beta_0$ is given by
the corresponding relation in \eqref{eq:nonparallel_edge_relations}.

Following \autocite[Appendix~A.3]{locatelli2018convex}, we first derive a candidate envelope function for a given point $(x, y)$ from the KKT equations. The convex combination
equations are
\[
(x,y)=(1-\lambda)\mathbf{x}_{e_1}(a,b)
+\lambda\mathbf{x}_{e_2}(a,b),
\qquad 0\leq\lambda\leq1.
\]
Differentiating $\eta_{e_1}(a(b),b)=\eta_{e_2}(a(b),b)$
with respect to $b$ yields
\begin{equation}\label{eq:contact_point_relation}
y_{e_1}(a,b)+\beta_1 x_{e_1}(a,b)
=y_{e_2}(a,b)+\beta_1 x_{e_2}(a,b).
\end{equation}
Since $(x,y)$ is a convex combination of the two contact
points, it satisfies
$y+\beta_1x=y_{e_i}(a,b)+\beta_1x_{e_i}(a,b)$ for $i=1,2$.
With $y_{e_i}(x,y)=m_{e_i}x_{e_i}(x,y)+q_{e_i}$ we obtain
\[
x_{e_i}(x,y)=\frac{y+\beta_1x-q_{e_i}}{m_{e_i}+\beta_1},
\qquad i=1,2.
\]
The denominators are nonzero for all relations considered above.
Using $a(b)+m_{e_1}b=2m_{e_1}x_{e_1}(x, y)+q_{e_1}$, we have
\begin{align*}
b^\ast(x,y)
&=\frac{2m_{e_1}x_{e_1}(x,y)+q_{e_1}-\beta_0}
{m_{e_1}+\beta_1},\\
a^\ast(x,y)&=\beta_1b^\ast(x,y)+\beta_0,\\
c^\ast(x,y)&=-m_{e_1}\bigl(x_{e_1}(x,y)\bigr)^2
-q_{e_1}b^\ast(x,y).
\end{align*}
Here, the expression for $c^\ast$ follows from
$c^\ast=\eta_{e_1}(a^\ast,b^\ast)$ and
\autocite[equation~(22)]{locatelli2018convex}.
Thus, the candidate envelope function for relation $\tau$ is
$F_J^{(\tau)}(x,y)=a^\ast(x,y)x+b^\ast(x,y)y+c^\ast(x,y)$.

To determine the associated KKT region, we follow
\autocite[Section~5.1]{locatelli2016_polyhedral}.
Substituting $a(b)=\beta_1b+\beta_0$ into
\[
(a(b),b)\in\operatorname{int}(D_{e_1})
\cap\operatorname{int}(D_{e_2})
\quad\text{and}\quad
\eta_{e_1}(a(b),b)\leq\eta_g(a(b),b)
\quad\text{for every }g\in G(\polygon)\setminus J
\]
and distinguishing the domains of all remaining edges as in
the preceding case yields univariate polynomial inequalities
of degree at most two in $b$. Solving these systems gives
feasible parameter intervals $I_{\tau,s}$, where $s$ indexes the
intervals obtained for the fixed relation $\tau$.

Writing $\Delta(x,y)\coloneqq x_{e_2}(x,y)-x_{e_1}(x,y)$, the convex
combination multiplier is
\[
\lambda(x,y)=\frac{x-x_{e_1}(x,y)}{\Delta(x,y)}.
\]
We have $\Delta(x,y)\neq 0$, since otherwise
\eqref{eq:contact_point_relation} would imply that the two
contact points coincide. This is impossible because they
lie in the relative interiors of distinct edges of a
simple polygon. The region associated with $I_{\tau,s}$ is
\[
\Gamma_J^{(\tau,s)}
=\left\{(x,y)\in\polygon:
0\leq\lambda(x,y)\leq1,\quad b^\ast(x,y)\in I_{\tau,s}\right\}.
\]
For a closed interval $I_{\tau,s}=[L_{\tau,s},U_{\tau,s}]$,
distinguishing the sign of $\Delta(x,y)$ gives
\[
\begin{aligned}
\Gamma_J^{(\tau,s)}
={}&\left\{(x,y)\in\polygon:
\begin{aligned}
&\Delta(x,y)>0,\\
&0\leq x-x_{e_1}(x,y)\leq\Delta(x,y),\\
&L_{\tau,s}\leq b^\ast(x,y)\leq U_{\tau,s}
\end{aligned}\right\}\\[1ex]
&\cup\left\{(x,y)\in\polygon:
\begin{aligned}
&\Delta(x,y)<0,\\
&\Delta(x,y)\leq x-x_{e_1}(x,y)\leq0,\\
&L_{\tau,s}\leq b^\ast(x,y)\leq U_{\tau,s}
\end{aligned}\right\}.
\end{aligned}
\]
If $I_{\tau,s}$ is an open or half-open interval, the inequalities
corresponding to its excluded endpoints become strict.
All displayed constraints are linear in $(x,y)$.
Finally, for all $\tau$ we have
\[
\Gamma_J^{(\tau)}=\bigcup_s\Gamma_J^{(\tau,s)},
\qquad
\conv_{f,\polygon}(x,y)=F_J^{(\tau)}(x,y)
\quad\text{for all }(x,y)\in\Gamma_J^{(\tau)}.
\]
and we can set $\Gamma_J = \bigcup_\tau \Gamma_J^{(\tau)}$. 

\subsection{Overall algorithm}

\Cref{alg:kkt_polygons} combines the preceding cases into an overall
procedure. For each remaining generator set $J$, the corresponding
case determines the feasible KKT regions and the associated
expressions for the convex envelope. Each returned pair $(\Gamma,F)$ satisfies $\conv_{f,\polygon}=F$ on $\Gamma$.
Values at points of $\polygon$ not covered by the returned regions
are recovered by continuous extension. The three subroutines are given in \Cref{app:kkt_algorithm}.

\begin{algorithm}[H]
\caption{KKT-based computation of the convex envelope of $f(x,y)=x \cdot y$ over polygons}
\label{alg:kkt_polygons}
\small
\algrenewcommand{\algorithmicrequire}{\textbf{Input:}}
\algrenewcommand{\algorithmicensure}{\textbf{Output:}}
\begin{algorithmic}[1]
\Require A simple polygon $\polygon$
\Ensure A collection $\mathcal{C}$ of KKT regions and associated envelope functions
\State $G(\polygon)\gets V(\polygon)\cup E'(\polygon)$
\State $\mathcal{J}\gets\{J\subseteq G(\polygon):2\leq |J|\leq3\}$
\State Discard the sets excluded by
\Statex \hspace{\algorithmicindent}\Cref{prop:kkt_index_reductions,prop:positive_slope_contacts,prop:exclude_same_line}
\State $\mathcal{C}\gets\varnothing$
\For{each $J\in\mathcal{J}$}
    \If{$|J|=3$}
        \State $\mathcal{C}\gets\mathcal{C}\cup\Call{ThreeGenerators}{\polygon,G(\polygon),J}$ \Comment{\Cref{alg:kkt_three_generators}}
    \ElsIf{$J=\{v,e\}$}
        \State $\mathcal{C}\gets\mathcal{C}\cup\Call{VertexEdge}{\polygon,G(\polygon),J}$ \Comment{\Cref{alg:kkt_vertex_edge}}
    \Else \Comment{$J=\{e_1,e_2\}$}
        \State $\mathcal{C}\gets\mathcal{C}\cup\Call{TwoEdges}{\polygon,G(\polygon),J}$ \Comment{\Cref{alg:kkt_two_edges}}
    \EndIf
\EndFor
\State \Return $\{(\Gamma,F)\in\mathcal{C}:\Gamma\neq\varnothing\}$
\end{algorithmic}
\end{algorithm}

\subsection{Polyhedral subdivision}

We now establish the piecewise structure of the convex envelope,
using the notion of a polyhedral subdivision also for possibly
nonconvex polygons, as defined below.

\begin{definition}[{cf. \autocite[Definition~1.1]{locatelli2016_polyhedral}}]
\label{def:polyhedral_subdivision}
A \textit{polyhedral subdivision} of a polygon $\polygon$ is a finite
collection $\{P_i:i\in I\}$ such that
\begin{enumerate}[label=(\roman*)]
    \item each $P_i$ is a polytope;
    \item $\bigcup_{i\in I}P_i=\polygon$;
    \item for $i\neq j$, the intersection $P_i\cap P_j$ is a
    possibly empty face of both $P_i$ and $P_j$.
\end{enumerate}
The subdivision does not add vertices if $V(P_i)\subseteq V(\polygon)$
for every $i\in I$.
\end{definition}

\noindent The preceding KKT cases yield the following extension of
\autocite[Theorem~1.1]{locatelli2016_polyhedral} to polygons:

\begin{theorem}\label{thm:polyhedral_subdiv}
Let $\polygon$ be a simple polygon.
The convex envelope of $f(x,y)=x\cdot y$ over $\polygon$ admits a finite
polyhedral subdivision $\{P_i:i\in I\}$, possibly adding vertices.
On each cell $P_i$, it is affine, quadratic, or a quotient of a
quadratic polynomial and a nonconstant affine function.
Here, quadratic means of degree exactly two.
At boundary points where a denominator vanishes, the corresponding
expression is interpreted by continuous extension.
\end{theorem}
\begin{proof}
Throughout the proof, we consider only generator sets that are
not ruled out by the preceding exclusion conditions.
We first establish that \Cref{alg:kkt_polygons} produces only finitely many regions.
There are finitely many generator sets $J$. Each three-generator case
produces finitely many candidates: 
Three noncollinear vertices uniquely determine the affine function
interpolating $f$ at these vertices. For two vertices and one edge, substituting $b=B(z)=\beta_2z^2+\beta_1z+\beta_0$ and
$a=z-m_eB(z)$ into $\eta_{v_1}(a, b)=\eta_{v_2}(a, b)$ gives
\[
\Delta_x z+(\Delta_y-m_e\Delta_x)B(z)=f(v_1)-f(v_2),
\]
where $\Delta_x=v_{1,x}-v_{2,x}$ and
$\Delta_y=v_{1,y}-v_{2,y}$.
Since $B$ has degree two, all polynomial coefficients could vanish
only if $\Delta_x=\Delta_y=0$, contradicting $v_1\neq v_2$. Thus, there are at most two candidates. For one vertex and two edges, substituting
$b=B(z)=\beta_2z^2+\beta_1z+\beta_0$ and
$a=z-m_{e_1}B(z)$ into
$\eta_{e_1}(a,b)=\eta_{e_2}(a,b)$ gives a polynomial equation
of degree at most four. Since $\beta_2\neq0$, its quartic coefficient
$\frac{(m_{e_2}-m_{e_1})^2\beta_2^2}{4m_{e_2}}$ is nonzero
if the slopes differ. If $m_{e_1}=m_{e_2}$, the distinct supporting
lines have $q_{e_1}\neq q_{e_2}$, and the equality reduces to
\[
z-2m_{e_1}B(z)=\frac{q_{e_1}+q_{e_2}}{2}.
\]
Its quadratic coefficient is $-2m_{e_1}\beta_2\neq0$.
Thus, there are at most four candidates. For three edges, relabel the edges so that, if two are parallel, they are $e_1,e_2$. Each linear relation $a=\beta_1b+\beta_0$
obtained from $\eta_{e_1}(a, b)=\eta_{e_2}(a, b)$ satisfies
$\beta_1^2=m_{e_1}m_{e_2}$.
Substituting it into $\eta_{e_2}(a,b)=\eta_{e_3}(a,b)$ gives
a polynomial equation of degree at most two, with quadratic coefficient
\[
\frac{(m_{e_3}-m_{e_1})(m_{e_3}-m_{e_2})}{4m_{e_3}}
\]
when written as $\eta_{e_2}(a, b)-\eta_{e_3}(a, b)=0$.
This coefficient is nonzero unless all three edges are parallel.
In that case, writing $m$ for their common slope, the two equalities
require both
\[
a=mb+\frac{q_{e_1}+q_{e_2}}{2},
\qquad
a=mb+\frac{q_{e_2}+q_{e_3}}{2},
\]
which is impossible since $q_{e_1}\neq q_{e_3}$.
Thus, there are at most four candidates. In the two-generator cases, each of finitely many domain combinations
yields a univariate polynomial inequality system whose solution set
is a finite union of intervals, including singletons.
These candidate intervals therefore generate finitely many regions.

We next examine the functional forms of the convex envelope on these regions.
In the three-generator cases, each candidate $(a^\ast,b^\ast,c^\ast)$ gives an affine function $a^\ast x+b^\ast y+c^\ast$ on its corresponding region.
For $J=\{e_1,e_2\}$, $a^\ast,b^\ast$
are affine in $(x,y)$ and $c^\ast$ is quadratic, so the resulting
function has degree at most two. Its homogeneous quadratic part is
\[
\frac{m_{e_1}}{(m_{e_1}+\beta_1)^2}(y+\beta_1x)^2.
\]
The coefficient of $y^2$ is strictly positive since $m_{e_1}>0$
and $m_{e_1}+\beta_1\neq0$. Hence, the function has degree exactly two.
For $J=\{v,e\}$, use $h,d$ from
\Cref{sec:kkt_one_vertex_one_edge} and set
$A(x,y)\coloneqq x(v_y-q_e)-v_x(y-q_e)$.
Since $\lambda=h/d$ and $x_e=A/h$, contact interpolation gives
\[
F_J=(1-\lambda)f(v)+\lambda f(\mathbf{x}_e)
=\frac{f(v)dh-f(v)h^2+m_eA^2+q_eAh}{dh}.
\]
The denominator $dh$ is a nonconstant affine function since
$d\neq0$ and $m_e>0$. The homogeneous quadratic part of the
numerator is $d(m_e v_y x^2-q_e xy-v_x y^2)$.
Since $m_e>0$ and $d\neq0$, its coefficients cannot all vanish:
this would force $v_x=v_y=q_e=0$ and hence $d=0$.
Thus, the numerator has degree exactly two.

Every region is described by finitely many linear inequalities
intersected with $\polygon$.
Refine a triangulation of $\polygon$ by all boundary lines of these
regions which results in a finite polyhedral subdivision of $\polygon$.
At every point in the interior of $\polygon$, the convex envelope
admits a supporting affine function, so an optimal affine
underestimator exists.
By the KKT representation and Carath\'eodory's theorem, the cases therefore
cover the interior of $\polygon$ except possibly a finite union of
points and segments associated with the discarded lower-dimensional
generator sets. Consequently, each cell interior is covered by a
KKT region and carries one of the stated functions. On overlaps,
these functions agree because each equals the convex envelope.

Finally, the convex envelope is closed as the pointwise supremum
of affine underestimators;
see \autocite[Theorem~9.4]{rockafellar1970convex}.
Since $\operatorname{conv}(\polygon)$ is a locally simplicial
subset of its effective domain, the convex envelope is continuous
there by \autocite[Theorem~10.2]{rockafellar1970convex}.
Thus, the cellwise formulas extend continuously to the cell
closures, covering all of $\polygon$.
\end{proof}

\subsection{Pair-based enumeration as an alternative}

An alternative method is the pair-based enumeration approach of
\autocite{locatelli2016_polyhedral}.
At an interior point $\mathbf{x}\in\polygon$, every optimal solution
of \eqref{eq:convex_program} has at least two active constraints
by the argument of \autocite[Lemma~3.1]{locatelli2016_polyhedral},
which also applies to nonconvex polygons.
One therefore enumerates pairs $h,w\in G(\polygon)$,
$h\neq w$, together with all combinations of the domains
$D_e^-,D_e^+,D_e$ for $e\in E'(\polygon)$.
Discarding pair--domain combinations on which $\eta_h=\eta_w$
holds identically, analogous variable substitutions and a decomposition
into feasible parameter intervals yield finitely many univariate
optimization problems.
The maximum of these feasible subproblem values gives the envelope. 
The subproblem value functions have the desired functional forms
on their respective polyhedral regions. However, a separate argument
is needed to show that their pointwise maximum admits a finite
polyhedral subdivision on which these forms are preserved.
Note that this construction explicitly enumerates up to
$3^{|E'(\polygon)|}$ domain combinations for each pair.

\section{Special Case: Orthogonal Polygons}
\label{sec:orthogonal_polygons}
In the following, we focus on the special case of orthogonal polygons:

\begin{definition}[Orthogonal polygon]
A polygon is called \textit{orthogonal} if all its edges are
horizontal or vertical.
\end{definition}

\noindent Vertices with interior angles of $90^\circ$ are called \textit{convex vertices} ($V^{\mathrm{conv}}(\polygon)$), while those
with interior angles of $270^\circ$ are called \textit{reflex vertices} ($V^{\mathrm{refl}}(\polygon)$). Vertices that are also vertices of $\operatorname{conv}(\polygon)$
are called \textit{outer vertices},
and the remaining vertices are \textit{inner vertices}. The advantages of this setting are twofold. First, the computation of such polygons as two-dimensional feasible regions arising from high-dimensional global optimization problems can be performed within a straightforward preprocessing step using optimization-based bound tightening techniques, which will be described later. Second, the computation of the convex envelope based on these precomputed polygons simplifies considerably, which will be shown in the following.

\subsection{Reduced collection of generators}
We first reduce $G(P)$ by identifying generators that can be discarded without changing the convex envelope. To this end, we introduce two relevant notions. Denote $h^{\boldsymbol{\alpha}}(x,y)\coloneqq x\cdot y-ax-by$ in the following.

\begin{definition}
    Let $\boldsymbol{\alpha} \coloneqq (a,b) \in \mathbb{R}^2$ be fixed. Furthermore, consider $\mathbf{x}_0, \mathbf{x}_1, \mathbf{x}_2 \in \polygon$. We then say that 
    \begin{itemize}
        \item $\mathbf{x}_0 \in \polygon$ is replaceable by $\mathbf{x}_1, \mathbf{x}_2 \in \polygon$ if the open line segment $(\mathbf{x}_1, \mathbf{x}_2)$ contains $\mathbf{x}_0$ and is either vertical or of nonpositive slope.
        \item $\mathbf{x}_0 \in \polygon$ is $\boldsymbol{\alpha}-$dominated by $\mathbf{x}_1\in \polygon$ if $h^{\boldsymbol{\alpha}}(\mathbf{x}_0) \ge h^{\boldsymbol{\alpha}}(\mathbf{x}_1)$.
    \end{itemize}
\end{definition}

We have the following relation between the latter properties: 
\begin{proposition}\label{prop:repl_dom}
    Let $\mathbf{x}_0, \mathbf{x}_1, \mathbf{x}_2 \in \polygon, \boldsymbol{\alpha} \in \mathbb{R}^2$. If $\mathbf{x}_0$ is replaceable by $\mathbf{x}_1, \mathbf{x}_2$ then $\mathbf{x}_0$ is $\boldsymbol{\alpha}-$dominated by $\mathbf{x}_1$ or $\mathbf{x}_2$ for all $\boldsymbol{\alpha}\in\R^2$. 
\end{proposition}
\begin{proof}
Due to the concavity of $x \cdot y$ along $(\mathbf{x_1}, \mathbf{x_2})$ we have for some $\lambda \in (0, 1)$ that \\
$h^{\boldsymbol{\alpha}}(\mathbf{x}_0) \ge \lambda h^{\boldsymbol{\alpha}}(\mathbf{x}_1) + (1 - \lambda) h^{\boldsymbol{\alpha}}(\mathbf{x}_2) \ge \min(h^{\boldsymbol{\alpha}}(\mathbf{x}_1), h^{\boldsymbol{\alpha}}(\mathbf{x}_2))$.
\end{proof}

We first use this domination property to show that only convex vertices are needed as generators.

\begin{lemma}\label{lemma:reduction}
    Let $\orth \subseteq \mathbb{R}^2$ be an orthogonal polygon. Then we have for all $\boldsymbol{\alpha} \in \mathbb{R}^2$ that 
    \begin{equation*}
        \min_{\mathbf{y} \in \orth} h^{\boldsymbol{\alpha}}(\mathbf{y}) = \min_{\mathbf{y} \in V^{\mathrm{conv}}(\orth)} h^{\boldsymbol{\alpha}}(\mathbf{y}).
    \end{equation*}
\end{lemma}
\begin{proof}
    We need to show that each point of the orthogonal polygon that is not a convex vertex is dominated by a convex vertex.
    To this end, let $p_0 \coloneqq (x_0, y_0)$ be, without loss of generality, either an interior point, a reflex vertex or a point in the relative interior of a horizontal edge of $\orth$. (The case for points in the relative interiors of vertical edges follows analogously.) Consider the following procedure: 

    \begin{algorithmic}[1]
        \State $(x, y) \gets (x_0, y_0)$
        \State direction $\gets$ horizontal
        \Repeat 
            \If{direction $=$ horizontal}:
                \State $p' \gets (\min_{(x', y) \in \orth}x', y)$, $p'' \gets (\max_{(x', y) \in \orth}x', y)$
                \State direction $\gets$ vertical
            \Else
                \State $p' \gets (x, \min_{(x, y') \in \orth}y')$, $p'' \gets (x, \max_{(x, y') \in \orth}y')$ 
                \State direction $\gets$ horizontal
            \EndIf
            \State $(x, y) \gets \argmin_{p \in \{p', p''\}}h^{\boldsymbol{\alpha}}(p)$
        \Until{$(x, y)$ is a convex vertex}
    \end{algorithmic}

    We claim that this procedure terminates after a finite number of iterations $n \in \mathbb{N}$ and produces a sequence of points $p_1, \dots, p_n \in \orth$ such that 
    \begin{equation*}
        h^{\boldsymbol{\alpha}}(p_0) \ge h^{\boldsymbol{\alpha}}(p_1) \ge \cdots \ge h^{\boldsymbol{\alpha}}(p_n)
    \end{equation*}
    and $p_n$ is a convex vertex. To prove the inequalities let $p$ denote a point produced in an arbitrary iteration in line $10$, and let $\tilde{p}$ denote the point in the subsequent iteration which is either $p'$ or $p''$. 
    One can easily see that, in the horizontal case, each of $p'$ and $p''$ either lies in the relative interior of a vertical edge or coincides with a convex vertex of $\orth$. Analogously, in the vertical case, each of them either lies in the relative interior of a horizontal edge or coincides with a convex vertex. Hence, unless $\tilde{p}$ is a convex vertex which triggers a termination, the procedure considers always distinct points $p', p''$ that are also distinct from $p$ and $p$ is replaceable by $p', p''$. With \Cref{prop:repl_dom} we then have that $h^{\boldsymbol{\alpha}}(p) \ge h^{\boldsymbol{\alpha}}(\tilde{p})$. 
    It remains to show that the procedure terminates after finitely many iterations. By construction, the final point $p_n$ must then be a convex vertex, since this is the only termination criterion. We prove termination by contradiction. Assume that the procedure generates an infinite sequence $(p_i)_{i \in \mathbb{N}_0}$ satisfying $h^{\boldsymbol{\alpha}}(p_i) \ge h^{\boldsymbol{\alpha}}(p_{i+1})$ for all $i \in \mathbb{N}_0$. Since the coordinates of the generated points are drawn from the start
    coordinates and the finitely many edge coordinates of $\orth$, only
    finitely many points can occur. Thus, an infinite sequence must revisit
    a point. Such a return requires at least four nontrivial steps, since
    the directions alternate between horizontal and vertical. By
    monotonicity, all steps between the two visits preserve the value of
    $h^{\boldsymbol{\alpha}}$. For a nontrivial horizontal step from $(x,y)$ to $(x',y)$, equality gives
    \[
    0=h^{\boldsymbol{\alpha}}(x',y)-h^{\boldsymbol{\alpha}}(x,y)
    =(x'-x)(y-a),
    \]
    hence $y=a$. Analogously, a nontrivial vertical step preserving
    $h^{\boldsymbol{\alpha}}$ requires $x=b$.
    The return path contains two horizontal steps separated by a single
    vertical step, or two vertical steps separated by a single horizontal
    step. In the first case, both horizontal steps must lie at height $a$,
    contradicting the nontriviality of the intervening vertical step.
    The second case yields an analogous contradiction. Therefore, no point can be revisited, and the procedure
    must terminate. By constructing an
    analogous procedure starting with vertical instead of horizontal direction one can see that also
    points from the relative interiors of vertical edges are dominated by convex vertices.
\end{proof}

The reduction can be strengthened by also discarding replaceable convex vertices.

\begin{lemma}\label{lemma:reduction-enh}
    Let $\orth\subseteq\mathbb{R}^2$ be an orthogonal polygon and let
    $R(\orth)\coloneqq\{\mathbf{x}_0\in\orth:\exists\,\mathbf{x}_1,\mathbf{x}_2\in\orth \text{ such that }
    \mathbf{x}_0 \text{ is replaceable by } \mathbf{x}_1,\mathbf{x}_2\}$. Then we have for all
    $\boldsymbol{\alpha}\in\mathbb{R}^2$ that
    \begin{equation*}
        \min_{\mathbf{y}\in\orth}h^{\boldsymbol{\alpha}}(\mathbf{y})
        =\min_{\mathbf{y}\in V^{\mathrm{conv}}(\orth)\setminus R(\orth)}h^{\boldsymbol{\alpha}}(\mathbf{y}),
    \end{equation*}
    i.e., replaceable (in particular replaceable inner) convex vertices can be discarded in
    \Cref{lemma:reduction}. Moreover, the right-hand set is nonempty, since every outer convex vertex,
    being an extreme point of $\operatorname{conv}(\orth)$, is not replaceable.
\end{lemma}
\begin{proof}
We can adapt the procedure from the proof of \Cref{lemma:reduction} as follows. Whenever we consider an
inner convex vertex $p$ of $\orth$, we first check whether it is replaceable. If this is not the case, we
terminate the algorithm.
Otherwise, we distinguish two subcases: either there exist $p',p'' \in \orth$ such that
$p \in (p',p'')$ (the open line segment between $p'$ and $p''$)
and $p',p''$ have the same $y$- or the same $x$-coordinate, or no such pair exists.
In the first case, we proceed exactly as in the proof of \Cref{lemma:reduction}, initializing the
direction as horizontal or vertical, respectively. If no such horizontal or vertical line segment
exists, then there must exist $p',p'' \in \orth$ such that $p \in (p',p'')$ and
$(p',p'')$ lies on a line of negative slope.
In this case, we even obtain strict domination, that is,
$h^{\boldsymbol{\alpha}}(p) > \min(h^{\boldsymbol{\alpha}}(p'), h^{\boldsymbol{\alpha}}(p''))$,
since $x\cdot y$ is strictly concave along a segment of negative slope, and, as
before, we choose the new point as $\argmin_{q \in \{p',p''\}} h^{\boldsymbol{\alpha}}(q)$. Due to the
strict domination, no cycle containing such a strict descent can occur.
Moreover, only finitely many strict descents can occur: such a step is performed only
at convex vertices, i.e., at the points at which the walk of \Cref{lemma:reduction} stops, and since
$h^{\boldsymbol{\alpha}}$ never increases, a convex vertex from which a strict descent was made is
never visited again; as $V^{\mathrm{conv}}(\orth)$ is finite, after finitely many such steps only
horizontal and vertical steps remain, for which termination follows as in
\Cref{lemma:reduction}.
Apart from this, the argument proceeds analogously to the proof of \Cref{lemma:reduction}. Using this
modified procedure, we finally obtain that every point in $\polygon$ is dominated either by an outer
convex vertex or by an inner convex vertex that is not replaceable.
\end{proof}

%\noindent \Cref{fig:domination-procedure} illustrates the procedure from
%\Cref{lemma:reduction} and its extension in \Cref{lemma:reduction-enh}.

%\input{figures/replacement_algo}

\noindent These reductions yield a characterization of the convex envelope using only the remaining vertex generators.

\begin{theorem}\label{thm:reduction}
    Let $\orth \subseteq \mathbb{R}^2$ be an orthogonal polygon. Then the convex envelope of the bilinear function $f(x, y) = x \cdot y$ over $\orth$ at some point $(x^*, y^*) \in \orth$ can be computed by solving the following problem
    \begin{equation}\label{eq:convex_program_ortho}
    	\begin{aligned}
    			\conv_{f,\orth}(x^*, y^*) = \max_{a, b, c} \quad & ax^* + by^* + c  \\
    			\text{s.t.} \quad &  x\cdot y - ax - by \ge c, \quad \forall (x, y) \in G_{\perp}(\orth),
    	\end{aligned}
    \end{equation}
    with the collection of orthogonal generators $G_{\perp}(\orth) \coloneqq V^{\mathrm{conv}}(\orth)\setminus R(\orth)$
    that are all outer convex vertices and inner convex vertices that are not replaceable. 
\end{theorem}

\subsection{Exclusion of generator sets}

Having reduced the collection of generators, we now turn to excluding generator sets $J$ of the remaining generators. For this purpose, we classify convex vertices by their orientation.

\begin{definition}
\label{def:north-vertices}
    Let $\epsilon, \delta > 0$. We call a convex vertex $(\bar{x},\bar{y})\in \polygon$ a vertex of type
    \begin{itemize}
        \item north-west (NW), if
  $\exists\, \epsilon > 0 \;\colon\;
   \forall\, \delta,\ 0 < \delta \le \epsilon \;\colon\;
   (\bar{x}+\delta,\, \bar{y}-\delta) \in \polygon$,
        \item north-east (NE), if
  $\exists\, \epsilon > 0 \;\colon\;
   \forall\, \delta,\ 0 < \delta \le \epsilon \;\colon\;
   (\bar{x}-\delta,\, \bar{y}-\delta) \in \polygon$,
        \item south-west (SW), if
  $\exists\, \epsilon > 0 \;\colon\;
   \forall\, \delta,\ 0 < \delta \le \epsilon \;\colon\;
   (\bar{x}+\delta,\, \bar{y}+\delta) \in \polygon$,
        \item south-east (SE), if
  $\exists\, \epsilon > 0 \;\colon\;
   \forall\, \delta,\ 0 < \delta \le \epsilon \;\colon\;
   (\bar{x}-\delta,\, \bar{y}+\delta) \in \polygon$.
    \end{itemize}
\end{definition}

\begin{figure}[htbp]
  \centering
\begin{tikzpicture}[scale=1.35,
    interior/.style = {fill=color1!15, draw=color1!80, line width=1.1pt,
                       line join=round},
    vtx/.style      = {circle, fill=color1!80, draw=white, line width=.5pt,
                       inner sep=1.7pt},
    reflex/.style   = {circle, draw=color1!80, fill=white, line width=.9pt,
                       inner sep=1.7pt},
    pt/.style       = {cross out, draw=color1!80, line width=.9pt,
                       inner sep=1.6pt},
    guide/.style    = {color1!80, dashed, dash pattern=on 2pt off 1.4pt,
                       -{Stealth[length=3.5pt]}, line width=.7pt},
    leader/.style   = {color1!50, line width=.4pt},
    lbl/.style      = {color1!80, font=\bfseries\footnotesize},
    coord/.style    = {color1!80, font=\scriptsize, inner sep=1pt}
  ]

  % ---- the orthogonal polygon (L-shape) ----
  \fill[interior]
    (0,0) -- (4,0) -- (4,1.5) -- (1.5,1.5) -- (1.5,4) -- (0,4) -- cycle;
  \node[color1!60!black, font=\large\itshape] at (0.75,2.7) {$\orth$};

  % ---- reflex vertex (interior angle 270 deg), not classified ----
  \node[reflex] (R) at (1.5,1.5) {};
  \node[coord, anchor=west] at (1.72,1.80) {reflex};

  % ================= NW vertex : interior to SE =================
  \node[vtx] at (0,4) {};
  \node[lbl, anchor=south] at (0,4.08) {\nwvertex};
  \draw[guide] (0,4) -- (0.24,3.76);
  \node[pt] (pNW) at (0.28,3.72) {};
  \node[coord, anchor=north] at (0.65,3.62) {$(\bar{x}+\delta,\;\bar{y}-\delta)$};

  % ================= NE vertex : interior to SW =================
  \node[vtx] at (4,1.5) {};
  \node[lbl, anchor=south west] at (4.06,1.54) {\nevertex};
  \draw[guide] (4,1.5) -- (3.76,1.26);
  \node[pt] (pNE) at (3.72,1.22) {};
  \node[coord, anchor=east] at (3.62,1.22) {$(\bar{x}-\delta,\;\bar{y}-\delta)$};

  % ================= SE vertex : interior to NW =================
  \node[vtx] at (4,0) {};
  \node[lbl, anchor=north] at (4,-0.10) {\severtex};
  \draw[guide] (4,0) -- (3.76,0.24);
  \node[pt] (pSW) at (3.72,0.28) {};
  \node[coord, anchor=east] at (3.62,0.28) {$(\bar{x}-\delta,\;\bar{y}+\delta)$};

  % ================= SW vertex : interior to NE =================
  \node[vtx] at (0,0) {};
  \node[lbl, anchor=east] at (-0.10,0) {\swvertex};
  \draw[guide] (0,0) -- (0.24,0.24);
  \node[pt] (pSE) at (0.28,0.28) {};
  \node[coord, anchor=west] at (0.38,0.28) {$(\bar{x}+\delta,\;\bar{y}+\delta)$};

\end{tikzpicture}
  \caption{Orthogonal polygon $\orth$ with vertices classified according to
    \Cref{def:north-vertices}}
  \label{fig:direction_vertices}
\end{figure}
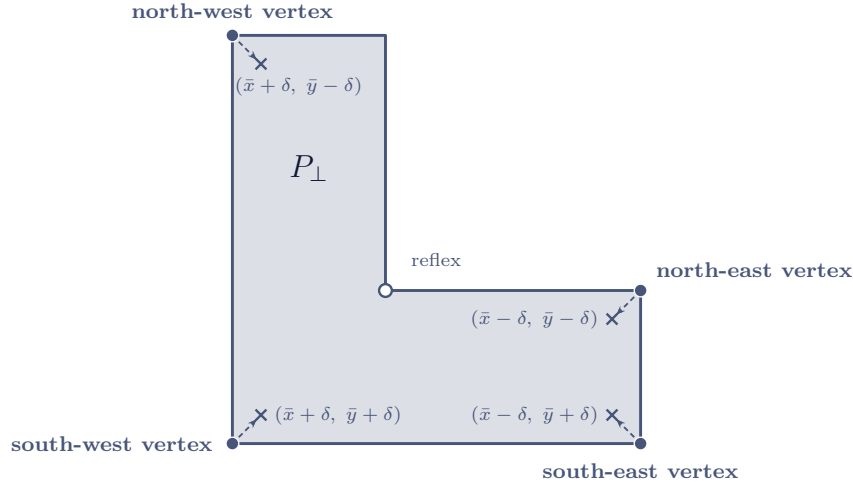

\noindent This classification, depicted in \Cref{fig:direction_vertices}, allows us to rule out certain vertex triples as candidate sets.

\begin{proposition}\label{prop:triple-exclusion}
    Let $\orth \subseteq \mathbb{R}^2$ be an orthogonal polygon and let
    $J \subseteq G_{\perp}(\orth)$ with $|J| = 3$ be a candidate generator set
    consisting of three vertices. Using the hyperbola geometry illustrated in
    \Cref{fig:hyperbola-cases}, together with the classification in
    \Cref{def:north-vertices}, we can exclude certain triples a priori.
    Indeed, the following triples need not be considered when constructing two-dimensional KKT regions:
    \begin{enumerate}[label=(\roman*)]
        \item triples with one of the following type combinations, up to permutation:
        \[
        \begin{array}{ll}
            (\mathrm{NW},\mathrm{NW},\mathrm{NE}), & (\mathrm{NW},\mathrm{NW},\mathrm{SW}), \\
            (\mathrm{NE},\mathrm{SE},\mathrm{SE}), & (\mathrm{SW},\mathrm{SE},\mathrm{SE}), \\
            (\mathrm{NW},\mathrm{NE},\mathrm{SW}), & (\mathrm{NE},\mathrm{NE},\mathrm{SW}), \\
            (\mathrm{NE},\mathrm{SW},\mathrm{SW}), & (\mathrm{NE},\mathrm{SW},\mathrm{SE}), \\
            (\mathrm{NW},\mathrm{NW},\mathrm{NW}), & (\mathrm{SE},\mathrm{SE},\mathrm{SE});
        \end{array}
        \]
        \item triples containing two \nwvertices{} or two \severtices{} whose connecting line
              has negative slope or is horizontal or vertical;
        \item triples containing two \nevertices{} or two \swvertices{} whose connecting line
              has positive slope or is horizontal or vertical.
    \end{enumerate}
\end{proposition}
\begin{proof}
    Consider an affine underestimator $\ell_{a,b,c}$. Its contact points satisfy
    $(x-b)(y-a)=\kappa$, whereas $\orth$ lies in the superlevel set
    $(x-b)(y-a)\geq\kappa$ (cf. \Cref{fig:hyperbola-cases}). The incident horizontal and vertical edges determine
    which vertex types can occur on each branch.
    If $\kappa>0$, contacts can only be \nevertices{} on the lower-left branch or
    \swvertices{} on the upper-right branch. Since $\orth$ is connected, only one
    of these two decreasing branches can contain contacts.
    If $\kappa<0$, contacts can only be \nwvertices{} on the upper-left branch or
    \severtices{} on the lower-right branch; both branches are increasing.
    If $\kappa=0$, simplicity of $\orth$ implies that it lies in one of the two
    feasible quadrants. Thus, in this case \nevertices{} and \swvertices{} cannot both be contacts,
    and contacts of type \nwvertex{} or \severtex{} lie on one boundary ray per type.

    Moreover, two vertices of the same type on a horizontal or vertical line cannot
    both belong to $G_{\perp}(\orth)$: one lies strictly between the other and a point
    on one of its own incident edges and is therefore replaceable.
    Consequently, the first eight combinations in (i) cannot consist entirely of
    contact points: four contain both a \nevertex{} and a \swvertex{}, while the
    other four require $\kappa=0$ and two contacts of the same type on one boundary ray.
    The two remaining combinations are excluded for a different reason. For $\kappa<0$,
    their three contacts lie on a single strictly convex or strictly concave hyperbola
    branch, and their convex hull meets the feasible set only at these contacts.
    For $\kappa=0$, they are collinear. Hence, neither combination yields a
    two-dimensional KKT region.

    For (ii), contacts of the same type lie on an increasing branch when $\kappa<0$
    and on a single boundary ray when $\kappa=0$; $\kappa>0$ is impossible.
    For (iii), they lie on a decreasing branch when $\kappa>0$; when $\kappa=0$,
    contacts on distinct boundary rays have negative connecting slope;
    $\kappa<0$ is impossible. In both cases, horizontal and vertical pairs are
    excluded by replaceability.
\end{proof}

The remaining ten vertex-type combinations cannot be excluded by this geometric argument alone, since each can be realized as a configuration of contact points of an affine underestimator on a suitable orthogonal polygon in one of the cases illustrated in \Cref{fig:hyperbola-cases}.

\subsection{Illustrative example: convex versus nonconvex polygonal domains}
The following example illustrates the benefit of computing convex envelopes directly over nonconvex polygonal domains rather than first replacing these domains by their convex hulls. 
To this end, consider the parametrized orthogonal polygon
$P_{\perp}(p_x,p_y)$, with $p_x,p_y\in(0,2)$, given by the
following ordered vertices:
\begin{equation*}
	 (v_1 \coloneqq (0, 0), v_2 \coloneqq (2, 0), v_3 \coloneqq (2, 2), v_4 \coloneqq (p_x, 2), \bar{v} \coloneqq (p_x, p_y), v_5 \coloneqq (0, p_y)). 
\end{equation*}
In the following, we compare the convex envelope of $f(x, y) = x \cdot y$ over $P_{\perp}(p_x, p_y)$ with the convex envelope of the same function over $\conv(P_{\perp}(p_x, p_y))$.
\Cref{fig:ortho_and_conv_hull} shows the \staircasepolygon and its convex hull for $p_x=1$ and $p_y=1$, where we define $\Delta^{\mathrm{low}} \coloneqq \conv(v_1, v_2, v_5), \Delta^{\mathrm{mid}} \coloneqq \conv(v_2, v_4, v_5)$, and $\Delta^{\mathrm{up}} \coloneqq \conv(v_2, v_3, v_4)$.
\begin{figure}[ht]
    \centering
    \begin{subfigure}{0.45\textwidth}
        \centering
        \def\px{1}
        \def\py{1}
        \begin{tikzpicture}[scale=1.7]
            \fill[pattern=north east lines, pattern color=color3!60]
            (0,0)
            -- (2,0)
            -- (2,2)
            -- (\px,2)
            -- (\px,\py)
            -- (0,\py)
            -- cycle;

            % Labels der Teilgebiete
            \node at (0.65,0.35) {$\Delta^{\mathrm{low}}$};
            \node at (1.3,0.7) {$\Delta^{\mathrm{mid}}$};
            \node at (1.65,1.35) {$\Delta^{\mathrm{up}}$};

            \draw[thick]
            (0,0) node[below left] {$v_1$}
            -- (2,0) node[below right] {$v_2$}
            -- (2,2) node[above right] {$v_3$}
            -- (\px,2) node[above] {$v_4$}
            -- (\px,\py) node[above left] {$\bar{v}$}
            -- (0,\py) node[left] {$v_5$}
            -- cycle;
            
            % Trennlinien
            \draw[thick,dashed]
            (0,\py) -- (2,0);
            \draw[thick,dashed]
            (\px,2) -- (2,0);
        \end{tikzpicture}
        \caption{The \staircasepolygon $P_{\perp}(1, 1)$}
    \end{subfigure}
    \hfill
    \begin{subfigure}{0.45\textwidth}
        \centering
        \def\px{1}
        \def\py{1}
        \begin{tikzpicture}[scale=1.7]
            \fill[pattern=north east lines, pattern color=color3!60]
            (0,0)
            -- (2,0)
            -- (2,2)
            -- (\px,2)
            -- (0,\py)
            -- cycle;

            % Labels der Teilgebiete
            \node at (0.65,0.35) {$\Delta^{\mathrm{low}}$};
            \node at (1.00,1.00) {$\Delta^{\mathrm{mid}}$};
            \node at (1.65,1.35) {$\Delta^{\mathrm{up}}$};

            \draw[thick]
            (0,0) node[below left] {$v_1$}
            -- (2,0) node[below right] {$v_2$}
            -- (2,2) node[above right] {$v_3$}
            -- (\px,2) node[above] {$v_4$};

            \draw[thick]
            (0,\py) node[left] {$v_5$}
            -- (0,0);

            \draw[thick]
            (0,\py) -- node[midway, above left] {$e$} (\px,2);

            % Trennlinien
            \draw[thick,dashed]
            (0,\py) -- (2,0);
            \draw[thick,dashed]
            (\px,2) -- (2,0);
        \end{tikzpicture}
        \caption{The convex hull of $P_{\perp}(1, 1)$}
    \end{subfigure}
    \caption{Comparison of polygonal and polyhedral domains}
    \label{fig:ortho_and_conv_hull}
\end{figure}
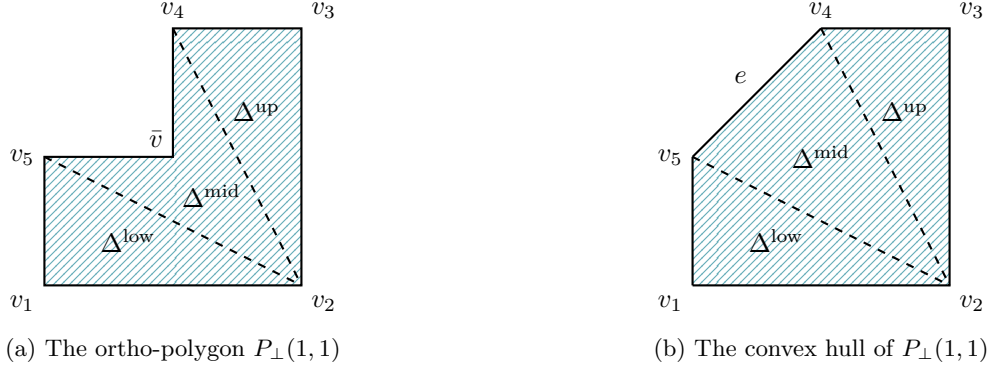

For the computation of the convex envelope over $\orth(p_x,p_y)$,
\Cref{thm:reduction} allows us to restrict attention to
$G_{\perp}(\orth(p_x,p_y))=\{v_1,v_2,v_3,v_4,v_5\}$,
since $\bar{v}$ is a reflex vertex and can therefore be excluded. 
Of the ten possible vertex triples, seven can be excluded by
\Cref{prop:positive_slope_contacts} and
\Cref{prop:triple-exclusion}~(i), leaving only three candidates,
namely $\{v_1,v_2,v_5\}$, $\{v_2,v_3,v_4\}$, and $\{v_2,v_4,v_5\}$:

\begin{enumerate}
    \item[(1)] For $J=\{v_1,v_2,v_5\}$, the corresponding active
    constraints in \eqref{eq:convex_program_ortho} yield
    $a=b=c=0$ and hence the affine function
    $f^{\mathrm{low}}(x,y)=0$.

    \item[(2)] For $J=\{v_2,v_4,v_5\}$, we analogously obtain
    \[
    f^{\mathrm{mid}}(x,y)
    =\frac{2p_xp_y}{(p_x-2)p_y+4}x
    +\frac{4p_x}{(p_x-2)p_y+4}y
    -\frac{4p_xp_y}{(p_x-2)p_y+4}.
    \]

    \item[(3)] Finally, for $J=\{v_2,v_3,v_4\}$, we obtain
    $f^{\mathrm{up}}(x,y)=2x+2y-4$.
\end{enumerate}
All three affine functions satisfy the remaining constraints
in \eqref{eq:convex_program_ortho} and therefore coincide with
the convex envelope on their respective KKT regions.
Thus, the convex envelope of $f(x,y)=x\cdot y$ over $\orth(p_x,p_y)$ is given by 
\begin{equation*}
    \conv_{f,\orth(p_x, p_y)}(x, y) = 
    \begin{cases}
        f^{\mathrm{low}}(x, y), & (x, y) \in \Delta^{\mathrm{low}}, \\
        f^{\mathrm{mid}}(x, y), & (x, y) \in \Delta^{\mathrm{mid}} \cap \orth(p_x, p_y), \\
        f^{\mathrm{up}}(x, y), & (x, y) \in \Delta^{\mathrm{up}}. 
    \end{cases}
\end{equation*}

Next, we compute the convex envelope over
$\widehat P\coloneqq\conv(\orth(p_x,p_y))$.
Its generating set is $G(\widehat P)=\{v_1,v_2,v_3,v_4,v_5,e\}$,
where $e=[v_4,v_5]$ is the only positive-slope edge, with
\[
m_e=\frac{2-p_y}{p_x},\qquad q_e=p_y.
\]
By \Cref{prop:kkt_index_reductions}, singleton generator sets and
pairs of vertices can be discarded. Since $\widehat P$ is a polytope,
triples containing an edge can also be excluded; see
\autocite[Observation~4.2]{locatelli2018convex}.
Among the vertex triples, \Cref{prop:positive_slope_contacts}
excludes all those containing one of the pairs
$\{v_1,v_3\}$, $\{v_1,v_4\}$, $\{v_3,v_5\}$, or $\{v_4,v_5\}$,
leaving only $\{v_1,v_2,v_5\}$ and $\{v_2,v_3,v_4\}$.
For vertex--edge pairs, \Cref{prop:positive_slope_contacts} excludes $\{v_1,e\}$ and
$\{v_3,e\}$, while \Cref{prop:exclude_same_line} excludes
$\{v_4,e\}$ and $\{v_5,e\}$, leaving only $\{v_2,e\}$.
In total, only three cases remain:

\begin{enumerate}
    \item[(1)] For $J=\{v_1,v_2,v_5\}$, we again obtain the
    affine function $f^{\mathrm{low}}(x,y)=0$, which coincides
    with the convex envelope on $\Delta^{\mathrm{low}}$.

    \item[(2)] For $J=\{v_2,e\}$, we follow
    \autocite[Appendix~A.2]{locatelli2018convex}.
    Since $e$ is an upper edge and
    $\widehat P\subseteq\{(x,y):x\leq2,\ y\geq0\}$,
    Observation~A.2 therein reduces the KKT-region constraints to the convex
    combination condition and the requirement that the contact
    point lies on $e$. Geometrically, these conditions describe all segments joining
    $v_2$ to points of $e$, whose union is $\Delta^{\mathrm{mid}}$.
    For $(x,y)\neq v_2$, equation~(27) in the reference gives the
    contact coordinate
    \[
    x_e(x,y)=\frac{p_yx+2y-2p_y}{y+m_e(2-x)}.
    \]
    The corresponding optimal affine underestimator agrees with $f$
    at $v_2$ and $\mathbf{x}_e$. Hence, using $f(v_2)=0$ and
    $y=\lambda(x,y)y_e(x,y)$, we obtain
    \[
    \widehat f^{\mathrm{mid}}(x,y)
    = \widehat f^{\mathrm{mid}}(\lambda \mathbf{x}_e + (1 - \lambda)v_2)
    = \lambda f(\mathbf{x}_e) + (1-\lambda)f(v_2)
    =y\,x_e(x,y)
    =\frac{y(p_yx+2y-2p_y)}{y+m_e(2-x)}.
    \]
    The denominator is nonzero on
    $\Delta^{\mathrm{mid}}\setminus\{v_2\}$, and we set
    $\widehat f^{\mathrm{mid}}(v_2)=0$ by continuous extension.

    \item[(3)] For $J=\{v_2,v_3,v_4\}$, we again obtain the
    affine function $f^{\mathrm{up}}(x,y)=2x+2y-4$, which
    coincides with the convex envelope on $\Delta^{\mathrm{up}}$.
\end{enumerate}
Thus, the convex envelope of $f(x,y)=x\cdot y$ over
$\widehat P=\conv(\orth(p_x,p_y))$ is given by
\[
\conv_{f,\widehat P}(x,y)=
\begin{cases}
    f^{\mathrm{low}}(x,y), & (x,y)\in\Delta^{\mathrm{low}},\\
    \widehat f^{\mathrm{mid}}(x,y), & (x,y)\in\Delta^{\mathrm{mid}},\\
    f^{\mathrm{up}}(x,y), & (x,y)\in\Delta^{\mathrm{up}}.
\end{cases}
\]
To quantify the improvement obtained by retaining the nonconvex domain,
we integrate the pointwise difference
$\conv_{f,\orth(p_x,p_y)}-\conv_{f,\widehat P}$
over the original polygon. Since $\orth(p_x,p_y)\subseteq\widehat P$, this difference is
nonnegative. Moreover, the envelopes agree on $\Delta^{\mathrm{low}}$
and $\Delta^{\mathrm{up}}$, so the enclosed volume is
\[
\begin{aligned}
V(p_x,p_y)
&\coloneqq\int_{\orth(p_x,p_y)}
\bigl(\conv_{f,\orth(p_x,p_y)}(x,y)-\conv_{f,\widehat P}(x,y)\bigr)
\,dx\,dy\\
&=\int_{\Delta^{\mathrm{mid}}\cap\orth(p_x,p_y)}
\bigl(f^{\mathrm{mid}}(x,y)-\widehat f^{\mathrm{mid}}(x,y)\bigr)
\,dx\,dy.
\end{aligned}
\]
Parametrize $\Delta^{\mathrm{mid}}$ by
\[
(x,y)=(1-\lambda)v_2+\lambda(t,m_et+p_y),
\qquad 0\leq t\leq p_x,\quad 0\leq\lambda\leq1.
\]
The absolute value of the Jacobian determinant of this parametrization
is $\lambda(2m_e+p_y)$.\\ 
Since $f^{\mathrm{mid}}(t,m_et+p_y)=2t$, the envelope difference
at the parametrized point equals $\lambda m_et(p_x-t)$.
Membership in $\orth(p_x,p_y)$ requires $x\geq p_x$ or $y\leq p_y$,
which is equivalent to $\lambda\leq\rho(t)$, where
\[
\rho(t)\coloneqq\max\left\{
\frac{2-p_x}{2-t},\frac{p_y}{m_et+p_y}\right\}.
\]
Integrating first with respect to $\lambda$ therefore gives
\[
\begin{aligned}
V(p_x,p_y)
&=\int_0^{p_x}\int_0^{\rho(t)}
\lambda^2 m_e(2m_e+p_y)t(p_x-t)\,d\lambda\,dt\\
&=\frac{m_e(2m_e+p_y)}{3}
\int_0^{p_x}t(p_x-t)\rho(t)^3\,dt>0.
\end{aligned}
\]
Thus, for every $p_x,p_y\in(0,2)$, convexifying the domain first
strictly weakens the resulting convex envelope on a subset of
the original polygon of positive area contained in
$\Delta^{\mathrm{mid}}\cap\orth(p_x,p_y)$.
For the example $p_x=p_y=1$ in \Cref{fig:ortho_and_conv_hull},
the two expressions defining $\rho$ coincide at $t=1/2$.
By symmetry, we obtain
\[
V(1,1)=2\int_0^{1/2}\frac{t(1-t)}{(1+t)^3}\,dt
=\frac{8}{9}-2\ln\!\left(\frac32\right)
\approx0.07796.
\]
The improvement is also visible directly at the reflex vertex $(1,1)$:
\[
\conv_{f,\orth(1,1)}(1,1)=\frac23,
\qquad
\conv_{f,\widehat P}(1,1)=\frac12.
\]
This example shows that replacing a nonconvex polygon by its convex hull
can lose relaxation strength at points in the original domain.
Computing the convex envelope directly over the polygon avoids this loss.

\subsection{Using Orthogonal Polygons on Nonlinear Programs}
In the present section we discuss how orthogonal polygons arise naturally in the presolving and relaxation of nonlinear mixed-integer programs,
and how they can be exploited to strengthen the convexification of bilinear terms.
Throughout, we consider a mixed-integer nonlinear optimization problem with feasible set
\[
F \subseteq \mathbb{R}^{n} \times \mathbb{Z}^{m}
\]
and focus on one bilinear term $z = x y$ with continuous variables $x$ and $y$. Let $(\ell_x,u_x)$ and $(\ell_y,u_y)$ denote the available variable bounds and let
\[
B := [\ell_x,u_x] \times [\ell_y,u_y]
\]
be the associated bounding box in the $(x,y)$–plane.
The standard McCormick relaxation for $z = x y$ is based on convexifying the graph of the bilinear term over $B$.
Our goal is to replace $B$ by a smaller set $\orth \subseteq B$, which is an orthogonal polygon containing the projection of $F$ onto the $(x,y)$–space,
and then to use our developed theory to work with the convex hull of the graph of $z = x y$ over $\orth$.

\subsubsection{Ortho-Polygons from Optimization-Based Bound Tightening}

Optimization-based bound tightening (OBBT) is a well-established presolving technique in global MINLP solvers, 
see, e.g., \cite{belotti2013mixed,bestuzheva2025-minlp-scip8}.
For our purposes, OBBT provides a systematic way of constructing orthogonal polygons that outer-approximate the projection of $F$ on the $(x,y)$–plane.

Let $X := \operatorname{proj}_{x,y}(F) \subseteq B$ denote the projection of the feasible set onto $(x,y)$. We discretize the interval $[\ell_x,u_x]$ into stripes
\[
\ell_x = \alpha_0 < \alpha_1 < \dots < \alpha_K = u_x,
\]
and, for each stripe $I_k := [\alpha_{k-1},\alpha_k]$, we solve the two optimization problems
\[
y_k^{\min} := \inf\{\, y \mid (x,y,w) \in F,\; x \in I_k \,\},\qquad
y_k^{\max} := \sup\{\, y \mid (x,y,w) \in F,\; x \in I_k \,\}.
\]
Each pair $(y_k^{\min},y_k^{\max})$ defines a vertical strip
\[
R_k := I_k \times [y_k^{\min},y_k^{\max}],
\]
which clearly satisfies $X \cap (I_k \times \mathbb{R}) \subseteq R_k$. Thus the union $\bigcup_{k=1}^K R_k$ is an outer approximation of $X$.
Repeating the same construction in the $y$–direction, using a partition
\[
\ell_y = \beta_0 < \beta_1 < \dots < \beta_L = u_y
\]
and solving analogous optimization problems for $x$ at fixed $y$–stripes, yields vertical strips
\[
S_\ell := [x_\ell^{\min},x_\ell^{\max}] \times J_\ell, \qquad J_\ell := [\beta_{\ell-1},\beta_\ell],
\]
whose union $\bigcup_{\ell=1}^L S_\ell$ again contains $X$. The intersection
\[
\orth := \Bigl(\bigcup_{k=1}^K R_k\Bigr) \cap \Bigl(\bigcup_{\ell=1}^L {S}_{\ell}\Bigr)
\]
is still an outer approximation of $X$ and defines an orthogonal polygon.

The set $\orth$ is then used as the domain on which the convex hull of the graph of $z = x y$ is constructed.

Figure~\ref{fig:stair-obbt} illustrates this construction.
The light gray rectangle is the bounding box $B$,
the dashed lines divide the $x$- and $y$–stripes, the curved line is the true projection $X$,
and the blue orthogonal polygon is the outer approximation $\orth$ obtained from the stripe-wise OBBT bounds.

\begin{figure}[htbp]
  \centering
  \begin{tikzpicture}[scale=1.5]
    % --- 1. Axes ---
    \draw[->, thick] (0,0) -- (4.8,0) node[below] {$x$};
    \draw[->, thick] (0,0) -- (0,3.5) node[left] {$y$};

    % --- 2. Grid (Horizontal & Vertical Dashed Lines) ---
    % Grid step size is 0.5 visual units
    \foreach \x in {0.5, 1.0, ..., 4.0} {
      \draw[gray!30, dashed, thin] (\x, 0) -- (\x, 3.2);
    }
    \foreach \y in {0.5, 1.0, ..., 3.0} {
      \draw[gray!30, dashed, thin] (0, \y) -- (4.5, \y);
    }

    % --- 3. Bounding Box B ---
    % A large rectangle enclosing the whole area of interest
    \draw[gray!80, thick] (0.5,0.5) rectangle (4.5,3.0);
    \node[gray!90, anchor=south east] at (4.5,3.0) {$B$};

    % --- 4. Staircase Outer Approximation S ---
    % Constructed as a union of grid cells that encompass the nonconvex shape X.
    % Notice the "hole" at the bottom (x=2.0 to 3.0) reflecting the nonconvexity.
    \fill[color1!10]
      (1.1, 0.68) -- (1.5, 0.68) -- % Left base
      (1.5, 0.68) -- (1.5, 0.7) -- % Left step up
      (1.5, 0.7) -- (2.0, 0.7) -- % Left step up
      (2.0, 0.7) -- (2.0, 1.1) -- % Left Arch (Nonconvex step up)
      (2.0, 1.1) -- (3.0, 1.1) -- % The Arch (Nonconvex step up)
      (3.0, 1.1) -- (3.0, 0.7) -- % Right step down
      (3.0, 0.7) -- (3.5, 0.7) -- % Bottom Right base
      (3.5, 0.7) -- (3.5, 0.68) -- % Right step down
      (3.5, 0.68) -- (3.9, 0.68) -- % Right base
      (3.9, 0.68) -- (3.9, 1.0) -- % Right side up
      (3.9, 1.0) -- (4.0, 1.0) -- % Right side up
      (4.0, 1.0) -- (4.0, 2.5) -- % Right side up
      (4.0, 2.5) -- (3.7, 2.5) -- % Right shoulder
      (3.7, 2.5) -- (3.7, 2.6) -- % Right Peak
      (3.7, 2.6) -- (3.5, 2.6) -- % Top Peak (wide)
      (3.5, 2.6) -- (3.5, 2.77) -- % Top Peak (wide)
      (3.5, 2.77) -- (3.0, 2.77) -- % Top Peak (wide)
      (3.0, 2.77) -- (3.0, 2.8) -- % Top Peak (wide)
      (3.0, 2.8) -- (2.0, 2.8) -- % Top Peak (wide)
        (2.0, 2.8) -- (2.0, 2.77) -- % Left Peak (wide)
        (2.0, 2.77) -- (1.5, 2.77) -- % Left Peak (wide)
        (1.5, 2.77) -- (1.5, 2.6) -- % Left Peak (wide)
        (1.5, 2.6) -- (1.3, 2.6) -- % Left shoulder
        (1.3, 2.6) -- (1.3, 2.5) -- % Left side down
        (1.3, 2.5) -- (1.0, 2.5) -- % Left side down
      (1.0, 2.5) -- (1.0, 1.0) -- % Left side down
        (1.0, 1.0) -- (1.1, 1.0) -- % Left side down
      cycle;

    % Outline of S
    \draw[color1, thick]
      (1.1, 0.68) -- (1.5, 0.68) -- % Left base
      (1.5, 0.68) -- (1.5, 0.7) -- % Left step up
      (1.5, 0.7) -- (2.0, 0.7) -- % Left step up
      (2.0, 0.7) -- (2.0, 1.1) -- % Left Arch (Nonconvex step up)
      (2.0, 1.1) -- (3.0, 1.1) -- % The Arch (Nonconvex step up)
      (3.0, 1.1) -- (3.0, 0.7) -- % Right step down
      (3.0, 0.7) -- (3.5, 0.7) -- % Bottom Right base
      (3.5, 0.7) -- (3.5, 0.68) -- % Right step down
      (3.5, 0.68) -- (3.9, 0.68) -- % Right base
      (3.9, 0.68) -- (3.9, 1.0) -- % Right side up
      (3.9, 1.0) -- (4.0, 1.0) -- % Right side up
      (4.0, 1.0) -- (4.0, 2.5) -- % Right side up
      (4.0, 2.5) -- (3.7, 2.5) -- % Right shoulder
      (3.7, 2.5) -- (3.7, 2.6) -- % Right Peak
      (3.7, 2.6) -- (3.5, 2.6) -- % Top Peak (wide)
      (3.5, 2.6) -- (3.5, 2.77) -- % Top Peak (wide)
      (3.5, 2.77) -- (3.0, 2.77) -- % Top Peak (wide)
      (3.0, 2.77) -- (3.0, 2.8) -- % Top Peak (wide)
      (3.0, 2.8) -- (2.0, 2.8) -- % Top Peak (wide)
        (2.0, 2.8) -- (2.0, 2.77) -- % Left Peak (wide)
        (2.0, 2.77) -- (1.5, 2.77) -- % Left Peak (wide)
        (1.5, 2.77) -- (1.5, 2.6) -- % Left Peak (wide)
        (1.5, 2.6) -- (1.3, 2.6) -- % Left shoulder
        (1.3, 2.6) -- (1.3, 2.5) -- % Left side down
        (1.3, 2.5) -- (1.0, 2.5) -- % Left side down
      (1.0, 2.5) -- (1.0, 1.0) -- % Left side down
        (1.0, 1.0) -- (1.1, 1.0) -- % Left side down
      cycle;

    \node[color1] at (1.2, 2.8) {$\orth$};

    % --- 5. Nonlinear Nonconvex Area X ---
    % An "Arch" or "Bridge" shape
    \filldraw[fill=color1!40, draw=color1!80, thick, fill opacity=0.7]
      plot[smooth cycle, tension=0.6] coordinates {
        (1.3, 0.7)  % Bottom Left
        (1.1, 2.3)  % Top Left
        (2.5, 2.8)  % Top Peak
        (3.9, 2.3)  % Top Right
        (3.7, 0.7)  % Bottom Right
        (2.5, 1.4)  % Inner Arch Point (Creates nonconvexity)
      };

    \node[white] at (2.5, 2.2) {\textbf{$X$}};

  \end{tikzpicture}
  \caption{Orthogonal polygon $\orth$ approximating a nonconvex nonlinear region $X$. Note the tighter bound at the bottom $(x \in [2,3])$ capturing the concavity.}
  \label{fig:stair-obbt}
\end{figure}
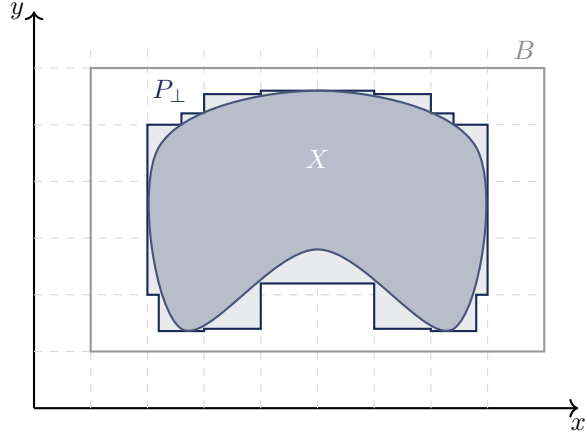

\subsubsection{Orthogonal polygons from norm-induced cuts}

Norm-induced cuts, as introduced in~\cite{goess2026norm},
provide a second mechanism for creating orthogonal polygons.
The basic idea is to exploit local information about infeasibility in the $(x,y)$–plane.
Suppose that we have a function $f : \mathbb{R}^2 \rightarrow \mathbb{R}$ that is
Lipschitz continuous with respect to some norm~$\|\cdot\|$ on $\mathbb{R}^2$
and maps from the $(x, y)$-plane into the reals.
If $f(x, y) \leq 0$ marks a constraint of the optimization problem
and~$(\bar{x},\bar{y}) \in B$ is an infeasible point in this regard, i.e.,
$f(\bar{x}, \bar{y}) > 0$, a norm-induced cut excludes the open ball
\[
\mathcal{B}(\bar{x},\bar{y};r) := \bigl\{ (x,y) \in \mathbb{R}^2 \,\bigm|\,
\|(x,y) - (\bar{x},\bar{y})\| < r \bigr\}
\]
from the box~$B$ without affecting feasible projection points.
That is, it holds true that
\[
X \cap \mathcal{B}(\bar{x},\bar{y};r) = \emptyset,
\]
%Such a ball can be obtained, for instance, by solving a distance-minimization problem that computes the minimal norm-distance from $(\bar{x},\bar{y})$ to $X$.
%The classical theory of intersection cuts shows how general convex sets, including norm balls,
%can be used to derive valid inequalities for mixed-integer problems \cite{goess2026norm}.
and we are concerned with the following relaxation of~$X$:
\[
X \subseteq B \setminus \mathcal{B}(\bar{x},\bar{y};r).
\]

Let's denote the ball as $\mathcal{B}$ for short
and use~$\partial$ as the boundary operator.
In general, the boundary of $B \setminus \mathcal{B}$ is a
subset of the axis-aligned boundary of~$B$
and the curved boundary of the ball.
If $\mathcal{B} \nsubseteq \mathrm{int}(B)$,
$\partial B \cap \mathcal{B} \neq \emptyset$
and it follows that $\partial B \nsubseteq \partial (B \setminus \mathcal{B})$.
We refer to~\Cref{fig:norm-cuts} for a visualization.
In this setting, we can obtain an orthogonal polygon when
approximating the circular arc by a sequence of horizontal and vertical segments
that lies inside the ball.
More precisely, we construct a polyline with vertices $(x_i,y_i)$ such that
consecutive vertices differ in exactly one coordinate and
the entire polyline lies in $\mathcal{B}$.
Together with the edges of $B$, this yields a new orthogonal polygon $\orth$ with
\[
X \subseteq \orth \subseteq B.
\]

When several norm-induced cuts are available that correspond to several
balls~$\mathcal{B}(\bar{x}^{(q)},\bar{y}^{(q)};r_q)$, fulfilling the
setting above,
we intersect the corresponding stair\-case approximations
and take the outer boundary of the resulting set.
The resulting orthogonal polygon combines information from OBBT
and from the local norm-based exclusion regions.

Figure~\ref{fig:norm-cuts} visualizes this construction. 
The shaded circle represents a norm ball containing no feasible projections,
while the blue polyline approximates the circular boundary by horizontal
and vertical segments that stay inside the ball,
yielding an orthogonal approximation of the complement.

\begin{figure}[htbp]
  \centering
  \begin{tikzpicture}[scale=1.5]
    % --- 1. Axes ---
    \draw[->, thick] (0,0) -- (4.8,0) node[below] {$x$};
    \draw[->, thick] (0,0) -- (0,3.5) node[left] {$y$};

    % --- 4. The Ball (Infeasible Region) ---
    % Drawn behind the fill slightly or on top with transparency
    % Center (3.8, 2.4), Radius 1.15
    % This circle enters from top right and covers the "missing" steps of S

        % --- 3. Bounding Box B ---
    \draw[gray!80, thick] (0.5,0.5) rectangle (4.5,3.0);
    \node[gray!90, anchor=south east] at (0.9, 3.0) {$B$};

    % --- 5. Staircase Approximation S ---
    % S is the Box B MINUS the inner approximation of the Ball.
    % We construct S by tracing the grid lines that remain feasible.
    % It looks like the box with a "bite" taken out of the top right.
    \fill[color2!10]
      (0.5, 0.5) -- (4.5, 0.5) --
      (4.5, 1.5) -- (4.0, 1.5) -- (4.0, 1.3) -- (3.5, 1.3) --
      (3.5, 1.57) -- (3.0, 1.57) --
        (3.0, 2.0) -- (2.71, 2.0) -- (2.71, 2.5) -- (2.71, 2.5) -- (2.81, 2.5) -- (2.81, 3.0)  -- (0.5, 3.0) --
      cycle;

    % Outline of S
    \draw[color2, thick]
      (0.5, 0.5) -- (4.5, 0.5) --
      (4.5, 1.5) -- (4.0, 1.5) -- (4.0, 1.3) -- (3.5, 1.3) --
      (3.5, 1.57) -- (3.0, 1.57) --
        (3.0, 2.0) -- (2.71, 2.0) -- (2.71, 2.5) -- (2.71, 2.5) -- (2.81, 2.5) -- (2.81, 3.0)  -- (0.5, 3.0) --
      cycle;

    \node[color2] at (2.0, 2.0) {$\orth$};

    \draw[gray!70, thick, dashed] (3.8, 2.4) circle (1.15);
    \node[gray!80] at (4.0, 2.6) {$\mathcal{B}$};

  \end{tikzpicture}
  \caption{Norm-induced cut: a Euclidean ball $\mathcal{B}$ entering the box $B$ is approximated from the inside by removed grid cells, yielding a feasible orthogonal polygon $\orth$.}
  \label{fig:norm-cuts}
\end{figure}
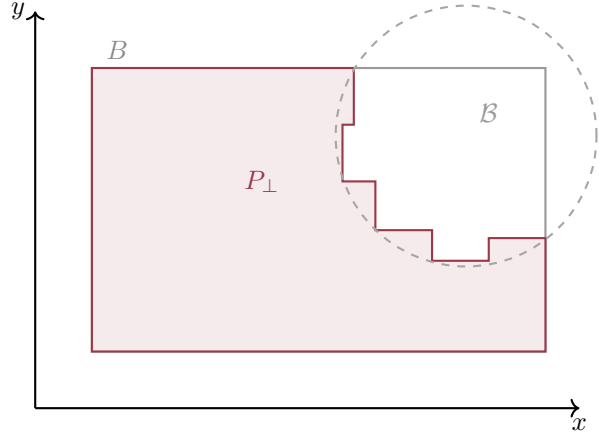

Once such an orthogonal polygon $\orth$ is available, we again take the convex hull of the graph of $z = x y$ over $\orth$.

\subsubsection{Orthogonal polygons from indicator constraints}
\label{sec:indicator}

Indicator constraints are another natural source of orthogonal polygons. They are also known in the literature as semi-continuous variables~\cite{halbig2025diving}.
In many MINLP formulations, especially in portfolio optimization~\cite{perold1984large} and unit commitment problems~\cite{chen2019unit},
certain variables are active only when an associated binary variable takes a prescribed value.
Such relationships are typically encoded by indicator constraints or by big-$M$ formulations.

Consider again a bilinear term $z = x y$ and suppose that the model implies an on/off structure of the form
\[
x = 0 \;\Longrightarrow\; y \in [\underline{y}_0,\overline{y}_0],\qquad
x \neq 0 \;\Longrightarrow\; y \in [\underline{y}_1,\overline{y}_1],
\]
with $\underline{y}_0 \leq \underline{y}_1$ and $\overline{y}_1 \leq \overline{y}_0$.
The corresponding projection $X$ of the feasible set on the $(x,y)$–plane is then contained in the union
\[
X \subseteq \bigl(\{0\} \times [\underline{y}_0,\overline{y}_0]\bigr)
   \cup \bigl([0,u_x] \times [\underline{y}_1,\overline{y}_1]\bigr).
\]
This union has the shape of a vertical line segment at $x=0$ together with a rectangle for $x \geq 0$, and its outer boundary is an orthogonal polygon.

If indicator constraints are explicitly present in the model,
this structure can be recognized by inspecting the logical conditions and the implied variable bounds for active and inactive indicator assignment.
Alternatively, it can be uncovered by a coarse OBBT procedure that considers only two stripes in each direction:
one stripe tightly centered around $x=0$ (or $y=0$) and one stripe covering the remainder of the domain.
Solving the four corresponding OBBT problems yields bounds that distinguish the free and bounded regimes of $y$ with respect to $x$, and vice versa.

Figure~\ref{fig:indicator-stair} shows a typical situation.
When $x = 0$, $y$ can take values in a wide interval (the vertical segment on the $y$–axis). When $x$ is nonzero,
$y$ is restricted to a smaller range (the rectangle on the right).
The union of these two components is an orthogonal polygon which can serve as the domain $\orth$ for convexifying $z = x y$.
Using the convex hull of the graph over this $\orth$ takes into account that large values of $y$ are incompatible with $x \neq 0$,
which cannot be captured by a simple box-based McCormick relaxation.

\begin{figure}[htbp]
  \centering
  \begin{tikzpicture}[scale=1.5]
    % --- 1. Axes ---
    \draw[->, thick] (0,0) -- (4.8,0) node[below] {$x$};
    \draw[->, thick] (0,0) -- (0,3.5) node[left] {$y$};

        % --- 3. Bounding Box B ---
    \draw[gray!80, thick] (0.0,0.0) rectangle (4.5,3.0);
    \node[gray!90, anchor=south east] at (0.9, 2.7) {$B$};

    % --- 5. Staircase Approximation S ---
    % S is the Box B MINUS the inner approximation of the Ball.
    % We construct S by tracing the grid lines that remain feasible.
    % It looks like the box with a "bite" taken out of the top right.
    \fill[color6!10]
      (0.0, 0.0) -- (0.0, 3.0) -- (0.0, 1.0) -- (4.5, 1.0) -- (4.5, 0.0) --
      cycle;

    % Outline of S
    \draw[color6, thick]
      (0.0, 0.0) -- (0.0, 3.0) -- (0.0, 1.0) -- (4.5, 1.0) -- (4.5, 0.0) --
      cycle;

    \node[color6] at (1.6, 0.5) {$\orth$};

  \end{tikzpicture}
  \caption{Orthogonal polygon arising from an indicator constraint: for $x = 0$ the variable $y$ is essentially free in a wide interval, whereas for $x \neq 0$ it is restricted to a smaller range.}
  \label{fig:indicator-stair}
\end{figure}
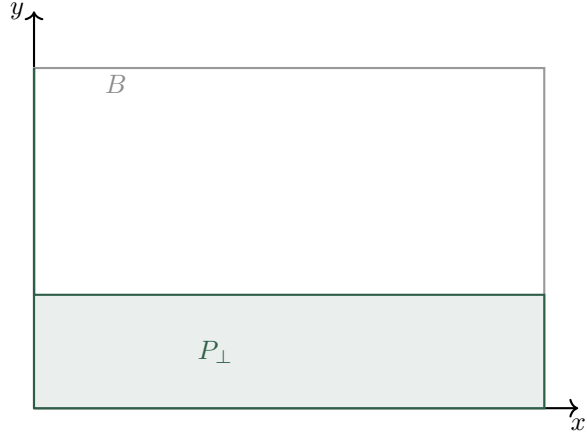

\begin{remark}
Strictly speaking, the orthogonal polygon arising from an indicator constraint is a degenerate
orthogonal polygon: the segment $\{0\} \times [\underline{y}_0,\overline{y}_0]$
has empty interior, so the boundary traversal visits the edge at $x = 0$ twice
and the polygon is not simple.
This is not a restriction for the computational study. For any $\varepsilon > 0$, replacing the
segment by the thin stripe $[0,\varepsilon] \times [\underline{y}_0,\overline{y}_0]$
yields a simple orthogonal polygon $\orth^\varepsilon \supseteq X$ to which our
theory applies, and $\orth^\varepsilon$ decreases monotonically to the
degenerate polygon $\orth^0$ as $\varepsilon \downarrow 0$. Since $(x,y) \mapsto xy$
is continuous and the sets are compact, the convex hulls of the graphs satisfy
$\operatorname{conv}\operatorname{gr}(xy;\orth^\varepsilon) \to
\operatorname{conv}\operatorname{gr}(xy;\orth^0)$ in the Hausdorff distance,
so the relaxation obtained for a sufficiently small $\varepsilon$ is arbitrarily
close to the one induced by the on/off structure. In our computational study we
therefore always work with $\orth^\varepsilon$ for a small fixed~$\varepsilon > 0$.
\end{remark}
\section{Computational Study}
\label{sec:computational_study}

In this section, we present a computational study to evaluate the practical impact of the
orthogonal polygon-based convexification techniques developed in the preceding sections.
The goal is to assess whether the tighter relaxations obtained from Ortho-Locatelli cuts
translate into measurable improvements in solver performance on instances from the MINLPLib~\cite{minlplib}.

\subsection{Implementation and Solver Framework}

All methods are implemented within the open-source solver \texttt{alpaca} (Adaptive Linear Piecewise Approximation with Combinatorial Augmentation), publicly available at
$$\texttt{https://github.com/utnopt/alpaca}.$$
The framework is written in Python~3.12.11 and utilizes Gurobi Optimizer version~13.0.0 as the underlying MIP solver.

\paragraph{Preprocessing Pipeline.}
The solution process for every instance follows a pipeline to prepare the model for cut generation. First, all multilinear terms are decomposed into chains of bilinear products. Second, we apply OBBT to all variables involved in bilinear expressions. This step reduces the initial bounding box for each bilinear term, providing a tight domain for the subsequent geometric analysis. A configurable total time budget is distributed equally across all projection subproblems: for the general grid-scanning procedure, the per-subproblem limit is the total budget divided by the product of the number of bilinear expressions, the number of grid intervals, and a factor of two (for the two coordinate directions); for the indicator procedure, the budget is split evenly between the two required subproblems per bilinear expression.

\paragraph{Domain Projection via Grid Scanning.}
For each bilinear expression $z = xy$, the feasible set is projected onto the $(x,y)$-plane to construct an orthogonal polygon. This projection is obtained via OBBT-based optimization subproblems that scan the variable domains. The algorithm proceeds in three phases, illustrated in Figure~\ref{fig:grid-scanning}.

In the first phase, boundary scanning, the domain of each variable is discretized into $N$ evenly spaced points (gray grid lines in Figure~\ref{fig:grid-scanning}), producing $N{-}1$ intervals per axis. Four sequential optimization passes are executed over the underlying model with all its nonlinear constraints. The first pass fixes $x$ to successive intervals $[x_i, x_{i+1}]$ and maximizes $y$, recording the objective bound as the piecewise-constant upper $y$-boundary over that interval. The second pass fixes $x$ to the same intervals and minimizes $y$, yielding the lower $y$-boundary. After resetting the bounds on $x$, the third pass fixes $y$ to successive intervals and maximizes $x$ to determine the right $x$-boundary, and the fourth pass minimizes $x$ over the same $y$-intervals to determine the left $x$-boundary. The boundary values are stored as condition--value pairs, where each condition is an interval membership test.

In the second phase, feasibility classification, the algorithm evaluates every point on the refined grid, which consists of the original equidistant points augmented by all boundary values discovered during scanning (dotted grid lines in Figure~\ref{fig:grid-scanning}). A grid point $(x_g, y_g)$ is classified as feasible if and only if it satisfies all four boundary constraints simultaneously: $y_g$ does not exceed the upper $y$-boundary at $x_g$, $y_g$ is not below the lower $y$-boundary at $x_g$, $x_g$ does not exceed the right $x$-boundary at $y_g$, and $x_g$ is not below the left $x$-boundary at $y_g$.

In the third phase, boundary tracing, the algorithm extracts the ordered vertices of the orthogonal polygon enclosing the feasible grid points using a left-hand rule for maze traversal. It observes all neighboring vertices and classifies them as feasible or infeasible. Based on the last direction of travel, starting left and proceeding clock-wise it finds the first feasible grid point and proceeds in this direction.

For the indicator test configuration described in Section~\ref{sec:indicator}, the same three-phase procedure is applied but with a specialized grid (see Figure~\ref{fig:indicator-polygon}). Rather than using $N$ equidistant points, the grid is constructed from only two subproblems: the lower bound of $x$ is perturbed upward by a small tolerance $\mu$, and $y$ is maximized to obtain $y^*_{ub}$; then the lower bound of $y$ is perturbed upward by $\mu$, and $x$ is maximized to obtain $x^*_{ub}$. The resulting degenerate polygon consists of the vertices $(x_{lb},y_{lb})$, $(x_{lb},y_{ub})$, $(x_{lb},y^*_{ub})$, $(x^*_{ub},y^*_{ub})$, $(x^*_{ub},y_{lb})$, $(x_{ub},y_{lb})$ connected by edges in this sequence, forming zero-width spikes along the bounding box edges.

\newcommand{\mazenode}[8][1]{
  % #1 = scale (optional, default=1)
  % #2 = x position
  % #3 = y position
  % #4 = wall North (0=open, 1=closed)
  % #5 = wall East  (0=open, 1=closed)
  % #6 = wall South (0=open, 1=closed)
  % #7 = wall West  (0=open, 1=closed)
  % #8 = heading angle (0=E, 90=N, 180=W, 270=S, or empty for no arrow)
  \begin{scope}[shift={(#2*#1,#3*#1)}, scale=#1]
    \fill[gray!10] (0,0) rectangle (1,1);
    \ifnum#4=1
      \draw[color3, line width=1pt] (0,1) -- (1,1);
    \else
      \draw[color11, line width=1pt] (0,1) -- (1,1);
    \fi
    \ifnum#5=1
      \draw[color3, line width=1pt] (1,0) -- (1,1);
    \else
      \draw[color11, line width=1pt] (1,0) -- (1,1);
    \fi
    \ifnum#6=1
      \draw[color3, line width=1pt] (0,0) -- (1,0);
    \else
      \draw[color11, line width=1pt] (0,0) -- (1,0);
    \fi
    \ifnum#7=1
      \draw[color3, line width=1pt] (0,0) -- (0,1);
    \else
      \draw[color11, line width=1pt] (0,0) -- (0,1);
    \fi
    \if\relax\detokenize{#8}\relax\else
      \draw[-stealth, color4, line width=1.5pt]
        (0.5,0.5) -- +(#8:0.35);
    \fi
  \end{scope}
}

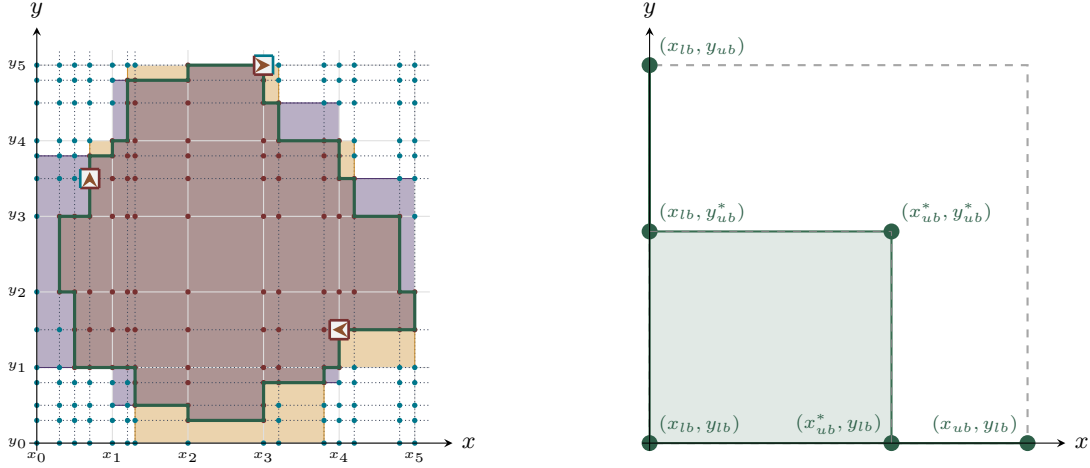
\begin{figure}[t]
\centering
\begin{subfigure}[t]{0.48\textwidth}
\centering
\begin{tikzpicture}[
    every node/.style={font=\small},
    grid line/.style={gray!30, thin},
]
    % Shade feasible region (background)
    % pattern={Lines[angle=0, distance=2pt, line width=0.4pt]}, pattern 
    \fill[color=color0, opacity=0.4]
        (1.3,0) -- (3.8,0) -- (3.8,1) -- (5.0,1) -- (5.0,2)
        -- (4.8,2) -- (4.8,3) -- (4.2,3) -- (4.2,4) -- (3.2,4)
        -- (3.2,5) -- (1.2,5) -- (1.2,4) -- (0.7,4) -- (0.7,3)
        -- (0.3,3) -- (0.3,2) -- (0.5,2) -- (0.5,1) -- (1.3,1)
        -- cycle;

    % pattern={Lines[angle=90, distance=2pt, line width=0.4pt]}, pattern 
    \fill[color=color5, opacity=0.4]
        (0, 3.8) -- (1, 3.8) -- (1, 4.8) -- (2, 4.8) -- (2, 5.0) -- (3, 5.0)
        -- (3, 4.5) -- (4, 4.5) -- (4, 3.5) -- (5, 3.5)
        -- (5, 1.5) -- (4, 1.5) -- (4, 0.8) -- (3, 0.8) -- (3, 0.3) -- (2, 0.3)
        -- (2, 0.5) -- (1, 0.5) -- (1, 1.0) -- (0, 1.0)
        -- cycle;

    % y-grid tick labels
    \foreach \y in {0,1,...,5} {
        \node[left, font=\tiny] at (0, \y) {$y_{\y}$};
    }

    % x-grid tick labels
    \foreach \x in {0,1,...,5} {
        \node[below, font=\tiny] at (\x, 0) {$x_{\x}$};
    }

    % Original grid lines (light gray)
    \foreach \x in {0,1,...,5} {
        \draw[gray!25, thin] (\x, 0) -- (\x, 5.2);
    }
    \foreach \y in {0,1,...,5} {
        \draw[gray!25, thin] (0, \y) -- (5.2, \y);
    }

    % Added horizontal grid lines from y bounds (slate dotted)
    \foreach \yval in {0.3, 0.5, 0.8, 1.0, 1.5, 3.5, 3.8, 4.5, 4.8, 5.0} {
        \draw[color8, thin, densely dotted] (0, \yval) -- (5.2, \yval);
    }

    % Added vertical grid lines from x bounds (slate dotted)
    \foreach \xval in {0.3, 0.5, 0.7, 1.3, 1.2, 3.2, 4.2, 3.8, 4.8, 5.0} {
        \draw[color8, thin, densely dotted] (\xval, 0) -- (\xval, 5.2);
    }

    % Vertical bars: max_x (right bounds) - solid color0
    \draw[color0] (3.8, 0) -- (3.8, 1);
    \draw[color0] (5.0, 1) -- (5.0, 2);
    \draw[color0] (4.8, 2) -- (4.8, 3);
    \draw[color0] (4.2, 3) -- (4.2, 4);
    \draw[color0] (3.2, 4) -- (3.2, 5);

    % Vertical bars: min_x (left bounds) - dashed color0
    \draw[color0] (1.3, 0) -- (1.3, 1);
    \draw[color0] (0.5, 1) -- (0.5, 2);
    \draw[color0] (0.3, 2) -- (0.3, 3);
    \draw[color0] (0.7, 3) -- (0.7, 4);
    \draw[color0] (1.2, 4) -- (1.2, 5);

    % Horizontal bars: max_y (upper bounds) - solid color5
    \draw[color5] (0, 3.8) -- (1, 3.8);
    \draw[color5] (1, 4.8) -- (2, 4.8);
    \draw[color5] (2, 5.0) -- (3, 5.0);
    \draw[color5] (3, 4.5) -- (4, 4.5);
    \draw[color5] (4, 3.5) -- (5, 3.5);

    % Horizontal bars: min_y (lower bounds) - dashed color5
    \draw[color5] (0, 1.0) -- (1, 1.0);
    \draw[color5] (1, 0.5) -- (2, 0.5);
    \draw[color5] (2, 0.3) -- (3, 0.3);
    \draw[color5] (3, 0.8) -- (4, 0.8);
    \draw[color5] (4, 1.5) -- (5, 1.5);

    % color11 dots - in BOTH regions (feasible intersection)
    \foreach \point in {
        (2,0.3), (3,0.3),
        (1.3,0.5), (2,0.5), (3,0.5),
        (1.3,0.8), (2,0.8), (3,0.8), (3.2,0.8), (3.8,0.8),
        (0.5,1), (0.7,1), (1,1), (1.2,1), (1.3,1), (2,1), (3,1), (3.2,1), (3.8,1), (4,1),
        (0.5,1.5), (0.7,1.5), (1,1.5), (1.2,1.5), (1.3,1.5), (2,1.5), (3,1.5), (3.2,1.5), (3.8,1.5), (4,1.5), (4.2,1.5), (4.8,1.5), (5,1.5),
        (0.3,2), (0.5,2), (0.7,2), (1,2), (1.2,2), (1.3,2), (2,2), (3,2), (3.2,2), (3.8,2), (4,2), (4.2,2), (4.8,2), (5,2),
        (0.3,3), (0.5,3), (0.7,3), (1,3), (1.2,3), (1.3,3), (2,3), (3,3), (3.2,3), (3.8,3), (4,3), (4.2,3), (4.8,3),
        (0.7,3.5), (1,3.5), (1.2,3.5), (1.3,3.5), (2,3.5), (3,3.5), (3.2,3.5), (3.8,3.5), (4,3.5), (4.2,3.5),
        (0.7,3.8), (1,3.8), (1.2,3.8), (1.3,3.8), (2,3.8), (3,3.8), (3.2,3.8), (3.8,3.8), (4,3.8),
        (1,4), (1.2,4), (1.3,4), (2,4), (3,4), (3.2,4), (3.8,4), (4,4),
        (1.2,4.5), (1.3,4.5), (2,4.5), (3,4.5), (3.2,4.5),
        (1.2,4.8), (1.3,4.8), (2,4.8), (3,4.8),
        (2,5), (3,5)
    } {
        \fill[color11] \point circle (1pt);
    }

    % color3 dots - NOT in both regions (infeasible)
    \foreach \point in {
        (0,0), (0.3,0), (0.5,0), (0.7,0), (1,0), (1.2,0), (1.3,0), (2,0), (3,0), (3.2,0), (3.8,0), (4,0), (4.2,0), (4.8,0), (5,0),
        (0,0.3), (0.3,0.3), (0.5,0.3), (0.7,0.3), (1,0.3), (1.2,0.3), (1.3,0.3), (3.2,0.3), (3.8,0.3), (4,0.3), (4.2,0.3), (4.8,0.3), (5,0.3),
        (0,0.5), (0.3,0.5), (0.5,0.5), (0.7,0.5), (1,0.5), (1.2,0.5), (3.2,0.5), (3.8,0.5), (4,0.5), (4.2,0.5), (4.8,0.5), (5,0.5),
        (0,0.8), (0.3,0.8), (0.5,0.8), (0.7,0.8), (1,0.8), (1.2,0.8), (4,0.8), (4.2,0.8), (4.8,0.8), (5,0.8),
        (0,1), (0.3,1), (4.2,1), (4.8,1), (5,1),
        (0,1.5), (0.3,1.5),
        (0,2),
        (0,3), (5,3),
        (0,3.5), (0.3,3.5), (0.5,3.5), (4.8,3.5), (5,3.5),
        (0,3.8), (0.3,3.8), (0.5,3.8), (4.2,3.8), (4.8,3.8), (5,3.8),
        (0,4), (0.3,4), (0.5,4), (0.7,4), (4.2,4), (4.8,4), (5,4),
        (0,4.5), (0.3,4.5), (0.5,4.5), (0.7,4.5), (1,4.5), (3.8,4.5), (4,4.5), (4.2,4.5), (4.8,4.5), (5,4.5),
        (0,4.8), (0.3,4.8), (0.5,4.8), (0.7,4.8), (1,4.8), (3.2,4.8), (3.8,4.8), (4,4.8), (4.2,4.8), (4.8,4.8), (5,4.8),
        (0,5), (0.3,5), (0.5,5), (0.7,5), (1,5), (1.2,5), (1.3,5), (3.2,5), (3.8,5), (4,5), (4.2,5), (4.8,5), (5,5)
    } {
        \fill[color3] \point circle (1pt);
    }

    \draw[very thick, color6] 
    (2,5) -- (3,5) -- 
    (3,4.5) -- (3.2,4.5) -- (3.2,4) --
    (4,4) -- (4,3.5) -- 
    (4.2,3.5) -- (4.2,3) -- (4.8,3) -- (4.8,2) --
    (5,2) -- (5,1.5) -- (4,1.5) --
    (4,1) -- (3.8,1) -- (3.8,0.8) -- (3,0.8) --
    (3,0.3) -- 
    (2,0.3) -- (2,0.5) --
    (1.3,0.5) -- (1.3,1) -- 
    (0.5,1) -- (0.5,2) -- 
    (0.3,2) -- (0.3,3) -- 
    (0.7,3) -- (0.7,3.8) -- (1,3.8) --
    (1,4) -- (1.2,4) -- 
    (1.2,4.8) -- (2,4.8) -- cycle;

    % Axes
    \draw[-stealth] (-0.1,0) -- (5.5,0) node[right] {$x$};
    \draw[-stealth] (0,-0.1) -- (0,5.5) node[above] {$y$};

    \mazenode[0.25]{2.3}{13.5}{0}{0}{0}{1}{90}
    \mazenode[0.25]{11.5}{19.5}{1}{1}{0}{0}{0}
    \mazenode[0.25]{15.5}{5.5}{0}{0}{0}{0}{180}

\end{tikzpicture}
\caption{Solid gray lines: equidistant grid ($N=6$). Dotted lines: refined grid from boundary values. Gold segments: $x$-bounds. Purple segments: $y$-bounds. Dark red dots: feasible grid points. Teal dots: infeasible points. green polygon: traced boundary~$\orth$. Brown arrows: walking direction for maze traversal. Surrounding color indicates feasible or infeasible next directions.}
\label{fig:grid-scanning}
\end{subfigure}
\hfill
\begin{subfigure}[t]{0.48\textwidth}
\centering
\begin{tikzpicture}
    % Bounding box
    \draw[dashed, gray!70, thick] (0,0) rectangle (5,5);
    
    % Main rectangle body (shaded)
    \fill[color10!15] (0,0) rectangle (3.2,2.8);
    
    % Draw the degenerate polygon path
    \draw[thick, color10] (0,0) -- (0,5) -- (0,2.8) 
        -- (3.2,2.8) -- (3.2,0) -- (5,0) -- (0,0);
    
    % Vertices
    \fill[color6] (0,0) circle (3pt) 
        node[above right, font=\scriptsize] {$(x_{lb},y_{lb})$};
    \fill[color6] (0,5) circle (3pt) 
        node[above right, font=\scriptsize] {$(x_{lb},y_{ub})$};
    \fill[color6] (0,2.8) circle (3pt) 
        node[above right, font=\scriptsize] {$(x_{lb},y^*_{ub})$};
    \fill[color6] (3.2,2.8) circle (3pt) 
        node[above right, font=\scriptsize] {$(x^*_{ub},y^*_{ub})$};
    \fill[color6] (3.2,0) circle (3pt)
        node[above left, font=\scriptsize] {$(x^*_{ub},y_{lb})$};
    \fill[color6] (5,0) circle (3pt) 
        node[above left, font=\scriptsize] {$(x_{ub},y_{lb})$};
    
    % Dashed reference lines
    \draw[thin, dashed, gray] (3.2,0) -- (3.2,2.8);
    \draw[thin, dashed, gray] (0,2.8) -- (3.2,2.8);

    % y-grid tick labels
    \foreach \y in {0,1,...,5} {
        \node[left, font=\tiny] at (0, \y) {\phantom{$y_{\y}$}};
    }

    % x-grid tick labels
    \foreach \x in {0,1,...,5} {
        \node[below, font=\tiny] at (\x, 0) {\phantom{$x_{\x}$}};
    }
    
    % Axis arrows
    \draw[-stealth] (-0.1,-0) -- (5.5,-0) node[right, font=\small] {$x$};
    \draw[-stealth] (-0,-0.1) -- (-0,5.5) node[above, font=\small] {$y$};
\end{tikzpicture}
\caption{Dashed box: original bounding box $[x_{lb},x_{ub}]\times[y_{lb},y_{ub}]$. Green polygon: degenerate indicator polygon with zero-width spikes along two edges.}
\label{fig:indicator-polygon}
\end{subfigure}
\caption{Domain projection via grid scanning.}
\label{fig:domain-projection}
\end{figure}

\paragraph{Configurations.}
We compare four configurations based on the handling of bilinear terms. The first configuration, \baseshort\, represents the native nonlinear handling capabilities of the solver: multilinear terms are reformulated into bilinear chains and passed to Gurobi as nonlinear constraints after OBBT, with no additional cuts generated. The second configuration, \locatellishort\, computes the full grid-scanning projection and applies a two-dimensional convex hull computation to the polygon vertices, yielding a convex polytope~$P$. Cuts are then generated exclusively from vertex triples of~$P$ as linear inequalities. Crucially, these linear cuts approximate but do not recover the three-dimensional convex hull of $\{(x,y,xy) : (x,y) \in P\}$, since the exact convex hull over a two-dimensional convex domain is in general a nonlinear set. The third configuration, \stairshort\, uses the full non-convex orthogonal polygon~$\orth$ directly for cut generation, exploiting the theory developed in the present paper. As established in the preceding sections, the three-dimensional convex hull of $\{(x,y,xy) : (x,y) \in \orth\}$ over an orthogonal polygon is a polytope whose facets lie in supporting hyperplanes determined by triples of lifted vertices of $\orth$. The cut generation procedure thus recovers the exact convex hull using only linear inequalities. The fourth configuration, \indicatorshort\, uses the simplified two-subproblem projection (Figure~\ref{fig:indicator-polygon}) for bilinear expressions arising from indicator or semi-continuous variable structures, with cuts generated from the resulting polygon in the same manner as \stairshort\. In all three cut-generating configurations, the bilinear terms are additionally passed to the solver as nonlinear constraints, and the generated cuts serve as valid inequalities that strengthen the relaxation.

\paragraph{Cut Generation.}
Once the domain has been determined, either as a convex polytope~$P$ in the \locatellishort\ configuration or as an orthogonal polygon~$\orth$ in the \stairshort\ and \indicatorshort\ configurations, we proceed with cut generation. Let $V = \{v_1, \ldots, v_k\}$, $k \geq 3$, denote the vertex set of the respective domain. Checkpoints are constructed by sampling evenly spaced points along each edge of the domain.

For every combination of three vertices from~$V$, the algorithm checks collinearity via the cross product (skipping collinear triples) and fits the unique hyperplane $z = ax + by + c$ by solving the $3 \times 3$ linear system determined by requiring $z_m = v_{m,x} \cdot v_{m,y}$ at each of the three vertices. If the system is singular, the triple is skipped. The hyperplane is then validated against all checkpoints: if $ax + by + c \leq xy$ holds at every checkpoint, the constraint $z \geq ax + by + c$ is added as a valid underestimator; conversely, if $ax + by + c \geq xy$ holds everywhere, the constraint $z \leq ax + by + c$ is added as an overestimator. If neither condition holds globally, the triple is discarded. The generated cuts are classified as either overestimators or underestimators and stored with the corresponding bilinear expression.

The key distinction between configurations lies in the vertex set and the validation domain. In \locatellishort\, vertex triples are drawn from the convex polytope~$P$ and checkpoints lie on the edges of~$P$. In \stairshort\ and \indicatorshort\, vertex triples are drawn from the orthogonal polygon~$\orth$ and checkpoints lie on the edges of~$\orth$.

\paragraph{Evaluation Metrics.}
We evaluate the computational impact of the proposed methods using four metrics.

The first metric is \textbf{Domain Volume Reduction}. For each bilinear expression $z = xy$, we compute the three-dimensional volume enclosed between the concave envelope and convex envelope of the bilinear function over the feasible $(x,y)$-domain. In the \locatellishort\ configuration, this is the volume of the three-dimensional convex hull of the surface $\{(x,y,xy) : (x,y) \in P\}$ over the convex polytope. In the \stairshort\ configuration, it is the volume of the three-dimensional convex hull of the same surface over the orthogonal polygon. In the \indicatorshort\ configuration, it is the volume of the three-dimensional convex hull of the same surface over the degenerate indicator polygon.  For reference, we also compute the McCormick relaxation volume over the full bounding box, which evaluates only the McCormick envelope constraints without any domain restriction. All volumes are computed numerically. The reduction from box volume to polytope volume isolates the benefit of OBBT and convex domain projection alone, while the further reduction from polytope volume to polygon volume demonstrates the additional value of non-convex geometric reasoning.

The second metric is \textbf{Root Gap Reduction}, which measures how much the bound obtained at the root node of the branch-and-bound tree improves relative to the baseline configuration. The third metric is the \textbf{Number of Branch-and-Bound Nodes}, a direct indicator of how much enumeration effort is required after the relaxation and cuts have been applied. The fourth metric is \textbf{Solving Time}, capturing the solving time of Gurobi without preprocessing and model buildup time.

\subsection{Results}
This section presents and analyzes the empirical results of our computational experiments across the four evaluation metrics.

\subsubsection{Domain Volume Reduction}
\input{studyplots/plot_volume_reduction_polygon}

Figures~\ref{plot:volume-reduction-polygon} and~\ref{plot:volume-reduction-polytope} present scatter plots of the domain volume reduction achieved by the \indicatorshort{} and \stairshort{} configurations against the \locatellishort{} baseline, measured over the polygon and the convex polytope, respectively. Each point represents one instance; points below the diagonal indicate that the respective configuration yields a larger relaxation volume than \locatellishort{}, while points above the diagonal indicate the opposite.

The scatter plots reveal that \stairshort{} consistently achieves volume reductions comparable to or slightly exceeding those of \locatellishort{} when measured over the polygon, and notably tighter relaxations when measured over the polytope. This is expected, since \stairshort{} exploits the full non-convex orthogonal polygon geometry, recovering facets of the exact three-dimensional convex hull that the convex polytope approximation cannot capture. The \indicatorshort{} configuration exhibits a markedly different behaviour: on instances where indicator or semi-continuous variable structures are present in the bilinear terms, \indicatorshort{} achieves very large volume reductions, often far surpassing both \locatellishort{} and \stairshort{}. However, on instances lacking such structure, the degenerate two-subproblem polygon provides almost no tightening beyond the bounding box, resulting in volume reductions near zero.

\definecolor{color1}{RGB}{25, 42, 86}
\definecolor{color2}{RGB}{156, 56, 72}
\definecolor{color3}{RGB}{0, 119, 139}

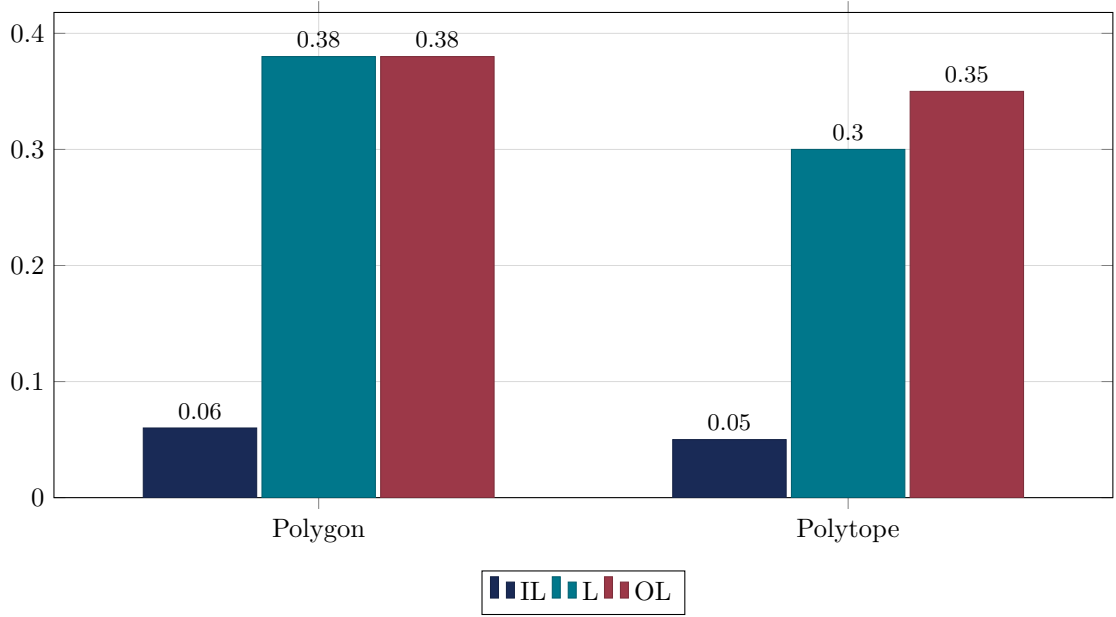
\begin{figure}
\begin{tikzpicture}
\begin{axis}[
    ybar,
    bar width=1.5cm,
    enlarge x limits=0.5,
    legend style={at={(0.5,-0.15)}, anchor=north, legend columns=-1},
    legend cell align=left,
    symbolic x coords={Polygon, Polytope},
    xtick=data,
    grid=major,
    grid style={line width=.1pt, draw=gray!30},
    ymin=0,
    width=\textwidth,
    height=8cm,
    nodes near coords,
    nodes near coords style={font=\small, /pgf/number format/.cd, fixed, precision=2}
]

% Configuration: indicator_locatelli
\addplot[fill=color1, draw=color1!80!black] coordinates {
    (Polygon, 0.06)
    (Polytope, 0.05)
};
\addlegendentry{\indicatorshort}

% Configuration: locatelli
\addplot[fill=color3, draw=color3!80!black] coordinates {
    (Polygon, 0.38)
    (Polytope, 0.30)
};
\addlegendentry{\locatellishort}

% Configuration: stair_locatelli
\addplot[fill=color2, draw=color2!80!black] coordinates {
    (Polygon, 0.38)
    (Polytope, 0.35)
};
\addlegendentry{\stairshort}

\end{axis}
\end{tikzpicture}
\caption{Shifted geometric mean of the domain volume reduction across all test instances, measured over the polygon and over the convex polytope. Higher values indicate a larger fraction of the McCormick bounding-box volume that is removed by the respective configuration.}
    \label{plot:volume-reduction-bar}
\end{figure}

Figure~\ref{plot:volume-reduction-bar} summarises these findings via the shifted geometric mean (SGM) of the volume reduction ratios across all instances. \stairshort{} attains an SGM of 0.38 over the polygon and 0.35 over the polytope, compared to 0.38 and 0.30 for \locatellishort{}. The improvement from \locatellishort{} to \stairshort{} is visible exclusively over the polytope, where the non-convex domain information translates most directly into additional valid cutting planes. \indicatorshort{} achieves SGM values of 0.06 and 0.05, reflecting the fact that its strong reductions on a subset of instances are averaged with negligible reductions on the remaining instances.

\subsubsection{Root Gap Reduction}

\definecolor{color0}{RGB}{205, 145, 50}
\definecolor{color1}{RGB}{25, 42, 86}
\definecolor{color2}{RGB}{156, 56, 72}
\definecolor{color3}{RGB}{0, 119, 139}

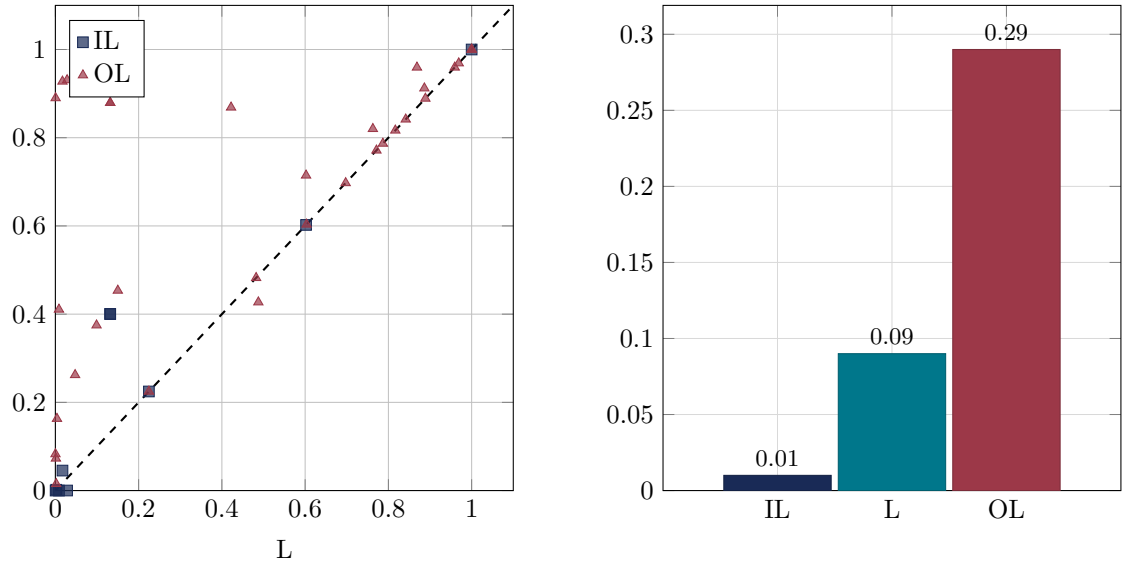
\begin{figure}[htbp]
\centering
\begin{minipage}[t]{0.49\textwidth}
    \centering
    \begin{tikzpicture}[baseline=(current bounding box.north)]
    \begin{axis}[
        xlabel={L},
        legend pos=north west,
        legend cell align=left,
        grid=both,
        grid style={line width=.1pt, draw=gray!20},
        major grid style={line width=.2pt, draw=gray!50},
        xmin=2.447406805925165e-07,
        xmax=1.1,
        ymin=2.447406805925165e-07,
        ymax=1.1,
        width=\linewidth,
        height=8cm
    ]

    % Diagonal reference line (y = x)
    \addplot[black, dashed, thick, domain=2.447406805925165e-07:1.1, forget plot] {x};

    % Configuration: indicator_locatelli
    \addplot[only marks, mark=square*, mark size=2pt, color=color1, fill=color1, fill opacity=0.7] coordinates {
                (-0.0, 0.0019064765002353243)
        (0.13171387193344078, 0.40037323587729756)
        (1.0, -0.0)
        (-0.0, -0.0)
        (-0.0, -0.0)
        (-0.0, -0.0)
        (-0.0, -0.0)
        (-0.0, -0.0)
        (-0.0, -0.0)
        (-0.0, 0.0010520547082670136)
        (-0.0, -0.0)
        (-0.0, -0.0)
        (-0.0, -0.0)
        (-0.0, -0.0)
        (0.002056541467256897, 0.0018374161720122887)
        (-0.0, 0.00010905727369345218)
        (-0.0, 5.4484249413369036e-05)
        (-0.0, -0.0)
        (0.8416144978484617, -0.0)
        (-0.0, -0.0)
        (-0.0, 0.00039587178330791485)
        (0.0005864297634535328, 1.3148649404824837e-05)
        (-0.0, -0.0)
        (0.028210101184006974, 0.0001410585434439632)
        (-0.0, -0.0)
        (-0.0, -0.0)
        (0.0474241919408453, -0.0)
        (0.7627103781065449, -0.0)
        (0.4220891340636073, -0.0)
        (-0.0, -0.0)
        (-0.0, -0.0)
        (0.0988972471480532, -0.0)
        (-0.0, -0.0)
        (-0.0, -0.0)
        (-0.0, -0.0)
        (-0.0, -0.0)
        (0.48224754869379394, -0.0)
        (0.8686153948324772, -0.0)
        (-0.0, -0.0)
        (0.7867896167396878, -0.0)
        (0.9594012515903654, -0.0)
        (0.6034371550557387, -0.0)
        (0.9689880413815426, -0.0)
        (-0.0, -0.0)
        (0.7713116053473111, -0.0)
        (0.8888883213281592, -0.0)
        (-0.0, -0.0)
        (0.8864299996077546, -0.0)
        (0.8166156337041521, -0.0)
        (-0.0, -0.0)
        (-0.0, -0.0)
        (-0.0, -0.0)
        (0.017042693779132725, 0.04535400922236471)
        (-0.0, -0.0)
        (-0.0, -0.0)
        (0.0038466650170672244, -0.0)
        (-0.0, 0.04038653868821847)
        (0.22467723749193294, 0.22467723749193294)
        (-0.0, 0.014561285990893505)
        (-0.0, 0.022926854143070934)
        (-0.0, 0.015081240157886951)
        (-0.0, -0.0)
        (-0.0, 0.005821384402829451)
        (-0.0, 2.5796142447009813e-05)
        (-0.0, -0.0)
        (-0.0, -0.0)
        (-0.0, -0.0)
        (-0.0, -0.0)
        (-0.0, -0.0)
        (0.001261974896337777, -0.0)
        (-0.0, -0.0)
        (-0.0, -0.0)
        (-0.0, -0.0)
        (-0.0, -0.0)
        (-0.0, -0.0)
        (-0.0, -0.0)
        (-0.0, -0.0)
        (-0.0, -0.0)
        (-0.0, -0.0)
        (-0.0, -0.0)
        (-0.0, -0.0)
        (-0.0, -0.0)
        (-0.0, -0.0)
        (-0.0, -0.0)
        (1.0, 1.0)
        (0.0, 0.0)
        (0.00026139378002320614, -0.0)
        (-0.0, -0.0)
        (-0.0, -0.0)
        (-0.0, -0.0)
        (0.6022631996071861, 0.6022631996071861)
        (-0.0, -0.0)
        (-0.0, -0.0)
        (0.00862479711238096, 1.4442058124034738e-05)
        (-0.0, -0.0)
        (0.13171387193344078, 0.40037323587729756)
        (0.15, -0.0)
        (0.4876101907602499, -0.0)
        (0.697368, -0.0)
        (-0.0048793817141411676, -0.0)
        (-0.0, -0.0)
        (-0.0, 0.00048195333053382587)
        (-0.0, 0.0006751584251543977)
        (-0.0, 0.00047657142857125174)
        (-0.0, 0.0003030000000000162)
        (-0.0, -0.0)
        (-0.0, -0.0)
        (-0.0, -0.0)
        (-0.0, -0.0)
        (-0.0, -0.0)
        (-0.0, -0.0)
        (-0.0, -0.0)
        (-4.229831747218066e-06, -4.229831747218066e-06)
        (-0.0, -0.0)
        (-0.0, -0.0)
        (-0.0, -0.0)
};
\addlegendentry{\indicatorshort}

% Configuration: stair_locatelli
\addplot[only marks, mark=triangle*, mark size=2pt, color=color2, fill=color2, fill opacity=0.7] coordinates {
        (-0.0, 0.029966498664616998)
        (0.13171387193344078, 0.8792013335406863)
        (1.0, 1.0)
        (-0.0, -0.0)
        (-0.0, -0.0)
        (-0.0, -0.0)
        (-0.0, -0.0)
        (-0.0, -0.0)
        (-0.0, -0.0)
        (-0.0, 0.5357687231978266)
        (-0.0, -0.0)
        (-0.0, 0.5)
        (-0.0, 0.26870500251073864)
        (-0.0, -0.0)
        (0.002056541467256897, 0.015218344295161348)
        (-0.0, 0.9188082449326268)
        (-0.0, 0.9351954148399652)
        (-0.0, 0.9117556692995678)
        (0.8416144978484617, 0.8416144978484617)
        (-0.0, 0.005030789113886598)
        (-0.0, 0.7684071867564513)
        (0.0005864297634535328, 0.8899579198275073)
        (-0.0, 0.9349922978793211)
        (0.028210101184006974, 0.9307767882015795)
        (-0.0, -0.0)
        (-0.0, -0.0)
        (0.0474241919408453, 0.2621910171178527)
        (0.7627103781065449, 0.8200932621069934)
        (0.4220891340636073, 0.8687540386095314)
        (-0.0, -0.0)
        (-0.0, 0.9328173569300326)
        (0.0988972471480532, 0.37482834325663117)
        (-0.0, 0.5092322003904899)
        (-0.0, 0.9996033019677392)
        (-0.0, 0.9708433621275754)
        (-0.0, 1.0000002326024833)
        (0.48224754869379394, 0.48224754869379394)
        (0.8686153948324772, 0.9593776593775393)
        (-0.0, 0.9416792529820702)
        (0.7867896167396878, 0.7867896167396878)
        (0.9594012515903654, 0.9594012515903654)
        (0.6034371550557387, 0.6034371550557387)
        (0.9689880413815426, 0.9689880413815426)
        (-0.0, 0.9805391754598236)
        (0.7713116053473111, 0.7713116053473111)
        (0.8888883213281592, 0.8888883213281592)
        (-0.0, -0.0)
        (0.8864299996077546, 0.9121588436318625)
        (0.8166156337041521, 0.8166156337041521)
        (-0.0, -0.0)
        (-0.0, -0.0)
        (-0.0, -0.0)
        (0.017042693779132725, 0.9276031305586045)
        (-0.0, -0.0)
        (-0.0, 0.18119751105979745)
        (0.0038466650170672244, 0.16295020477728445)
        (-0.0, 0.2280384456013255)
        (0.22467723749193294, 0.22467723749193294)
        (-0.0, 0.2753709984016611)
        (-0.0, 0.27405362705471664)
        (-0.0, 0.2481301829391243)
        (-0.0, 0.243965032986499)
        (-0.0, 0.22902684563758377)
        (-0.0, 0.0001085769837462018)
        (-0.0, 0.31919200025535377)
        (-0.0, 0.3308829201854134)
        (-0.0, 0.3068782329328105)
        (-0.0, 0.280991910477957)
        (-0.0, 0.019163678968002668)
        (0.001261974896337777, 0.0728944652662521)
        (-0.0, -0.0)
        (-0.0, 0.20697700000000055)
        (-0.0, 0.04545447726382972)
        (-0.0, 0.424812)
        (-0.0, 0.16666699999999993)
        (-0.0, 0.63721775)
        (-0.0, 0.6371855007746582)
        (-0.0, 0.25)
        (-0.0, 0.14285799997342938)
        (-0.0, 0.14285799999714396)
        (-0.0, 0.2500003749998124)
        (-0.0, -0.0)
        (-0.0, -0.0)
        (-0.0, -0.0)
        (1.0, 1.0)
        (0.0, 0.0)
        (0.00026139378002320614, 0.08279396642058519)
        (-0.0, -0.0)
        (-0.0, -0.0)
        (-0.0, 0.42481100000000027)
        (0.6022631996071861, 0.7145878134311632)
        (-0.0, 0.011224489796802708)
        (-0.0, 0.953488799999237)
        (0.00862479711238096, 0.4105906009117081)
        (-0.0, -0.0)
        (0.13171387193344078, 0.8792013335406863)
        (0.15, 0.453571)
        (0.4876101907602499, 0.427049630209282)
        (0.697368, 0.697368)
        (-0.0048793817141411676, 0.01103393661923336)
        (-0.0, 0.7266099700028615)
        (-0.0, 0.4069473718423276)
        (-0.0, 0.6696272728341133)
        (-0.0, -0.0)
        (-0.0, -0.0)
        (-0.0, -0.0)
        (-0.0, -0.0)
        (-0.0, -0.0)
        (-0.0, -0.0)
        (-0.0, -0.0)
        (-0.0, -0.0)
        (-0.0, -0.0)
        (-4.229831747218066e-06, 0.22089238840783862)
        (-0.0, 0.14969205390484477)
        (-0.0, 0.15677981247309375)
        (-0.0, 0.22641716581656882)
    };
    \addlegendentry{\stairshort}

    \end{axis}
    \end{tikzpicture}
    \subcaption{Scatter plot comparing the root gap reduction of \indicatorshort\ and \stairshort\ against \locatellishort\ for each instance; points above the diagonal indicate superior performance of the respective configuration.}
    \label{plot:root-gap-reduction}
\end{minipage}%
\hfill
\begin{minipage}[t]{0.49\textwidth}
    \centering
    \begin{tikzpicture}[baseline=(current bounding box.north)]
    \begin{axis}[
        ybar,
        bar width=1.42cm,
        enlarge x limits=0.5,
        symbolic x coords={\indicatorshort, \locatellishort, \stairshort},
        xtick={\indicatorshort, \locatellishort, \stairshort},
        grid=major,
        grid style={line width=.1pt, draw=gray!30},
        ymin=0,
        width=\linewidth,
        height=8cm,
        tick align=inside,
        nodes near coords,
        nodes near coords style={font=\small, anchor=south, /pgf/number format/fixed},
        yticklabel style={/pgf/number format/fixed}
    ]

    \addplot[fill=color1, draw=color1!80!black, bar shift=0pt] coordinates {(\indicatorshort, 0.01)};
    \addplot[fill=color3, draw=color3!80!black, bar shift=0pt] coordinates {(\locatellishort, 0.09)};
    \addplot[fill=color2, draw=color2!80!black, bar shift=0pt] coordinates {(\stairshort, 0.29)};

    \end{axis}
    \end{tikzpicture}
    
        \vspace{5.25mm} 
    \subcaption{Shifted geometric mean of the root gap reduction for each configuration.}
    \label{plot:root-gap-reduction-bar}
\end{minipage}
\caption{
Root gap reduction relative to the nonlinear baseline (\baseshort). Compared 116 instances where at least one bilinear term has non-empty domain and the root gap for \baseshort\ is not 0.}
\end{figure}

Figure~\ref{plot:root-gap-reduction} displays the root gap reduction of \indicatorshort{} and \stairshort{} relative to the \locatellishort{} configuration for each instance, and Figure~\ref{plot:root-gap-reduction-bar} reports the corresponding shifted geometric means. The root gap reduction is defined as the fraction of the initial relaxation gap that is closed at the root node by the addition of the respective cuts before branching begins.

The scatter plot in Figure~\ref{plot:root-gap-reduction} demonstrates that \stairshort{} consistently outperforms \locatellishort{} across the majority of instances. Nearly all \stairshort{} points lie above the diagonal, confirming that the additional facets recovered by the orthogonal polygon convexification translate into meaningfully stronger root relaxations. The \indicatorshort{} configuration produces mixed results: on instances with exploitable indicator structure, \indicatorshort{} achieves root gap reductions competitive with or exceeding \locatellishort{}, but on other instances it provides little to no improvement, with several points clustered near zero on the vertical axis.

The shifted geometric means in Figure~\ref{plot:root-gap-reduction-bar} quantify the aggregate effect. \stairshort{} achieves an SGM root gap reduction of 0.29, more than triple the value of 0.09 obtained by \locatellishort{}. \indicatorshort{} achieves an SGM of only 0.01, reflecting the heterogeneity of its performance across the test set. These results confirm that the theoretical tightness guarantees of the Ortho-Locatelli convexification carry over to practical solver performance at the root node.

\subsubsection{Branch-and-Bound Nodes}

\definecolor{color0}{RGB}{205, 145, 50}
\definecolor{color1}{RGB}{25, 42, 86}
\definecolor{color2}{RGB}{156, 56, 72}
\definecolor{color3}{RGB}{0, 119, 139}

\begin{figure}[htbp]
\centering
\begin{minipage}[t]{0.49\textwidth}
    \centering
    \begin{tikzpicture}[baseline=(current bounding box.north)]
    \begin{axis}[
        xlabel={\baseshort},
        legend pos=north west,
        legend cell align=left,
        grid=both,
        grid style={line width=.1pt, draw=gray!20},
        major grid style={line width=.2pt, draw=gray!50},
        xmin=0.9,
        xmax=6883437.000000001,
        ymin=0.9,
        ymax=6883437.000000001,
        width=\linewidth,
        height=8cm,
        xmode=log,
        ymode=log
    ]

    % Diagonal reference line (y = x)
    \addplot[black, dashed, thick, domain=0.9:6883437.000000001, forget plot] {x};

    % Configuration: indicator_locatelli
    \addplot[only marks, mark=*, mark size=2pt, color=color1, fill=color1, fill opacity=0.7] coordinates {
        (5.0, 7.0)
        (100.0, 131.0)
        (111.0, 139.0)
        (1.0, 60.0)
        (1.0, 60.0)
        (1.0, 1.0)
        (1.0, 1.0)
        (2079.0, 1950.0)
        (39.0, 39.0)
        (3.0, 3.0)
        (2512771.0, 1873251.0)
        (10802.0, 9447.0)
        (67.0, 65.0)
        (18558.0, 17428.0)
        (1892.0, 1911.0)
        (3505.0, 3393.0)
        (393.0, 363.0)
        (547.0, 556.0)
        (6702.0, 6377.0)
        (18.0, 20.0)
        (321.0, 210.0)
        (55138.0, 52529.0)
        (19932.0, 26252.0)
        (3394322.0, 4462353.0)
        (20965.0, 21689.0)
        (69693.0, 119559.0)
        (17852.0, 17132.0)
        (6741.0, 6585.0)
        (11361.0, 8998.0)
        (2391.0, 1647.0)
        (87.0, 99.0)
        (57.0, 77.0)
        (79.0, 81.0)
        (1575.0, 1687.0)
        (201.0, 502.0)
        (13245.0, 12471.0)
        (157.0, 418.0)
        (213.0, 243.0)
        (47577.0, 71541.0)
        (2195.0, 3005.0)
        (890513.0, 368599.0)
        (667.0, 2213.0)
        (4171.0, 4155.0)
        (32916.0, 36397.0)
        (115575.0, 88276.0)
        (59487.0, 37471.0)
        (29026.0, 27240.0)
        (29.0, 29.0)
        (156.0, 137.0)
        (6.0, 1.0)
        (6464.0, 688.0)
        (881.0, 1158.0)
        (5390.0, 1696.0)
        (5261.0, 5822.0)
        (2947.0, 2969.0)
        (1992.0, 2554.0)
        (7.0, 5.0)
        (1739.0, 1395.0)
        (13555.0, 20850.0)
        (1414872.0, 1345228.0)
        (879.0, 445.0)
        (1022.0, 497.0)
        (204.0, 183.0)
        (1.0, 1.0)
        (1.0, 5.0)
        (5.0, 3.0)
        (5.0, 5.0)
        (5.0, 5.0)
        (67.0, 45.0)
        (61.0, 45.0)
        (25.0, 25.0)
        (23.0, 23.0)
        (21.0, 33.0)
        (3.0, 1.0)
        (25.0, 25.0)
        (33.0, 36.0)
        (187.0, 193.0)
        (517.0, 109.0)
        (607.0, 601.0)
        (924.0, 845.0)
        (246504.0, 195984.0)
        (6130.0, 5976.0)
        (4568.0, 4712.0)
        (2301.0, 6298.0)
        (25.0, 8.0)
};
\addlegendentry{\indicatorshort}

% Configuration: locatelli
\addplot[only marks, mark=square*, mark size=2pt, color=color1, fill=color1, fill opacity=0.7] coordinates {
                (5.0, 7.0)
        (100.0, 131.0)
        (111.0, 139.0)
        (1.0, 61.0)
        (1.0, 61.0)
        (1.0, 27.0)
        (1.0, 27.0)
        (2079.0, 1993.0)
        (39.0, 39.0)
        (3.0, 3.0)
        (2512771.0, 4716458.0)
        (10802.0, 10571.0)
        (67.0, 65.0)
        (18558.0, 19769.0)
        (1892.0, 873.0)
        (3505.0, 3249.0)
        (393.0, 357.0)
        (547.0, 541.0)
        (6702.0, 6100.0)
        (18.0, 43.0)
        (321.0, 263.0)
        (55138.0, 65826.0)
        (19932.0, 50837.0)
        (3394322.0, 5046806.0)
        (20965.0, 14496.0)
        (69693.0, 67323.0)
        (17852.0, 18179.0)
        (6741.0, 4693.0)
        (11361.0, 12622.0)
        (2391.0, 1879.0)
        (87.0, 79.0)
        (57.0, 79.0)
        (79.0, 65.0)
        (1575.0, 1391.0)
        (201.0, 137.0)
        (13245.0, 14193.0)
        (157.0, 54.0)
        (213.0, 221.0)
        (47577.0, 23990.0)
        (2195.0, 2711.0)
        (890513.0, 1020014.0)
        (667.0, 2871.0)
        (4171.0, 5283.0)
        (32916.0, 51348.0)
        (115575.0, 106569.0)
        (59487.0, 53413.0)
        (29026.0, 18649.0)
        (29.0, 17.0)
        (156.0, 224.0)
        (6.0, 1.0)
        (6464.0, 4485.0)
        (881.0, 918.0)
        (5390.0, 1273.0)
        (5261.0, 4337.0)
        (2947.0, 3610.0)
        (1992.0, 1551.0)
        (7.0, 5.0)
        (1739.0, 1547.0)
        (13555.0, 14080.0)
        (1414872.0, 1387402.0)
        (879.0, 577.0)
        (1022.0, 487.0)
        (204.0, 189.0)
        (1.0, 635.0)
        (1.0, 5.0)
        (5.0, 3.0)
        (5.0, 5.0)
        (5.0, 5.0)
        (67.0, 53.0)
        (61.0, 53.0)
        (25.0, 19.0)
        (23.0, 17.0)
        (21.0, 29.0)
        (3.0, 1.0)
        (25.0, 21.0)
        (33.0, 31.0)
        (187.0, 205.0)
        (517.0, 109.0)
        (607.0, 601.0)
        (924.0, 219.0)
        (246504.0, 83469.0)
        (6130.0, 5312.0)
        (4568.0, 2641.0)
        (2301.0, 6098.0)
        (25.0, 9.0)
};
\addlegendentry{\locatellishort}

% Configuration: stair_locatelli
\addplot[only marks, mark=triangle*, mark size=2pt, color=color2, fill=color2, fill opacity=0.7] coordinates {
        (5.0, 5.0)
        (100.0, 65.0)
        (111.0, 34.0)
        (1.0, 1.0)
        (1.0, 1.0)
        (1.0, 1.0)
        (1.0, 1.0)
        (2079.0, 2005.0)
        (39.0, 33.0)
        (3.0, 3.0)
        (2512771.0, 3285338.0)
        (10802.0, 10441.0)
        (67.0, 1.0)
        (18558.0, 18228.0)
        (1892.0, 1034.0)
        (3505.0, 2988.0)
        (393.0, 589.0)
        (547.0, 595.0)
        (6702.0, 5471.0)
        (18.0, 288.0)
        (321.0, 346.0)
        (55138.0, 67121.0)
        (19932.0, 20978.0)
        (3394322.0, 5841080.0)
        (20965.0, 7281.0)
        (69693.0, 32204.0)
        (17852.0, 18786.0)
        (6741.0, 5696.0)
        (11361.0, 13482.0)
        (2391.0, 405.0)
        (87.0, 71.0)
        (57.0, 61.0)
        (79.0, 57.0)
        (1575.0, 1013.0)
        (201.0, 199.0)
        (13245.0, 19197.0)
        (157.0, 40.0)
        (213.0, 69.0)
        (47577.0, 21256.0)
        (2195.0, 1990.0)
        (890513.0, 892913.0)
        (667.0, 2152.0)
        (4171.0, 6648.0)
        (32916.0, 36598.0)
        (115575.0, 155926.0)
        (59487.0, 45020.0)
        (29026.0, 30522.0)
        (29.0, 11.0)
        (156.0, 199.0)
        (6.0, 1.0)
        (6464.0, 629.0)
        (881.0, 144.0)
        (5390.0, 1413.0)
        (5261.0, 4649.0)
        (2947.0, 2576.0)
        (1992.0, 1376.0)
        (7.0, 7.0)
        (1739.0, 1537.0)
        (13555.0, 20058.0)
        (1414872.0, 1581330.0)
        (879.0, 371.0)
        (1022.0, 1.0)
        (204.0, 160.0)
        (1.0, 117.0)
        (1.0, 1.0)
        (5.0, 7.0)
        (5.0, 3.0)
        (5.0, 5.0)
        (67.0, 53.0)
        (61.0, 58.0)
        (25.0, 29.0)
        (23.0, 21.0)
        (21.0, 28.0)
        (3.0, 1.0)
        (25.0, 25.0)
        (33.0, 35.0)
        (187.0, 205.0)
        (517.0, 109.0)
        (607.0, 601.0)
        (924.0, 219.0)
        (246504.0, 233859.0)
        (6130.0, 5828.0)
        (4568.0, 2151.0)
        (2301.0, 4571.0)
        (25.0, 9.0)
    };
    \addlegendentry{\stairshort}

    \end{axis}
    \end{tikzpicture}
    \subcaption{Scatter plot on a logarithmic scale comparing each configuration against \baseshort; points below the diagonal indicate fewer nodes than the baseline.}
\label{plot:nr-nodes}
\end{minipage}%
\hfill
\begin{minipage}[t]{0.49\textwidth}
    \centering
    \begin{tikzpicture}[baseline=(current bounding box.north)]
    \begin{axis}[
        ybar,
        bar width=1cm,
        enlarge x limits=0.4,
        symbolic x coords={\baseshort, \indicatorshort, \locatellishort, \stairshort},
        xtick={\baseshort, \indicatorshort, \locatellishort, \stairshort},
        grid=major,
        grid style={line width=.1pt, draw=gray!30},
        ymin=0,
        width=\linewidth,
        height=8cm,
        nodes near coords,
        tick align=inside,
        nodes near coords style={font=\small, anchor=south, /pgf/number format/fixed},
        yticklabel style={/pgf/number format/fixed}
    ]

    \addplot[fill=color0, draw=color0!80!black, bar shift=0pt] coordinates {(\baseshort, 375)};
    \addplot[fill=color1, draw=color1!80!black, bar shift=0pt] coordinates {(\indicatorshort, 379)};
    \addplot[fill=color3, draw=color3!80!black, bar shift=0pt] coordinates {(\locatellishort, 376)};
    \addplot[fill=color2, draw=color2!80!black, bar shift=0pt] coordinates {(\stairshort, 312)};

    \end{axis}
    \end{tikzpicture}
        \vspace{5.5mm} 
        \subcaption{Shifted geometric mean of the node count across all instances.}
\label{plot:nr-nodes-bar}
\end{minipage}
\caption{Number of branch-and-bound nodes explored. Compared 85 instances where at least one bilinear term has non-empty domain and at least one configuration was not solved at the root node.}
\end{figure}
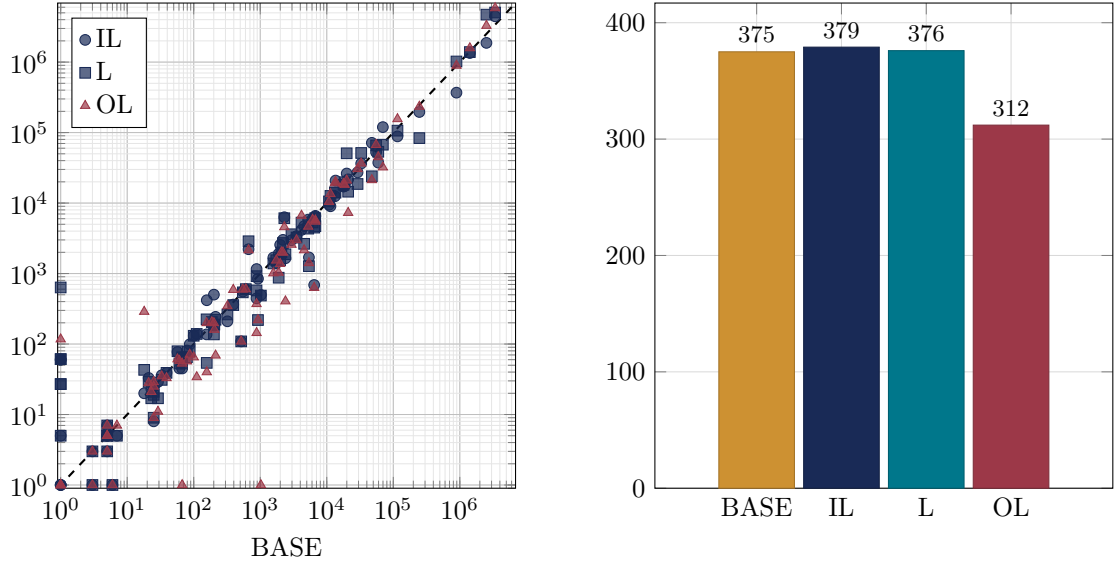

Figure~\ref{plot:nr-nodes} shows the number of branch-and-bound nodes explored by each configuration plotted against the nonlinear baseline (\baseshort{}) on a logarithmic scale. Figure~\ref{plot:nr-nodes-bar} reports the shifted geometric means.

The scatter plot does not reveal a uniformly dominant configuration. For some instances, the baseline solves the problem at the root node while \locatellishort{} or \stairshort{} require additional branching, whereas for other instances the converse holds. However, two patterns emerge upon closer inspection. First, the cases where \baseshort{} solves at the root but a cut-augmented configuration does not occur more frequently for \locatellishort{} than for \stairshort{}, suggesting that the stronger \stairshort{} relaxation more reliably preserves the root-solvability of easy instances. Second, for instances requiring substantial enumeration (those with node counts in the thousands or higher), \stairshort{} and \indicatorshort{} tend to explore fewer nodes than \baseshort{}.

These observations are reflected in the shifted geometric means: \stairshort{} requires 312 nodes at shifted geometric mean, a reduction of approximately 17\% compared to the 375 nodes required by \baseshort{}. The \locatellishort{} configuration requires 376 nodes, essentially matching \baseshort{}, indicating that its cuts do not consistently provide sufficient bound improvement to reduce the overall enumeration effort. \indicatorshort{} requires 379 nodes, likewise equivalent to \baseshort{} in aggregate.

\subsubsection{Solving Time}

\definecolor{color0}{RGB}{205, 145, 50}
\definecolor{color1}{RGB}{25, 42, 86}
\definecolor{color2}{RGB}{156, 56, 72}
\definecolor{color3}{RGB}{0, 119, 139}

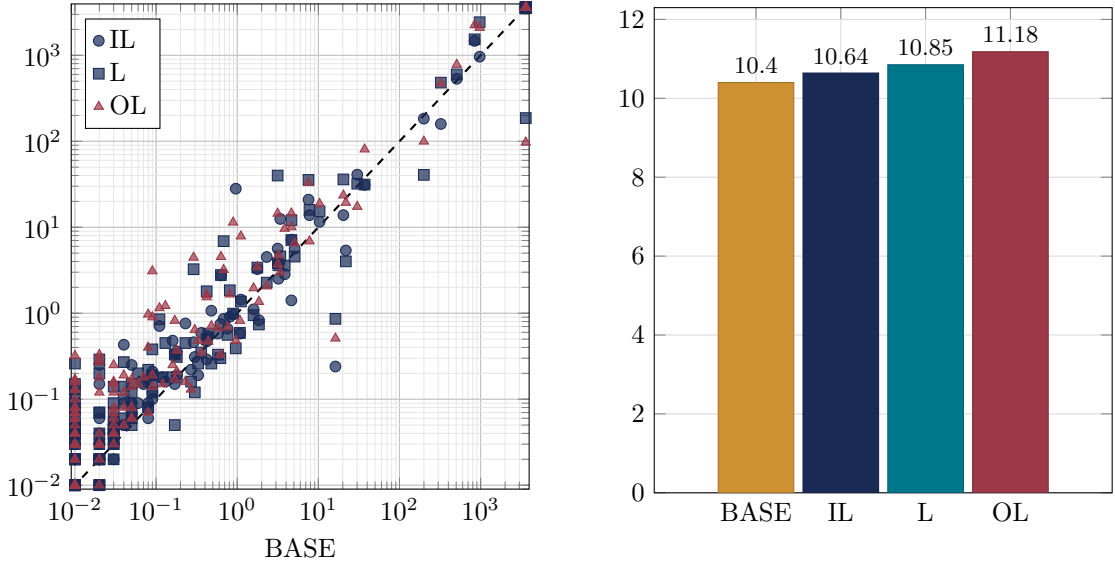
\begin{figure}[htbp]
\centering
\begin{minipage}[t]{0.49\textwidth}
    \centering
    \begin{tikzpicture}[baseline=(current bounding box.north)]
    \begin{axis}[
        xlabel={\baseshort},
        legend pos=north west,
        legend cell align=left,
        grid=both,
        grid style={line width=.1pt, draw=gray!20},
        major grid style={line width=.2pt, draw=gray!50},
        xmin=0.009000000000000001,
        xmax=3960.0000000000005,
        ymin=0.009000000000000001,
        ymax=3960.0000000000005,
        width=\linewidth,
        height=8cm,
        xmode=log,
        ymode=log
    ]

    % Diagonal reference line (y = x)
    \addplot[black, dashed, thick, domain=0.009000000000000001:3960.0000000000005, forget plot] {x};

    % Configuration: indicator_locatelli
    \addplot[only marks, mark=*, mark size=2pt, color=color1, fill=color1, fill opacity=0.7] coordinates {
        (0.03, 0.03)
        (3600.0, 3600)
        (3600, 3600.0)
        (3600.0, 3600.0)
        (3600.0, 3600)
        (0.02, 0.03)
        (0.0, 0.0)
        (0.0, 0.0)
        (0.04, 0.05)
        (0.05, 0.06)
        (0.03, 0.04)
        (0.02, 0.06)
        (0.03, 0.04)
        (0.03, 0.04)
        (0.01, 0.01)
        (0.01, 0.01)
        (0.01, 0.0)
        (0.01, 0.01)
        (0.0, 0.0)
        (0.02, 0.0)
        (0.02, 0.02)
        (0.01, 0.01)
        (0.02, 0.02)
        (0.0, 0.02)
        (0.0, 0.02)
        (0.0, 0.02)
        (0.18, 0.36)
        (0.01, 0.04)
        (0.03, 0.02)
        (0.01, 0.01)
        (0.16, 0.48)
        (0.02, 0.04)
        (973.2, 963.86)
        (0.68, 0.86)
        (0.03, 0.05)
        (1.77, 3.26)
        (0.01, 0.0)
        (0.01, 0.01)
        (0.0, 0.02)
        (0.0, 0.02)
        (0.0, 0.01)
        (0.05, 0.09)
        (0.62, 0.75)
        (3600.0, 3600.0)
        (0.42, 0.29)
        (0.13, 0.16)
        (0.09, 0.12)
        (0.89, 0.97)
        (0.17, 0.33)
        (0.02, 0.02)
        (0.01, 0.01)
        (0.01, 0.01)
        (0.01, 0.01)
        (0.01, 0.04)
        (0.05, 0.09)
        (0.08, 0.17)
        (7.56, 20.93)
        (3.15, 5.65)
        (843.18, 1471.04)
        (3.39, 12.49)
        (30.15, 40.91)
        (0.01, 0.02)
        (3.85, 2.84)
        (1.59, 1.11)
        (3600.0, 3600.0)
        (0.01, 0.11)
        (4.65, 1.41)
        (0.42, 0.47)
        (0.09, 0.21)
        (0.04, 0.05)
        (0.23, 0.76)
        (0.44, 0.54)
        (0.48, 1.07)
        (3.22, 2.52)
        (0.96, 28.09)
        (0.58, 0.58)
        (7.78, 13.82)
        (0.81, 0.91)
        (323.23, 159.0)
        (0.11, 0.71)
        (1.85, 0.83)
        (4.66, 7.18)
        (37.14, 31.09)
        (3600.0, 3600.0)
        (20.35, 13.83)
        (10.35, 11.57)
        (0.01, 0.01)
        (0.01, 0.01)
        (0.01, 0.02)
        (0.03, 0.05)
        (0.08, 0.06)
        (0.02, 0.01)
        (16.23, 0.24)
        (0.18, 0.19)
        (1.08, 0.58)
        (1.11, 1.44)
        (5.1, 5.64)
        (3600.0, 3600.0)
        (3600.0, 3600)
        (0.36, 0.59)
        (0.0, 0.0)
        (0.02, 0.15)
        (0.76, 0.66)
        (21.8, 5.37)
        (507.65, 533.23)
        (3600, 3600.0)
        (3600.0, 3600)
        (0.09, 0.1)
        (0.33, 0.19)
        (0.07, 0.15)
        (0.01, 0.05)
        (0.01, 0.07)
        (0.29, 0.46)
        (0.01, 0.01)
        (0.0, 0.01)
        (0.01, 0.01)
        (0.02, 0.03)
        (0.01, 0.02)
        (0.01, 0.01)
        (0.01, 0.01)
        (0.01, 0.01)
        (0.17, 0.15)
        (0.3, 0.31)
        (0.04, 0.43)
        (3600.0, 3600.0)
        (0.0, 0.0)
        (0.0, 0.0)
        (3600.0, 3600)
        (3600.0, 3600)
        (3600, 3600.0)
        (0.06, 0.09)
        (3600.0, 3600)
        (0.01, 0.01)
        (0.01, 0.01)
        (0.01, 0.01)
        (0.0, 0.0)
        (0.0, 0.0)
        (0.01, 0.01)
        (0.01, 0.0)
        (0.01, 0.0)
        (0.02, 0.02)
        (0.03, 0.04)
        (0.01, 0.03)
        (0.01, 0.02)
        (0.0, 0.01)
        (0.0, 0.0)
        (0.01, 0.0)
        (0.01, 0.01)
        (0.0, 0.0)
        (0.0, 0.0)
        (0.01, 0.01)
        (0.02, 0.02)
        (0.0, 0.0)
        (0.04, 0.09)
        (0.03, 0.06)
        (0.01, 0.04)
        (0.03, 0.03)
        (0.01, 0.04)
        (0.02, 0.25)
        (0.05, 0.25)
        (0.0, 0.04)
        (0.0, 0.0)
        (0.08, 0.09)
        (0.12, 0.18)
        (0.27, 0.22)
        (199.4, 183.41)
        (0.01, 0.01)
        (0.01, 0.01)
        (3.17, 3.78)
        (3600.0, 3600)
        (3600, 3600)
        (0.05, 0.07)
        (0.03, 0.05)
        (2.31, 4.51)
        (0.63, 2.75)
        (0.09, 0.19)
        (0.0, 0.0)
};
\addlegendentry{\indicatorshort}

% Configuration: locatelli
\addplot[only marks, mark=square*, mark size=2pt, color=color1, fill=color1, fill opacity=0.7] coordinates {
        (0.03, 0.03)
        (3600.0, 3600)
        (3600, 3600.0)
        (3600.0, 3600.0)
        (3600.0, 3600)
        (0.02, 0.03)
        (0.0, 0.0)
        (0.0, 0.0)
        (0.04, 0.06)
        (0.05, 0.06)
        (0.03, 0.04)
        (0.02, 0.04)
        (0.03, 0.05)
        (0.03, 0.04)
        (0.01, 0.01)
        (0.01, 0.01)
        (0.01, 0.0)
        (0.01, 0.01)
        (0.0, 0.01)
        (0.02, 0.02)
        (0.02, 0.02)
        (0.01, 0.26)
        (0.02, 0.07)
        (0.0, 0.03)
        (0.0, 0.04)
        (0.0, 0.03)
        (0.18, 0.31)
        (0.01, 0.01)
        (0.03, 0.02)
        (0.01, 0.08)
        (0.16, 0.18)
        (0.02, 0.01)
        (973.2, 2426.45)
        (0.68, 6.89)
        (0.03, 0.14)
        (1.77, 3.39)
        (0.01, 0.03)
        (0.01, 0.05)
        (0.0, 0.02)
        (0.0, 0.0)
        (0.0, 0.01)
        (0.05, 0.05)
        (0.62, 0.3)
        (3600.0, 3600.0)
        (0.42, 0.5)
        (0.13, 0.45)
        (0.09, 0.16)
        (0.89, 0.99)
        (0.17, 0.33)
        (0.02, 0.02)
        (0.01, 0.01)
        (0.01, 0.01)
        (0.01, 0.01)
        (0.01, 0.04)
        (0.05, 0.12)
        (0.08, 0.22)
        (7.56, 35.52)
        (3.15, 39.96)
        (843.18, 1533.05)
        (3.39, 4.58)
        (30.15, 32.1)
        (0.01, 0.09)
        (3.85, 3.59)
        (1.59, 0.95)
        (3600.0, 3600.0)
        (0.01, 0.15)
        (4.65, 7.1)
        (0.42, 1.8)
        (0.09, 0.19)
        (0.04, 0.14)
        (0.23, 0.45)
        (0.44, 0.58)
        (0.48, 0.26)
        (3.22, 3.61)
        (0.96, 0.39)
        (0.58, 0.33)
        (7.78, 15.96)
        (0.81, 1.85)
        (323.23, 481.52)
        (0.11, 0.85)
        (1.85, 0.74)
        (4.66, 12.08)
        (37.14, 31.37)
        (3600.0, 3600.0)
        (20.35, 36.02)
        (10.35, 15.31)
        (0.01, 0.01)
        (0.01, 0.01)
        (0.01, 0.02)
        (0.03, 0.03)
        (0.08, 0.08)
        (0.02, 0.01)
        (16.23, 0.86)
        (0.18, 0.18)
        (1.08, 0.59)
        (1.11, 1.37)
        (5.1, 4.56)
        (3600.0, 3600.0)
        (3600.0, 3600.0)
        (0.36, 0.35)
        (0.0, 0.0)
        (0.02, 0.2)
        (0.76, 0.56)
        (21.8, 4.01)
        (507.65, 599.53)
        (3600, 3600.0)
        (3600.0, 3600.0)
        (0.09, 0.12)
        (0.33, 0.26)
        (0.07, 0.16)
        (0.01, 0.06)
        (0.01, 0.13)
        (0.29, 3.25)
        (0.01, 0.01)
        (0.0, 0.01)
        (0.01, 0.02)
        (0.02, 0.29)
        (0.01, 0.02)
        (0.01, 0.01)
        (0.01, 0.05)
        (0.01, 0.01)
        (0.17, 0.05)
        (0.3, 0.12)
        (0.04, 0.27)
        (3600.0, 186.24)
        (0.0, 0.01)
        (0.0, 0.01)
        (3600.0, 3600.0)
        (3600.0, 3600.0)
        (3600, 3600.0)
        (0.06, 0.2)
        (3600.0, 3600.0)
        (0.01, 0.01)
        (0.01, 0.01)
        (0.01, 0.0)
        (0.0, 0.0)
        (0.0, 0.0)
        (0.01, 0.01)
        (0.01, 0.01)
        (0.01, 0.0)
        (0.02, 0.01)
        (0.03, 0.04)
        (0.01, 0.04)
        (0.01, 0.02)
        (0.0, 0.0)
        (0.0, 0.0)
        (0.01, 0.0)
        (0.01, 0.01)
        (0.0, 0.0)
        (0.0, 0.0)
        (0.01, 0.01)
        (0.02, 0.02)
        (0.0, 0.0)
        (0.04, 0.1)
        (0.03, 0.05)
        (0.01, 0.04)
        (0.03, 0.07)
        (0.01, 0.04)
        (0.02, 0.07)
        (0.05, 0.15)
        (0.0, 0.0)
        (0.0, 0.0)
        (0.08, 0.08)
        (0.12, 0.18)
        (0.27, 0.16)
        (199.4, 40.69)
        (0.01, 0.01)
        (0.01, 0.01)
        (3.17, 3.83)
        (3600.0, 3600)
        (3600, 3600)
        (0.05, 0.07)
        (0.03, 0.09)
        (2.31, 2.27)
        (0.63, 2.78)
        (0.09, 0.38)
        (0.0, 0.0)
};
\addlegendentry{\locatellishort}

% Configuration: stair_locatelli
\addplot[only marks, mark=triangle*, mark size=2pt, color=color2, fill=color2, fill opacity=0.7] coordinates {
        (0.03, 0.04)
        (3600.0, 3600.0)
        (3600, 3600)
        (3600.0, 3600.0)
        (3600.0, 3600.0)
        (0.02, 0.03)
        (0.0, 0.0)
        (0.0, 0.0)
        (0.04, 0.05)
        (0.05, 0.06)
        (0.03, 0.03)
        (0.02, 0.04)
        (0.03, 0.05)
        (0.03, 0.04)
        (0.01, 0.04)
        (0.01, 0.03)
        (0.01, 0.0)
        (0.01, 0.01)
        (0.0, 0.08)
        (0.02, 0.27)
        (0.02, 0.18)
        (0.01, 0.32)
        (0.02, 0.12)
        (0.0, 0.5)
        (0.0, 1.58)
        (0.0, 2.98)
        (0.18, 0.37)
        (0.01, 0.01)
        (0.03, 0.07)
        (0.01, 0.06)
        (0.16, 0.25)
        (0.02, 0.03)
        (973.2, 2096.14)
        (0.68, 3.2)
        (0.03, 0.25)
        (1.77, 3.46)
        (0.01, 0.1)
        (0.01, 0.13)
        (0.0, 0.12)
        (0.0, 0.2)
        (0.0, 0.01)
        (0.05, 0.17)
        (0.62, 0.33)
        (3600.0, 3600.0)
        (0.42, 1.64)
        (0.13, 1.22)
        (0.09, 3.09)
        (0.89, 11.4)
        (0.17, 0.82)
        (0.02, 0.02)
        (0.01, 0.01)
        (0.01, 0.01)
        (0.01, 0.01)
        (0.01, 0.06)
        (0.05, 0.14)
        (0.08, 0.4)
        (7.56, 32.78)
        (3.15, 14.46)
        (843.18, 2247.04)
        (3.39, 2.99)
        (30.15, 17.38)
        (0.01, 0.16)
        (3.85, 9.65)
        (1.59, 1.96)
        (3600.0, 3600.0)
        (0.01, 0.08)
        (4.65, 14.52)
        (0.42, 1.54)
        (0.09, 0.19)
        (0.04, 0.08)
        (0.23, 0.16)
        (0.44, 0.47)
        (0.48, 0.71)
        (3.22, 4.73)
        (0.96, 0.48)
        (0.58, 0.67)
        (7.78, 6.88)
        (0.81, 1.66)
        (323.23, 474.25)
        (0.11, 1.15)
        (1.85, 1.36)
        (4.66, 10.04)
        (37.14, 80.71)
        (3600.0, 3600)
        (20.35, 23.52)
        (10.35, 18.79)
        (0.01, 0.01)
        (0.01, 0.01)
        (0.01, 0.02)
        (0.03, 0.06)
        (0.08, 0.97)
        (0.02, 0.01)
        (16.23, 0.51)
        (0.18, 0.21)
        (1.08, 0.82)
        (1.11, 7.86)
        (5.1, 6.55)
        (3600.0, 3600.0)
        (3600.0, 3600.0)
        (0.36, 0.35)
        (0.0, 0.0)
        (0.02, 0.33)
        (0.76, 0.69)
        (21.8, 19.36)
        (507.65, 779.52)
        (3600, 3600.0)
        (3600.0, 3600.0)
        (0.09, 0.14)
        (0.33, 0.48)
        (0.07, 0.18)
        (0.01, 0.06)
        (0.01, 0.08)
        (0.29, 4.45)
        (0.01, 0.07)
        (0.0, 0.06)
        (0.01, 0.04)
        (0.02, 0.29)
        (0.01, 0.04)
        (0.01, 0.05)
        (0.01, 0.08)
        (0.01, 0.07)
        (0.17, 0.17)
        (0.3, 0.64)
        (0.04, 0.19)
        (3600.0, 97.29)
        (0.0, 0.01)
        (0.0, 0.0)
        (3600.0, 3600)
        (3600.0, 3600.0)
        (3600, 3600.0)
        (0.06, 0.15)
        (3600.0, 3600.0)
        (0.01, 0.01)
        (0.01, 0.01)
        (0.01, 0.01)
        (0.0, 0.0)
        (0.0, 0.01)
        (0.01, 0.01)
        (0.01, 0.01)
        (0.01, 0.01)
        (0.02, 0.03)
        (0.03, 0.08)
        (0.01, 0.07)
        (0.01, 0.02)
        (0.0, 0.01)
        (0.0, 0.0)
        (0.01, 0.03)
        (0.01, 0.03)
        (0.0, 0.02)
        (0.0, 0.02)
        (0.01, 0.03)
        (0.02, 0.03)
        (0.0, 0.02)
        (0.04, 0.12)
        (0.03, 0.16)
        (0.01, 0.17)
        (0.03, 0.12)
        (0.01, 0.13)
        (0.02, 0.03)
        (0.05, 0.08)
        (0.0, 0.03)
        (0.0, 0.0)
        (0.08, 0.07)
        (0.12, 0.15)
        (0.27, 0.13)
        (199.4, 99.74)
        (0.01, 0.03)
        (0.01, 0.0)
        (3.17, 3.72)
        (3600.0, 3600)
        (3600, 3600)
        (0.05, 0.06)
        (0.03, 0.15)
        (2.31, 2.06)
        (0.63, 4.53)
        (0.09, 0.9)
        (0.0, 0.0)
    };
    \addlegendentry{\stairshort}

    \end{axis}
    \end{tikzpicture}
        \subcaption{Scatter plot on a logarithmic scale comparing each configuration against \baseshort; points below the diagonal indicate faster solving.}
\label{plot:nr-solving-time}
\end{minipage}%
\hfill
\begin{minipage}[t]{0.49\textwidth}
    \centering
    \begin{tikzpicture}[baseline=(current bounding box.north)]
    \begin{axis}[
        ybar,
        bar width=1cm,
        enlarge x limits=0.4,
        symbolic x coords={\baseshort, \indicatorshort, \locatellishort, \stairshort},
        xtick={\baseshort, \indicatorshort, \locatellishort, \stairshort},
        grid=major,
        grid style={line width=.1pt, draw=gray!30},
        ymin=0,
        width=\linewidth,
        height=8cm,
        tick align=inside,
        nodes near coords,
        nodes near coords style={font=\small, anchor=south, /pgf/number format/fixed},
        yticklabel style={/pgf/number format/fixed}
    ]

    \addplot[fill=color0, draw=color0!80!black, bar shift=0pt] coordinates {(\baseshort, 10.40)};
    \addplot[fill=color1, draw=color1!80!black, bar shift=0pt] coordinates {(\indicatorshort, 10.64)};
    \addplot[fill=color3, draw=color3!80!black, bar shift=0pt] coordinates {(\locatellishort, 10.85)};
    \addplot[fill=color2, draw=color2!80!black, bar shift=0pt] coordinates {(\stairshort, 11.18)};

    \end{axis}
    \end{tikzpicture}
        \vspace{5.5mm} 
            \subcaption{Shifted geometric mean of the solving time across all instances.}
\label{plot:nr-solving-time-bar}
\end{minipage}
\caption{Total solving time without preprocessing. Compared 177 instances where at least one bilinear term has non-empty domain. Instances that exceed the timelimit of 1 hour are counted as 3600.}
\end{figure}

Figure~\ref{plot:nr-solving-time} presents the total wall-clock solving time for each configuration against the baseline on a logarithmic scale, and Figure~\ref{plot:nr-solving-time-bar} reports the shifted geometric means.

No configuration emerges as a clear winner in terms of total solving time. The shifted geometric means are closely grouped: \baseshort{} at 10.40 seconds, \indicatorshort{} at 10.64 seconds, \locatellishort{} at 10.85 seconds, and \stairshort{} at 11.18 seconds. Inspection of the scatter plot suggests that easier instances (those solved in under one second by \baseshort{}) tend to be solved slightly faster without the overhead of domain projection and cut generation, while harder instances (those approaching the time limit) benefit more frequently from the tighter relaxations provided by \locatellishort{}, \stairshort{}, and \indicatorshort{}.

This outcome is not unexpected. The current implementation computes all valid cutting planes during preprocessing and adds them to the model as static constraints before the solver begins its branch-and-bound procedure. No cut separation is performed during the tree search, and no interaction with Gurobi's spatial branching strategy takes place. Consequently, the solver must process a potentially large number of additional constraints at every node, even when only a small subset would be active. The fact that this static approach does not significantly increase total solving time, despite the overhead, demonstrates the potential of the method. Integrating the cut generation into a callback-based separation framework that adds cuts dynamically during the branch-and-bound process, and coordinating it with the solver's native spatial branching decisions, constitutes a natural direction for future work that can be expected to translate the observed relaxation improvements into consistent solving time reductions.

\section{Conclusion}
We developed a KKT-based algorithm for computing the convex envelope of
the bilinear function $xy$ over possibly nonconvex polygonal domains.
An analysis of the contact-point geometry identified the additional
generator configurations that arise beyond the convex case and required
their explicit treatment. We also extended the polyhedral subdivision
result from polytopes to polygons, showing that the envelope retains a
piecewise affine, quadratic, or rational representation. For orthogonal
polygons, reductions of the generators and their combinations simplify
the computation substantially. We presented ways to obtain such domains
in nonlinear optimization and demonstrated the value of retaining their
nonconvex geometry. Experiments on MINLPLib instances showed that the
resulting cuts can strengthen root relaxations and reduce branching
effort, although the current implementation does not yield a consistent
improvement in solving time.

Beyond the directions discussed in \Cref{sec:computational_study}, several opportunities for future work remain. The extension to fractional
terms and other bivariate functions satisfying \Cref{assumption} would
broaden the scope of the approach. Further theoretical investigation of
two-dimensional projections could identify additional ways to construct
useful polygonal relaxations from the constraints of a MINLP. On the
computational side, efficient and parallelizable preprocessing routines
are needed to obtain these domains at low cost and make their associated
cuts practical for larger problems. Understanding the trade-off between
the effort spent constructing polygonal domains and the strength of the
resulting relaxations will be important for their effective use in
deterministic global optimization.

\sloppy %hyphenation in bibliography
\printbibliography

\clearpage
\appendix
\section{KKT-based algorithm}
\label{app:kkt_algorithm}

\begin{algorithm}[H]
\caption{Three generators (\textsc{ThreeGenerators})}
\label{alg:kkt_three_generators}
\small
\algrenewcommand{\algorithmicrequire}{\textbf{Input:}}
\algrenewcommand{\algorithmicensure}{\textbf{Output:}}
\begin{algorithmic}[1]
\Require A simple polygon $\polygon$, its generating set $G(\polygon)$, and a generator set $J$ with $|J|=3$
\Ensure A collection $\mathcal{C}$ of KKT regions and associated envelope functions
\State $\mathcal{C}\gets\varnothing$
\Statex \Comment{see \Cref{sec:kkt_three_vertices,sec:kkt_two_vertices_one_edge,sec:kkt_one_vertex_two_edges,sec:kkt_three_edges}}
\State Compute candidate triples $(a_s^\ast,b_s^\ast,c_s^\ast)$ from the $\eta$-equalities,
\Statex using the expressions of $\eta_e$ on $\operatorname{int}(D_e)$ for every $e\in J\cap E'(\polygon)$
\State Retain candidates satisfying $(a_s^\ast,b_s^\ast)\in\bigcap_{e\in J\cap E'(\polygon)}\operatorname{int}(D_e)$ and
\Statex $\eta_g(a_s^\ast,b_s^\ast)\leq\eta_r(a_s^\ast,b_s^\ast)\quad\forall g\in J,\ r\in G(\polygon)\setminus J$
\For{each feasible candidate $s$}
    \State $\Gamma_J^{(s)}\gets\polygon\cap\operatorname{conv}\{\mathbf{x}_g(a_s^\ast,b_s^\ast):g\in J\}$
    \State $\mathcal{C}\gets\mathcal{C}\cup\{(\Gamma_J^{(s)},a_s^\ast x+b_s^\ast y+c_s^\ast)\}$, $c_s^\ast=\eta_g(a_s^\ast,b_s^\ast)$ for any $g\in J$
\EndFor
\State \Return $\mathcal{C}$
\end{algorithmic}
\end{algorithm}

\begin{algorithm}[H]
\caption{One vertex -- one edge (\textsc{VertexEdge})}
\label{alg:kkt_vertex_edge}
\small
\algrenewcommand{\algorithmicrequire}{\textbf{Input:}}
\algrenewcommand{\algorithmicensure}{\textbf{Output:}}
\begin{algorithmic}[1]
\Require A simple polygon $\polygon$, its generating set $G(\polygon)$, and a generator set $J$ with $J=\{v,e\}$
\Ensure A collection $\mathcal{C}$ of KKT regions and associated envelope functions
\State $\mathcal{C}\gets\varnothing$
\State Compute $a^\ast(x,y)$, $b^\ast(x,y)$ and $c^\ast(x,y)$ as in \Cref{sec:kkt_one_vertex_one_edge}
\Statex $F_J(x,y)\gets a^\ast(x,y)x+b^\ast(x,y)y+c^\ast(x,y)$
\State $C\coloneqq4m_e(q_e+m_e v_x-v_y)$, $\beta_2\coloneqq-1/C$,
\Statex $\beta_1\coloneqq(2q_e+4m_e v_x)/C$, $\beta_0\coloneqq-(q_e^2+4m_e v_xv_y)/C$
\State $\mathcal{I}\gets\varnothing$
\For{each choice $\widehat D_{e'}\in\{D_{e'}^-,D_{e'},D_{e'}^+\}$, $e'\in E'(\polygon)\setminus J$}
    \State Substitute $b(z)=\beta_2z^2+\beta_1z+\beta_0$, $a(z)=z-m_e b(z)$ into
    \Statex \hspace{\dimexpr\algorithmicindent*1\relax}$(a(z),b(z))\in\operatorname{int}(D_e)\cap\bigcap_{e'\in E'(\polygon)\setminus J}\widehat D_{e'}$ and
    \Statex \hspace{\dimexpr\algorithmicindent*1\relax}$\eta_v(a(z),b(z))\leq\eta_g(a(z),b(z))\quad\forall g\in G(\polygon)\setminus J$
    \State Solve the resulting system and add its candidate intervals $I_s$ to $\mathcal{I}$
\EndFor
\State $d\coloneqq v_y-q_e-m_e v_x$, $N(x,y)\coloneqq2m_e\bigl(x(v_y-q_e)-v_x(y-q_e)\bigr)+q_e h(x,y)$,
\Statex where $h(x,y)\coloneqq m_e(x-v_x)+v_y-y$
\For{each candidate interval $I_s\in\mathcal{I}$ with endpoints $L_s,U_s$}
    \If{$d>0$}
        \State $\Gamma_J^{(s)}\gets\left\{(x,y)\in\polygon:\begin{aligned}
            &0<h(x,y)\leq d,\\
            &L_s h(x,y)\leq N(x,y)\leq U_s h(x,y)
        \end{aligned}\right\}$
    \Else
        \State $\Gamma_J^{(s)}\gets\left\{(x,y)\in\polygon:\begin{aligned}
            &d\leq h(x,y)<0,\\
            &U_s h(x,y)\leq N(x,y)\leq L_s h(x,y)
        \end{aligned}\right\}$
    \EndIf
    \State Make interval inequalities strict at excluded endpoints of $I_s$
    \State $\mathcal{C}\gets\mathcal{C}\cup\{(\Gamma_J^{(s)},F_J(x,y))\}$
\EndFor
\State \Return $\mathcal{C}$
\end{algorithmic}
\end{algorithm}

\begin{algorithm}[H]
\caption{Two edges (\textsc{TwoEdges})}
\label{alg:kkt_two_edges}
\small
\algrenewcommand{\algorithmicrequire}{\textbf{Input:}}
\algrenewcommand{\algorithmicensure}{\textbf{Output:}}
\begin{algorithmic}[1]
\Require A simple polygon $\polygon$, its generating set $G(\polygon)$, and a generator set $J$ with $J=\{e_1,e_2\}$
\Ensure A collection $\mathcal{C}$ of KKT regions and associated envelope functions
\State $\mathcal{C}\gets\varnothing$
\State Compute $\beta_{1,\tau}$ and $\beta_{0,\tau}$ as in \Cref{sec:kkt_two_edges} for all relations $\tau$
\State $\mathcal{I}\gets\varnothing$
\For{each relation $\tau$}
    \State Compute $a_\tau^\ast(x,y)$, $b_\tau^\ast(x,y)$ and $c_\tau^\ast(x,y)$ as in \Cref{sec:kkt_two_edges}
    \Statex \hspace{\dimexpr\algorithmicindent*1\relax}$F_J^{(\tau)}(x,y)\gets a_\tau^\ast(x,y)x+b_\tau^\ast(x,y)y+c_\tau^\ast(x,y)$
    \For{each choice $\widehat D_{e'}\in\{D_{e'}^-,D_{e'},D_{e'}^+\}$, $e'\in E'(\polygon)\setminus J$}
        \State Substitute $a_\tau(b)=\beta_{1,\tau}b+\beta_{0,\tau}$ into
        \Statex \hspace{\dimexpr\algorithmicindent*2\relax}$(a_\tau(b),b)\in\operatorname{int}(D_{e_1})\cap\operatorname{int}(D_{e_2})\cap\bigcap_{e'\in E'(\polygon)\setminus J}\widehat D_{e'}$ and
        \Statex \hspace{\dimexpr\algorithmicindent*2\relax}$\eta_{e_1}(a_\tau(b),b)\leq\eta_g(a_\tau(b),b)\quad\forall g\in G(\polygon)\setminus J$
        \State Solve the resulting system and add its candidate intervals $I_{\tau,s}$ to $\mathcal{I}$
    \EndFor
\EndFor
\For{each candidate interval $I_{\tau,s}\in\mathcal{I}$ with endpoints $L_{\tau,s},U_{\tau,s}$}
    \State $x_{e_i,\tau}(x,y)\coloneqq\dfrac{y+\beta_{1,\tau}x-q_{e_i}}{m_{e_i}+\beta_{1,\tau}}$, $i=1,2$, $\Delta_\tau(x,y)\coloneqq x_{e_2,\tau}(x,y)-x_{e_1,\tau}(x,y)$
    \State $\begin{aligned}[t]
        \Gamma_J^{(\tau,s)}\gets{}&\left\{(x,y)\in\polygon:\begin{aligned}
            &\Delta_\tau(x,y)>0,\quad 0\leq x-x_{e_1,\tau}(x,y)\leq\Delta_\tau(x,y),\\
            &L_{\tau,s}\leq b_\tau^\ast(x,y)\leq U_{\tau,s}
        \end{aligned}\right\}\\
        &\cup\left\{(x,y)\in\polygon:\begin{aligned}
            &\Delta_\tau(x,y)<0,\quad\Delta_\tau(x,y)\leq x-x_{e_1,\tau}(x,y)\leq0,\\
            &L_{\tau,s}\leq b_\tau^\ast(x,y)\leq U_{\tau,s}
        \end{aligned}\right\}
    \end{aligned}$
    \State Make interval inequalities strict at excluded endpoints of $I_{\tau,s}$
    \State $\mathcal{C}\gets\mathcal{C}\cup\{(\Gamma_J^{(\tau,s)},F_J^{(\tau)}(x,y))\}$
\EndFor
\State \Return $\mathcal{C}$
\end{algorithmic}
\end{algorithm}

\end{document}